\documentclass[11pt]{amsart}

\usepackage{amsmath, amssymb, amsthm}
\usepackage[colorlinks=true]{hyperref}
\usepackage[alphabetic]{amsrefs}
\usepackage{verbatim}
\usepackage{xcolor}
\usepackage{url}
\usepackage{enumitem}
\usepackage{microtype}
\usepackage[pdftex]{graphicx}
\usepackage[margin=.85 in]{geometry}
\usepackage{lmodern}
\usepackage{mathrsfs}
\usepackage{tikz}
\usepackage{caption}
\usepackage{subcaption}

\newtheorem{lemma}{Lemma}[section]
\newtheorem{theorem}[lemma]{Theorem}
\newtheorem{corollary}[lemma]{Corollary}
\newtheorem{proposition}[lemma]{Proposition}
\theoremstyle{definition}
\newtheorem{definition}[lemma]{Definition}

\newtheorem{question}[lemma]{Question}

\newtheorem{example*}[lemma]{Example}
\newtheorem{remark}[lemma]{Remark}

\theoremstyle{remark}

\makeatletter
\newtheorem*{rep@theorem}{\rep@title}
\newcommand{\newreptheorem}[2]{%
	\newenvironment{rep#1}[1]{%
		\def\rep@title{{\bf #2 \ref{##1}}}%
		\begin{rep@theorem}}%
		{\end{rep@theorem}}}
\makeatother
\newreptheorem{theorem}{Theorem}

\DeclareRobustCommand{\qedify}[1]{%
	\ifmmode \quad\hbox{#1}
	\else
	\leavevmode\unskip\penalty9999 \hbox{}\nobreak\hfill
	\quad\hbox{#1}%
	\fi
}
\newenvironment{example}{\begin{example*}\pushQED{\qedify{$\diamondsuit$}}}{\popQED\end{example*}}

\numberwithin{equation}{section}

\makeatletter
\newcommand{\superimpose}[2]{{\ooalign{$#1\@firstoftwo#2$\cr\hfil$#1\@secondoftwo#2$\hfil\cr}}}
\makeatother
\newcommand{\ttimes}{\hspace{0.3mm}{\mathpalette\superimpose{{\circ}{\cdot}}}\hspace{0.3mm}}
\newcommand{\tplus}{\mathrel{\oplus}}
\newcommand{\meet}{\wedge}
\newcommand{\join}{\vee}

\newcommand{\note}[1]{{\color{blue}#1}}

\DeclareMathOperator{\trop}{trop}
\DeclareMathOperator{\val}{val}
\DeclareMathOperator{\Cox}{Cox}
\DeclareMathOperator{\spann}{span}
\DeclareMathOperator{\rk}{rk}

\DeclareMathOperator{\Hom}{Hom}
\DeclareMathOperator{\Spec}{Spec}

\DeclareMathOperator{\supp}{supp}
\DeclareMathOperator{\divv}{div}
\DeclareMathOperator{\cone}{cone}
\DeclareMathOperator{\im}{im}
\DeclareMathOperator{\GL}{GL}

\newcommand{\Rbar}{\overline{\mathbb R}}

\begin{document}
	
	\title{Tropical vector bundles}

\author{Bivas Khan}
	\address{Department of Mathematics, Chennai Mathematical Institute, Chennai, India}
	\email{bivaskhan10@gmail.com}
        
        \author{Diane Maclagan}
        \address{Mathematics Institute, University of Warwick, Coventry, CV4 7AL, United Kingdom}
        \email{D.Maclagan@warwick.ac.uk}
        
        \begin{abstract}
The goal of this paper is to introduce a construction of a vector
bundle on a tropical variety.  When the base is a tropical toric
variety these tropicalize toric vector bundles, and are described by
the data of a valuated matroid and some flats in the lattice of flats
of the underlying matroid.  The fibers are tropical linear spaces.  We
define global sections for tropical toric vector bundles, stability,
and Jordan-H\"older and Harder-Narasimhan filtrations.  Many of these
require additional modularity assumptions on the defining matroid.
\end{abstract}

	\maketitle

\section{Introduction}

The goal of this paper is to  introduce a notion of vector bundles on tropical varieties.

The theory of divisors, or equivalently line bundles, on tropical
curves/metric graphs has been intensively studied since the beginning
of tropical geometry.  This has been particularly successful in Brill-Noether theory \cites{CDPR,JensenRanganathan} and other aspects of the geometry of the moduli space $M_g$ \cite{FarkasJensenPayne}.

The theory of higher rank vector bundles, by contrast, is much less
well developed.  After early work by Allermann \cite{Allermann}, the
story was restarted recently in the context of metric graphs by Gross,
Ulirsch, and Zakharov \cite{GUZ}.  In \cite{JMT20} Jun, Mincheva, and
Tolliver defined and studied vector bundles on tropical toric
varieties (in a slightly more general framework).  However their
construction has the drawback that all vector bundles are direct sums
of line bundles.  Similarly, all vector bundles on a simply connected
tropical variety are trivial in Allermann's theory.

In this paper we give a definition of a vector bundle on a tropical
toric variety that does not have these limitations, and define the
tropicalization of an equivariant vector bundle on a toric variety.
This definition then extends to vector bundles on subschemes of
tropical toric varieties, in the sense of \cite{TropicalIdeals}, and
to the tropicalization of some vector bundles on  subschemes of toric
varieties.

Our approach is to first tropicalize presentations of toric vector bundles
on a toric variety $X_{\Sigma}$.  By \cites{CoxHomogeneous, PerT, Perling}  an
equivariant sheaf on a toric variety can be described by an $\mathbb Z^{|\Sigma(1)|}$-graded module over the Cox ring of $X_{\Sigma}$, which is a polynomial ring with variables indexed by the rays of $\Sigma$. 
An equivariant surjection
\begin{equation} \label{eqtn:introsurjection}
  \oplus_{\mathbf{w} \in \mathcal G} \mathcal O(D_{\mathbf{w}})
  \rightarrow \mathcal E
\end{equation}
from a direct sum of line bundles onto a toric vector bundle $\mathcal E$ gives rise to a
presentation of this module (see Section~\ref{s:CoxModule}).  This
presentation can then be tropicalized in the sense of
\cites{TropicalIdeals, Balancing}.  The surjection of line bundles
induces the structure of a matroid on the index set $\mathcal G$, in
the sense of DiRocco, Jabbusch, and Smith \cite{DJS}.  It also gives a
linear embedding of the fibers into $\mathbb A^{|\mathcal G|}$, which
tropicalizes to a tropical linear space in $\Rbar^{|\mathcal G|}$.
See \S\ref{s:matroid} and \S\ref{ss:totalspace} respectively for
more details.

Our definition of a tropical toric vector bundle on a tropical toric variety takes these
properties as part of the definition, building on the Klyachko
classification of toric vector bundles \cite{Kly}.  We write
$\trop(X_{\Sigma})$ for the tropical toric variety \cites{Kajiwara,
  Payne} with fan $\Sigma$.

\begin{definition} \label{d:introdefn}
Let $\trop(X_{\Sigma})$ be a tropical toric variety.  A {\em tropical toric
reflexive sheaf} $\mathcal E = (\mathcal M, \mathcal G, \{E^i(j)\})$ on
$\trop(X_{\Sigma})$ is given by a simple valuated matroid $\mathcal M$
of rank $r$ on a ground set $\mathcal G$, and for each ray $\rho_i$ of
$\Sigma$ a collection $\{E^i(j) : j \in \mathbb Z \}$ of flats of the
underlying matroid $\underline{\mathcal M}$ with the property that $E^i(j) \leq
E^i(j')$ if $j>j'$, $E^i(j) = \emptyset$ for $j \gg 0$, and $E^i(j) =
\mathcal G$ for $j \ll 0$.

The reflexive sheaf $\mathcal E$ is a {\em tropical toric vector
  bundle} if in addition for every maximal cone $\sigma$ of $\Sigma$
there is a multiset $\mathbf{u}(\sigma) \subseteq M$, and a basis
$B_{\sigma} = \{ \mathbf{w}_{\mathbf{u}} : \mathbf{u} \in
\mathbf{u}(\sigma)\}$ of $\underline{\mathcal M}$ such that for every
ray $\rho_i \in \sigma$ and $j \in \mathbb Z$ we have
\begin{equation} \label{eqtn:tropicalcompatibilityintro}
	E^i(j) = \bigvee\limits_{\mathbf{u} \cdot \mathbf{v}_i \geq j} \mathbf{w}_{\mathbf{u}},
\end{equation}
is the smallest flat containing all the $\mathbf{w}_{\mathbf{u}}$ with $\mathbf{u} \cdot \mathbf{v}_i \geq j$.

\end{definition}

Klyachko \cite{Kly} showed that the category of toric vector
bundles with equivariant morphisms is equivalent to a category of
compatible filtrations of a vector space.  Our definition is an abstraction of  Klyachko's classification.
    The poset of all
intersections of filtered pieces, ordered by inclusions, embeds into
the lattice of flats of a matroid in the sense of \cite{DJS}.  See \S\ref{s:matroid} for more details.  In our
setting the \cite{DJS} matroid  is the underlying matroid of $\mathcal M$.
The compatibility
relation \eqref{eqtn:tropicalcompatibilityintro} is the abstraction of
the usual Klyachko compatibility relation.

A surjection of the form \eqref{eqtn:introsurjection} induces a
realizable valuated matroid $\mathcal M$, and thus a tropical toric
vector bundle (see Definition~\ref{pd:tropE}).  Tropical toric vector bundles that arise in this fashion are called {\em realizable}.  This tropicalization
depends on the choice of surjection.  This is 
analogous to the way in which the tropicalization of a projective
variety depends on its projective embedding, including the choice of
coordinates on $\mathbb P^n$.  However some tropical invariants are
constant for any, or for many, choices of surjection.

While the tropicalization of toric vector bundles is the main
motivation for our construction, not all tropical toric vector bundles
are realizable.  A tropical toric vector bundle can fail to be
realizable in two different ways: if the valuated matroid $\mathcal M$
is not realizable, or if the choice of flats $E^i(j)$ cannot come from
a filtration of subspaces.  See Examples~\ref{e:Fanoeg},
\ref{e:vamostotalspace}, and \ref{vbonP1split}\eqref{item:P1notdecomposable} for more on this phenomenon.

A coherent sheaf on a toric variety can be described by a module over
the Cox ring of the toric variety.  We describe this explicitly for toric vector bundles.

\begin{proposition}
  Let $S$ be the Cox ring of a smooth toric variety $X_{\Sigma}$, and let
  $\mathcal E$ be a rank-$r$ toric vector bundle on $X_{\Sigma}$ with
  an equivariant surjection of the form \eqref{eqtn:introsurjection}.
  Let $G$ be an $r \times m$ matrix whose columns are a realization of the associated valuated matroid $\mathcal M$ with ground set $\mathcal G$.
  Then $\mathcal E$ is the sheafification
  $\widetilde{P}$ of the $\mathbb Z^{|\Sigma(1)|}$-graded module $P$
  given by
  $$P = \bigoplus_{\mathbf{w} \in \mathcal G}
  S(\mathbf{d}_{\mathbf{w}})/ R,$$ where the $\mathbf{d}_{\mathbf{w}} \in
  \mathbb Z^{|\Sigma(1)|}$ encode the subspaces $E^i(j)$, and $R$ is a
  submodule of $\oplus_{\mathbf{w} \in \mathcal G}
  S(\mathbf{d}_{\mathbf{w}})$ obtained by homogenizing $\ker(G)$.
  \end{proposition}

A precise statement is given in Proposition~\ref{p:CoxModuleTVB}.
This motivates the following definition in the tropical setting.

\begin{definition}
Let $\mathcal E = (\mathcal M, \mathcal G, \{E^i(j) \})$ be a tropical
toric vector bundle on a tropical toric variety $\trop(X_{\Sigma})$,
and let $S =\Rbar[x_1,\dots,x_s]$ be the Cox semiring of
$\trop(X_{\Sigma})$.  Then $\mathcal E$ can be described by the
semimodule
\begin{equation}   P = \bigoplus_{\mathbf{w} \in \mathcal G}
                     S(\mathbf{d}_{\mathbf{w}}) / \mathcal B(R),
             \end{equation}
             where $\mathcal B()$ is the Giansiracusa bend congruence, and $R$ is a homogenization of the equations for the tropical linear space of $\mathcal M$.  
\end{definition}

A consequence of this is a description of the total space of a
tropical toric vector bundle as a tropical subvariety of
$\trop(X_{\Sigma}) \times \Rbar^m$.  The fiber over a point in
$\trop(X_{\Sigma})$ is a tropical linear space in $\Rbar^m$.  When the
point is the tropicalization of the identity of the torus of
$X_{\Sigma}$ the corresponding tropical linear space is the one
defined by the valuated matroid $\mathcal M$.  See \S~\ref{ss:totalspace} for more details.

Tropical toric vector bundles have the following properties that are familiar from  toric vector bundles.

\begin{theorem} \label{t:introtheorem}
  \begin{enumerate}
  \item \label{item:thmh0}    There is a definition of $h^0(\mathcal E)$ for a tropical
    toric vector bundle.  When $\mathcal E = \trop(\mathcal F)$ and
    the matroid for $\mathcal F$ is chosen carefully, we have
    $h^0(\mathcal E) = h^0(\mathcal F)$ (Definition~\ref{d:parliament} and Lemma~\ref{l:tvbh0}).
\item \label{item:thmsplits}  With extra modularity and minimality assumptions on $\mathcal
  M$, tropical toric vector bundles on $\trop(\mathbb P^1)$ split as a direct sum of line bundles (Proposition~\ref{p:P1splits}).
  \item \label{item:thmstability} There is a notion of stability for
    tropical toric vector bundles (see Section \ref{s:stability}).
    When $\mathcal E = \trop(\mathcal F)$ stability for $\mathcal F$
    implies stability for $\mathcal E$.  If the matroid for $\mathcal
    F$ is chosen carefully, $\mathcal E$ is stable if and only if
    $\mathcal F$ is.
\item \label{item:thmJHHN} With extra modularity assumptions on $\mathcal M$,
  Jordan-H\"older and Harder-Narasimhan filtrations exist
  (Theorems~\ref{t:JH} and \ref{t:HNfilt}).
\end{enumerate}    

\end{theorem}

Part~\ref{item:thmh0} of Theorem~\ref{t:introtheorem} is inspired by
and generalizes \cite{DJS}.  Part~\ref{item:thmstability} generalizes
\cite{Lucie} and \cites{DDK-Stab,ErSEVB}.  The classical proof that every
vector bundle on $\mathbb P^1$ splits as a direct sum of line bundles
already relies in a crucial way in some foundational results from
lattice theory, and Part~\ref{item:thmsplits} of the theorem can be regarded as
abstracting back to those results.

Modularity is a property of lattices that abstracts the linear algebra
fact that $\dim(V + W) + \dim(V \cap W) = \dim(V) + \dim(W)$ for two
subspaces $V,W$ of a finite-dimensional vector space.  Lattices of
flats of matroids are always {\em semimodular} (the analogous relation
holds with $\leq$ instead of $=$), but are much less likely to be
modular, even when realizable.  While we do not require complete
modularity in Theorem~\ref{t:introtheorem}, some form is necessary, as
Part~\ref{item:thmsplits} is false without the modularity assumption
(see Example~\ref{vbonP1split}).  More work is required to
completely understand this condition for realizable tropical toric
vector bundles.

The work of Jun, Mincheva, and Tolliver \cite{JMT20} also tropicalizes
toric vector bundles.  However in their setting all tropical vector
bundles on $\trop(X_{\Sigma})$ are a direct sum of line bundles, which
is far from the case in non-tropical algebraic geometry.  The issue
here is one would expect transition functions for a rank-$r$ tropical vector
bundle to live in a tropicalization of $\GL(r)$ over a valued field $K$.  However the naive
form of tropical $\GL(r)$  consists of tropical $r \times r$
matrices with only one entry not equal to $\infty$ in each row and
column, so is isomorphic to $S_r \ltimes \mathbb R^r$.  This is the
tropicalization of those elements of $\GL(r)$ that preserve the torus $(K^*)^r$ when acting on $\mathbb A^r$.  A toric
vector bundle is determined by a choice of basis for the fiber $E$ of
$\mathcal E$ over the identity of the torus for each cone $\sigma \in
\Sigma$, and such a matrix only permutes the elements of this basis.
Thus we would have the same basis for every cone of $\Sigma$, and so
the bundle $\mathcal E$ would decompose as a direct sum of line
bundles.

By contrast, in the present work we use the tropicalization of
$\GL(m)$ for transition maps, where $m=|\mathcal G|$ is the size of the ground set of the
valuated matroid $\mathcal M$.  This allows the
construction of indecomposable tropical toric vector bundles.  One way
to think of the matroid $\mathcal M$ in the realizable case is that we
only require finitely many matrices in $\GL(r)$ to construct the
transition functions of a vector bundle, and a realization of $\mathcal M$
can be thought of as a map from $K^m$ to $K^r$ that has all of the
columns of these matrices among the images of the standard basis
vectors.

We also extend the definition of tropical vector bundles to more
general bases.  Embedded tropicalization of a variety $X$ in the sense
of \cite{TropicalBook} proceeds by first embedding $X$ into a toric
variety. In an analogous fashion we {\em define} a tropical vector
bundle on a tropical subscheme of a tropical toric variety to be the
restriction of a tropical toric variety on the ambient toric variety.
Vector bundles on $X$ that can be tropicalized in this fashion are
those that are the pullback of a toric vector bundle $\mathcal E$ on a
toric variety $X_{\Sigma}$ under {\em some} embedding of $X$ into
$X_{\Sigma}$.  Note that every line bundle on $X$ has this form; for
example, when $\mathcal L$ is a very ample line bundle corresponding
to an embedding $\phi \colon X \rightarrow \mathbb P^m$, we have
$\mathcal L = \phi^*(\mathcal O_{\mathbb P^m}(1))$.  As with the
tropicalization of $X$, the freedom in the choice of embedding is
crucial to get interesting answers.  It would be interesting to
understand which vector bundles can be constructed in this fashion.
This is covered in Section~\ref{s:nontoric}.

The tropicalization of vector bundles on curves was considered already
by Gross, Ulirsch, and Zakharov \cite{GUZ}, in the context of
tropicalization taking curves to metric graphs.  Their vector bundles
are constructed as pushforwards of tropical line bundles on a cover of
the metric graph, and the vector bundles on a classical curve that can
be tropicalized are those that can be realized as a pushforward in
this sense.  It will also be important to understand the interaction
between this approach and the one of this paper, and in particular
which vector bundles on curves can be tropicalized in each sense.

As this paper was nearing completion we became aware of the similar
project of Kiumars Kaveh and Chris Manon on {\em Toric matroid
  bundles} \cite{KavehManon24}.  Their toric matroid bundles are the same as our tropical
toric vector bundles in the case that the valuated matroid structure
is trivial.  Both works have the same definition of first Chern class
and global sections, and highlight the importance of modularity. 

The structure of this paper is as follows.  In
Section~\ref{s:background} we recall the necessary background on
tropical geometry and lattice theory needed in the paper.  In
Section~\ref{s:matroid} we explain the connection between toric vector
bundles and matroids.  Section~\ref{s:CoxModule} describes the
Cox module of a toric vector bundle.  The tropical material begins in
Section~\ref{s:tropicaltoricdefns}, with the definition of a tropical
toric reflexive sheaf and a tropical toric vector bundle, and a
description of the total space of a tropical toric vector bundle.
Further properties of tropical toric vector bundles are explained in
Section~\ref{s:properties}: decomposition into direct sums, tensor
product with line bundles, and parliaments of polytopes.  Stability of
tropical toric vector bundles, and the Jordan-H\"older and
Harder-Narasimhan filtrations, are covered in
Section~\ref{s:stability}.  Finally, Section~\ref{s:nontoric} extends
these definition to tropical vector bundles on more general tropical
schemes.

\noindent {\bf Acknowledgements: } We thank Michel Brion for pointing
out Remark~\ref{r:allpullbacks}\eqref{item:brion}, and Lucie Devey
for conversations about \cite{Lucie}.  Thanks are also due to Chris
Manon for explaining his earlier work with Kiumars Kaveh, and to both
Kaveh and Manon for explaining their forthcoming work.  Maclagan also
thanks Milena Hering and Greg Smith for many useful conversations over
the years about toric vector bundles, and Georg Loho for useful
conversations about polymatroids.  Both authors were partially
supported by EPSRC grant EP/R02300X/1.  Khan was also supported by
postdoctoral fellowships from the National Board for Higher
Mathematics and a grant from the Infosys Foundation, and Maclagan was
also supported by ESPRC Fellowship EP/X02752X/1.

\noindent {\bf Notation: } We follow standard notation for toric
	geometry as in \cite{Fulton} or \cite{CLS}, so $N \cong \mathbb Z^n$
	always denotes a rank-$n$ lattice, and $M = \Hom(N,\mathbb Z)$ is the
	dual lattice.  We write $s = |\Sigma(1)|$ for the number of rays of the fan $\Sigma$.  
	
	We use $\mathcal M$ to denote a valuated matroid, and
        $\underline{\mathcal M}$ to denote the corresponding
        underlying matroid of $\mathcal M$.  The ground set $\mathcal
        G$ of $\mathcal M$ has size $m$.  Note the typographical
        distinction between the lattice $M \cong \mathbb Z^n$ and the valuated matroid
        $\mathcal M$.

        The symbol $\sqcup$ denotes disjoint union of sets.

	\section{Background on tropical and lattice techniques} \label{s:background}
	We recall here three key concepts from tropical geometry: tropical toric
	varieties, valuated matroids, and the Giansiracusa bend congruence.  We also recall the basics of lattice theory that we will need.

We follow standard notation for toric varieties from
\cite{Fulton} or \cite{CLS}.  
We write $n$ for the dimension
        of the toric variety $X_{\Sigma}$, $s=|\Sigma(1)|$ for the number of rays,  $\mathbf{v}_1,\dots,\mathbf{v}_s$ for the first lattice points 
        on the rays $\rho_1,\dots,\rho_s$ of the fan $\Sigma$, and $D_i$ for the torus-invariant divisor corresponding to the $i$th ray of $\Sigma$.
	The lattice of $X_{\Sigma}$ is  $N \cong \mathbb Z^n$, and
        $M = \Hom(N,\mathbb Z)$ is the dual lattice, where the homomorphism is as abelian
	groups.  We denote by  $N_{\mathbb R}$ the vector space $N \otimes \mathbb R$.
 	We assume unless otherwise noted that the fan $\Sigma$ is complete.
	
	We denote the tropical semiring by $\Rbar = (\mathbb R \cup
        \{\infty\}, \tplus, \ttimes)$, where $\tplus \, = \min$, and
        $\ttimes = +$.
	Let $K$ be a field equipped with a valuation $\text{val} :K
        \rightarrow \Rbar$.  For a vector $\mathbf{v} \in \Rbar^n$ we
        denote by $\supp(\mathbf{v})$ the set $\{ i : v_i \neq \infty
        \}$.

	\subsection{Tropical toric varieties} \label{preliTTVar}

        Recall that for a pointed rational cone $\sigma \subseteq N_{\mathbb R}$ the
	affine toric variety $U_{\sigma}$ is $\Spec(K[\sigma^{\vee} \cap M])$,
	with $K$-rational points $\Hom(K[\sigma^{\vee} \cap M],K) =
	\Hom(\sigma^{\vee} \cap M ,K)$, where the first homomorphism is as
	$K$-algebras, and the second is as semigroup homomorphisms from
	$(\sigma^{\vee} \cap M ,+)$ to $(K, \times)$.  The {\em tropical affine toric
		variety} $\trop(U_{\sigma})$ is
	$$\Hom(\sigma^{\vee} \cap M,\Rbar),$$ where the homomorphism is as
	semigroups from $(\sigma^{\vee} \cap M,+)$ to $(\Rbar, \ttimes)$.  The topology on $\trop(U_{\sigma})$ comes from viewing it as a subspace of $\Rbar^{\sigma^{\vee} \cap M}$.
	Given a rational polyhedral fan $\Sigma$, the
	{\em tropical toric variety} $\trop(X_{\Sigma})$ is made by
	gluing together the sets $\trop(U_{\sigma})$ along faces as in the standard case: if $\tau$ is a face of $\sigma$, then $\sigma^{\vee}$ is a sub-semigroup of $\tau^{\vee}$, so
	$\Hom(\tau^{\vee} \cap M, \Rbar)$ is a subset of $\Hom(\sigma^{\vee} \cap M,\Rbar)$.  The coordinate ring of $\trop(U_{\sigma})$ is the semiring $\Rbar[\sigma^{\vee} \cap M]$.
	
  The Cox semiring of a tropical toric variety is
  $$\Cox(\trop(X_{\Sigma})) = \Rbar[x_1,\dots,x_s].$$
  This is graded by the class group of
	$X_{\Sigma}$:
	\begin{equation} \label{eqtn:A1} A^1(\Sigma) = \text{coker}(M \stackrel{\phi}\rightarrow \mathbb Z^{s}). \end{equation}
	Here $(\phi(\mathbf{u}))_i = \mathbf{u} \cdot \mathbf{v}_i$,
        and the grading is $\deg(x_i) = \mathbf{e}_i + \im \phi$.

For $\sigma \in \Sigma$ write $\mathbf{x}^{\hat{\sigma}} = \prod_{\rho_i \not \in \sigma} x_i$.          The irrelevant ideal $B_{\Sigma}$ is $\langle
\mathbf{x}^{\hat{\sigma}} : \sigma \in \Sigma \rangle \subseteq
        \Rbar[x_1,\dots,x_s]$.
  The tropical Cox construction of
        $\trop(X_{\Sigma})$ is
        $$(\Rbar^s \setminus V(B_{\Sigma}))/H,$$ where $H = \im(\phi
        \otimes \mathbb R) \subseteq \mathbb R^s$ is the tropical dual of the class group.  For more details
        on tropical toric varieties, see \cite{Kajiwara},
        \cite{PayneTropical}, or \cite{TropicalBook}*{\S 6.2}.

	\subsection{Valuated matroids} \label{ss:valuatedmatroids}
	
	Valuated matroids are a generalization of matroids introduced by Dress
	and Wenzel \cite{DressWenzel} that have proved particularly useful in
	tropical geometry.  We assume familiarity with the basics of matroid
	theory as in \cite{Oxley}.
	
	A valuated matroid $\mathcal M$ of rank $r$ on a finite set $\mathcal
	G$ is given by a basis valuation function $$\nu \colon \binom{\mathcal G}{r} \rightarrow
	\Rbar$$ satisfying that $\nu(B) \neq \infty$ for some $B \in
	\binom{\mathcal G}{r}$ and for all $B_1, B_2 \in \binom{\mathcal
		G}{r}$ and $i \in B_1 \setminus B_2$ there is $j \in B_2 \setminus B_1$ with
	$$\nu(B_1) \ttimes \nu(B_2) \geq \nu(B_1 \setminus \{i \} \cup
        \{j \}) \ttimes \nu(B_2 \setminus \{j\} \cup \{i\}).$$ Any
        matroid can be given a valuated matroid structure by setting
        $\nu(B)=0$ if $B$ is a basis of the matroid, and
        $\nu(B)=\infty$ otherwise.  We call this the trivial valuated
        matroid structure on the matroid.  Conversely, every valuated
        matroid $\mathcal M$ has an underlying matroid
        $\underline{\mathcal M}$ with bases $\{ B \in \binom{\mathcal
          G}{r} \colon \nu(B) \neq \infty \}$.  Two different
        functions $\nu_1, \nu_2$ define the same valuated matroid if there is a
        constant $\lambda \neq \infty$ such that $\nu_1(B) = \lambda
        \ttimes \nu_2(B)$.  The matroids of $\nu_1$ and $\nu_2$ are
                {\em projectively equivalent} if there is $\mathbf{b}
                \in \mathbb R^{|\mathcal G|}$ with $\nu_1(B) =
                \nu_2(B) + \sum_{i \in B} b_i$ for all subsets $B$.

	The {\em circuits} of a valuated matroid are the elements of $\Rbar^{|\mathcal G|}$ of the form
	\begin{equation} \label{eqtn:circuits}  (C_{B,i})_j = \begin{cases} \lambda \ttimes \nu(B \cup \{i \}  \setminus \{j\} ) &  j \in B \cup \{i \} \\
		\infty & \text{otherwise} \\
	\end{cases}
	\end{equation}
	for $\lambda \in \mathbb R$, $B$ a basis of $\underline{\mathcal M}$, and $i \in \mathcal G \setminus B$.  The circuit $C_{B,i}$ with $\lambda = - \nu(B)$ is the {\em fundamental circuit} of $i$ over $B$.  Note that we may have $C_{B,i} = C_{B',i'}$ if $B \neq B'$.   A {\em vector} of $\mathcal M$ is a tropical sum of circuits.

The set of circuits can be characterised by the circuit axioms
\cite{MurotaTamura}.  These say that a collection $\mathcal C$ of
elements of $\Rbar^{|\mathcal G|}$ is the set of circuits of a
valuated matroid if:
\begin{enumerate}
\item $(\infty,\dots,\infty) \not \in \mathcal C$,
  \item if
$\mathbf{c} \in \mathcal C$, then $\lambda \ttimes \mathbf{c} \in
    \mathcal C$ for all $\lambda \in \mathbb R$,
   \item we do not have
$\supp(\mathbf{c}) \subsetneq \supp(\mathbf{c}')$ for $\mathbf{c},
     \mathbf{c}' \in \mathcal C$, and
\item      
     whenever $\mathbf{c}, \mathbf{c}' \in
\mathcal C$ with $c_i = c'_i <\infty$ and $c_j < c'_j$ there is
$\mathbf{c}'' \in \mathcal C$ with $\mathbf{c}'' \geq \mathbf{c}
\tplus \mathbf{c}'$, $c''_i = \infty$, and $c''_j=c_j$.
\end{enumerate}
	
        An $r$-dimensional subspace of $K^n$ gives rise to a 
        valuated matroid on the ground set
        $\mathcal G = \{1,\dots,n\}$ by fixing an $r \times n$ matrix
        $A$ whose rows form a basis for the subspace, and setting
	$$\nu(B) = \val(\det(A_B)),$$ where $A_B$ denotes the
        submatrix of $A$ with columns indexed by $B$.  Different
        choices of the matrix $A$ may scale $\nu$ by a constant.  A
        collection of $n$ vectors $\mathcal G$ in $K^r$ also induces a
        rank-$r$ valuated matroid on $\mathcal G$ in the same way by
        forming the matrix $A$ with columns the vectors in $\mathcal
        G$.  Valuated matroids coming from a matrix $A$ in these ways are called {\em realizable}.

Deletion and contraction are fundamental operations for matroids, and
these have versions for valuated matroids which we will need.  Let $\mathcal G' \subseteq
\mathcal G$ be a subset of the ground set $\mathcal G$ of a valuated
matroid $\mathcal M$ of rank $r$.  The {\em restriction} $\mathcal M|_{\mathcal
  G'}$, or deletion of $\mathcal G \setminus \mathcal G'$, is a
valuated matroid on the ground set $\mathcal G'$ with basis valuation
function defined as follows.  Let $r' = \rk_{\underline{\mathcal
    M}}(\mathcal G')$.  Fix $E \subseteq \mathcal G \setminus \mathcal
G'$ with $\rk(\mathcal G' \cup E) = \rk(\mathcal M)$.  We define
a basis valuation function $\nu' \colon \binom{\mathcal G'}{r'} \rightarrow \Rbar$ by
$$\nu'(B) = \nu(B \cup E).$$ See \cite{DressWenzel}*{Proposition 1.2}
for a proof that this defines a valuated matroid on $\mathcal G'$.
While it depends on the choice of set $E$, all such choices define the same  valuated matroid.  The underlying matroid of $\mathcal M|_{\mathcal G'}$ is the restriction $\underline{\mathcal M}|_{\mathcal G'}$ of the underlying matroid of $\mathcal M$.

The {\em contraction} $\mathcal M/\mathcal G'$ is the valuated matroid
on $\mathcal G \setminus \mathcal G'$ with basis valuation function
defined as follows.  Let $r''$ be the rank of the contraction
$\underline{\mathcal M}/\mathcal G'$ of the underlying matroid
$\underline{\mathcal M}$.  Fix $E \subseteq \mathcal G'$ with $|E| =
\rk(E)= r-r''$.  We define a basis valuation function $\nu''
\colon \binom{\mathcal G \setminus \mathcal G'}{r''} \rightarrow
\Rbar$ by
$$\nu''(B) = \nu(B \cup E).$$ See \cite{DressWenzel}*{Proposition 1.3}
for a proof that this defines a valuated matroid on $\mathcal G
\setminus \mathcal G'$.  While it depends on the choice of set $E$,
all such choices define the same  valuated matroid.  The underlying
matroid of $\mathcal M/{\mathcal G'}$ is the contraction
$\underline{\mathcal M}/{\mathcal G'}$ of the underlying matroid of
$\mathcal M$.
	
	\subsection{The bend congruence} \label{ss:bendcongruence}
	
	When working over semirings instead of rings, the notion of
        taking the quotient of a semimodule by a submodule is replaced
        by taking the quotient by a {\em congruence}.  This is an
        equivalence relation $\sim$ on the semimodule that is
        compatible with addition and multiplication by the semiring:
        $a \sim b, c \sim d$ implies $a+c \sim b+d$ and $\lambda a
        \sim \lambda b$.
	
	In \cite{Giansiracusa2} Jeff and Noah  Giansiracusa introduce a particular
	congruence on submodules of $R^n$ when $R$ is an idempotent semiring
	(such as $\Rbar$ or $\Rbar[x_1,\dots,x_s]$) that is particularly useful in tropical geometry.  We now recall this in the context of the tropical semiring.
	
	\begin{definition}
	Let $P$ be a an $\Rbar$-subsemimodule of $\Rbar^m$.  We write
        $p \in P$ as $p=\sum_{i=1}^m p_i \mathbf{e}_i$ for $p_i \in
        \Rbar$, and set $p^{\hat{i}} = \sum_{j \neq i } p_j
        \mathbf{e}_j$.  The {\em bend congruence} of $P$ is the
        congruence $\mathcal B(P)$ on $\Rbar^m$ generated by
		$$ \{  p \sim p^{\hat{i}} \colon p \in P, 1 \leq i \leq m \}.$$
	\end{definition}  
	
	This concept generalizes to ideals in polynomial semirings, or
        more generally monoid semirings.  For a (tropical) polynomial
        $f = \sum c_{\mathbf{u}} \mathbf{x}^{\mathbf{u}} \in
        \Rbar[x_1, \dots, x_s]$ we denote by $f_{\hat{\mathbf{v}}}$
        the polynomial
	$$f_{\hat{\mathbf{v}}} = \sum_{\mathbf{u} \neq \mathbf{v}} c_{\mathbf{u}}
	\mathbf{x}^{\mathbf{u}}.$$
	
	\begin{definition}
		Fix a pointed rational polyhedral cone, and let $I$ be an ideal in $\Rbar[\sigma^{\vee} \cap M]$.  The {\em bend congruence} $\mathcal B(I)$ of $I$ is
		the congruence on $\Rbar[\sigma^{\vee} \cap M]$ generated by
		$$ \{ f \sim f_{\hat{\mathbf{v}}} : f \in I \}.$$
	\end{definition}

        \subsection{Lattices and modularity} \label{ss:latticesmodularity}

We recall the notion and some basic properties of modular flats in
geometric lattice from \cite{Ziegler91, Stanley16}*{Chapter
  3}. Modular flats play an important role in our theory. Recall that
a lattice is a poset where every pair of elements has a unique least
upper bound (join $\join$), and a unique greatest lower bound (meet
$\meet$).  Geometric lattices are atomic lattices (every
element is the join of rank one elements) that are semimodular ($\rk(x
\meet y) + \rk(x \join y) \leq \rk(x)+\rk(y))$ for all $x,y$ in the
lattice).  These are precisely the lattices of flats of matroids
\cite{StanleyEC1}*{\S 3.3}. Hence, we also call elements of \(L\) 
flats.

A lattice is modular if
\begin{equation} \label{eqtn:modular} \rk(x \join y) + \rk(x \meet y) = \rk(x) +
	\rk(y) \end{equation} for all $x, y$ in the lattice.  Lattices
of flats of matroids are always semimodular,  but not always modular.  

\begin{definition}
We say that a pair
$(x,y)$ in a semimodular lattice is modular if the equality 
\eqref{eqtn:modular} holds \cite{Birkhoff}*{\S IV.3}.  An element $x$ is modular
if $(x,y)$ is a modular pair for {\em all} elements $y \in L$.
\end{definition}

	\begin{proposition}\cite[Definition 2.1]{Ziegler91} \label{defmodflat}
		Let \(L\) be a geometric lattice. The following conditions are equivalent:
		\begin{enumerate}
			\item A flat \(F \in L\) is modular.
			\item  For \(G \leq H\) in \(L\), \(G \vee (F \wedge H)=(G \vee F) \wedge H.\)
			\item For \(G \in L\) and \(H \leq F\), \(H \vee (G \wedge F)= (H \vee G) \wedge F\). 
		\end{enumerate}
        \end{proposition}
	
		\begin{proposition} \label{propmodflat}
			\begin{enumerate}
				\item \cite[Proposition
                                  2.2(2)]{Ziegler91} Let \(L\) be a
                                  geometric lattice and \(F \in L\) be
                                  a modular flat. If \(G\) and \(H\)
                                  are flats in \(L\) with \(G \leq
                                  H\), then \[G \vee (F \wedge H)=(G
                                  \vee F) \wedge H\] is a modular flat
                                  of \([G, H]\).
				
				\item \label{item:propmodflat2}
                                  \cite[Proposition 2.3]{Ziegler91}
                                  Let \(\underline{\mathcal{M}}\) be a
                                  matroid on a ground set
                                  $\mathcal{G}$. Then
				\begin{itemize}
					\item $\emptyset$, $\mathcal{G}$ and all atoms are modular flats.
					\item  The meet of modular flats is always modular. 
				\end{itemize}
			\end{enumerate}
		\end{proposition}
		
		\begin{remark} \label{rmkmodflat}
			Let \(F\) be a modular flat in the lattice of flats of a matroid $\underline{\mathcal{M}}$ and \(F'\) be any other flat.
			\begin{enumerate}
				\item Then taking \(G=\emptyset\), and \(H=F'\) in Proposition \ref{propmodflat}(1), we see \(F' \meet F \) is a modular flat of $[\emptyset,F']$.
				
				\item Taking \(G=F'\), and
                                  \(H=\mathcal{G}\) in Proposition
                                  \ref{propmodflat}(1), we see
                                  \(E^i(j) \join F\) is a modular flat
                                  of $[F',\mathcal G]$.
			\end{enumerate}
		\end{remark}

\section{Matroids of toric vector bundles}

\label{s:matroid}

In this section we discuss matroids associated to toric vector bundles in the spirit of
\cite{DJS}.

\subsection{The polymatroid of a toric vector bundle}
Let $X_{\Sigma}$ be an $n$-dimensional toric variety.
A toric vector bundle $\mathcal E$ of rank \(r\) on a toric variety
$X_{\Sigma}$ is described by the following data, due to Klyachko
(\cite{Kly}): For each ray of the fan we associate a descending
filtration $E^i(j) \supseteq E^i(j+1)$ of a vector space $E \cong K^r$, with $E^i(j) = E$ for $j \ll 0$ and $E^i(j) = \{\mathbf{0}\}$ for $j \gg 0$.    These
filtrations need to be {\em compatible} in the sense that for every
cone $\sigma \in \Sigma$, we can write $E =
\bigoplus\limits_{\mathbf{u} \in M} E^{\sigma}_{\mathbf{u}}$
so that $E^i(j) = \sum_{\mathbf{u} \cdot \mathbf{v}_i \geq j }
E^{\sigma}_{\mathbf{u}}$ when $i \in \sigma$.  When $X_{\Sigma}$ is
smooth, so $\sigma$ is generated by part of a basis for $N$, this is
equivalent to requiring that there is a basis for $E$ such that each
$E^i(j)$ with $\rho_i \in \sigma$ is spanned by part of the
basis (cf. \cite[Remark 2.2.2]{Kly}, \cite[Page 76-77]{KD}).
	
If the filtrations do not satisfy these compatibility requirements we get a
toric reflexive sheaf, instead of a toric vector bundle; see \cite{PerT}*{\S
  4.5}.

The subspaces $E^i(j)$ induce a {\em polymatroid} as follows.  Let
$\mathcal H$ be the set of all intersections $\cap_i
 E^i(j_i)$ 
of the
$E^i(j)$
not equal to $\{\mathbf{0}\}$ or $E$. 
We define the
rank function $\rk \colon 2^{\mathcal H} \rightarrow \mathbb N$ by
\begin{equation} \label{eqtn:polymatroidrank} \rk(A) = \dim \left(\sum_{V \in A} V \right).
  \end{equation}
This satisfies the polymatroid axioms:
\begin{enumerate}
\item $\rk(\emptyset)=0$,
\item $\rk(A) \leq \rk(B)$ when $A \subseteq B$, and
  \item $\rk(A \cap B) + \rk(A \cup B) \leq \rk(A) + \rk(B)$ for all $A, B \subseteq \mathcal H$.
\end{enumerate}

\begin{definition} \label{d:polymatroid}
  The {\em polymatroid} of a toric vector bundle $\mathcal E$ on a
  toric variety $X_{\Sigma}$ is the polymatroid on the set $\mathcal
  H$ of all intersections of the $E^i(j)$ not equal to $\{\mathbf{0}\}$ or $E$ with rank function given by \eqref{eqtn:polymatroidrank}.
\end{definition}

A flat of this polymatroid is a set $A \subseteq \mathcal H$ that is
maximal among sets of its rank.  We can identify each intersection of
$E^i(j) \in \mathcal H$ with the smallest flat containing it.  The
poset of all flats under inclusion is a lattice, with meet given by
intersection.  This poset can be identified with the lattice
$L(\mathcal E)$ whose elements are all spans of intersections of
subsets of the $E^i(j)$, ordered by inclusion.

Note that for each $1 \leq i \leq s$ the $E^i(j)$ for all $j \in
\mathbb N$ form a chain in this lattice.

\begin{example} \label{e:trickypolymatroid}
Write $\rho_0=\cone(-1,-1)$, $\rho_1 = \cone(1,0)$, and
$\rho_2=\cone(0,1)$ for the rays of the fan of $\mathbb P^2$ as a toric variety.  Let
$\mathcal E$ be the toric vector bundle on $\mathbb P^2$ given by the
following filtrations on $E = K^3$ with basis
$\mathbf{e}_1,\mathbf{e}_2,\mathbf{e}_3$:{\small
$$E^0(j) = \begin{cases} E & j\leq 0
    \\ \spann(\mathbf{e}_1,\mathbf{e}_2 ) & j =1 \\ 0 & j \geq 2 \\
		\end{cases}, \, 
		E^1(j) = \begin{cases} E &  j\leq 0 \\
			\spann(\mathbf{e}_1+ \mathbf{e}_3 ) &  j =1 \\
			0  &  j \geq 2 \\
		\end{cases},  \, 
		E^2(j) = \begin{cases} E &  j\leq 0 \\
			\spann(\mathbf{e}_3 ) &  j =1 \\
			0  &  j \geq 2 \\
		\end{cases}.$$}
		
We have $\mathcal H = \{ a_1= \spann(\mathbf{e}_1, \mathbf{e}_2),
a_2=\spann(\mathbf{e}_1+\mathbf{e}_3), a_3
=\spann(\mathbf{e}_3) \}$.
The function $\rk \colon 2^{\mathcal H} \rightarrow \mathbb N$ is:
\begin{align*} & \rk(\{a_1,a_2,a_3\}) = \rk(\{a_1,a_2\}) = \rk(\{a_1,a_3\}) = 3 \\
  &   \rk(\{a_1 \}) = \rk(\{a_2,a_3\}) = 2\\
  & \rk( \{a_2\} ) = \rk( \{a_3 \})  = 1.\\
\end{align*}  
The sets $\{a_1,a_2,a_3\}, \{a_2, a_3\}, \{a_1\}, \{a_2\}, \{a_3\}$ are flats.  Figure~\ref{f:latticetricky} shows the poset of intersections of the $E^i(j)$, and the lattice of flats is shown in Figure~\ref{f:latticetricky2}.

\begin{figure}
\centering
\begin{minipage}{.4\textwidth} 
\centering
\begin{tikzpicture}
					[scale=.6,auto=left]
					\node (n0) at (3,0) {$\emptyset$};
					\node (n1) at (-0.6,3)  {\({\scriptsize\spann(e_1, e_2)}\)};
					\node (n3) at (3,2) {\({\scriptsize\spann(e_1+ e_3)}\)};
					\node (n4) at (6.5,2)  {\({\scriptsize\spann(e_3)}\)};
					\node  (n12) at (3, 6) {$\mathcal G$};
					\foreach \from/ \to in {n4/n0,n3/n0,n1/n0, n12/n1, n12/n3, n12/n4}
					\draw (\from) -- (\to);
				\end{tikzpicture}
				\captionof{figure}{ }
				\label{f:latticetricky}
			\end{minipage}%
			\begin{minipage}{.4\textwidth}
				\centering
					\begin{tikzpicture}
					[scale=.6,auto=left]
					\node (n0) at (3,0) {$\emptyset$};
					\node (n1) at (-.15,4)  {\({\scriptsize\spann(e_1, e_2)}\)};
					\node (n3) at (4.7,2) {\({\scriptsize\spann(e_1+ e_3)}\)};
					\node (n4) at (8.4,2)  {\({\scriptsize\spann(e_3)}\)};
					\node (n6) at (6.5,4) {\({\scriptsize\spann(e_1, e_3)}\)};
					\node  (n12) at (3, 6) {$\mathcal G$};
					\foreach \from/ \to in {n4/n0,n3/n0,n1/n0, n12/n1, n12/n6, n6/n3, n6/n4}
					\draw (\from) -- (\to);
				\end{tikzpicture}
				\captionof{figure}{ }
				\label{f:latticetricky2}
			\end{minipage}
		\end{figure}
\end{example}

\subsection{Matroids of a toric vector bundle}

In \cite{DJS} the authors associate a representable matroid
$\underline{\mathcal M}(\mathcal E)$ to a toric vector bundle
$\mathcal E$.  The key property of this matroid is that all
intersections of the subspaces $E^i(j)$ occur as flats of this
matroid.

Let $L'(\mathcal E)$ be the poset of all intersections $\cap_i
E^i(j_i)$, ordered by inclusion.  This is isomorphic to a subposet of
the lattice of flats of the polymatroid of $\mathcal E$ of
Definition~\ref{d:polymatroid}.  It is well known that every
polymatroid rank function is induced from that of a matroid.  This is
a special case of the fact that every lattice embeds into a geometric
lattice.  We now recall this, in a version that is relevant for this
paper.

  A function $\phi \colon L \rightarrow L'$
is an embedding of lattices if $\phi$ is an injection, $\phi(a \join
b) = \phi(a) \join \phi(b)$, and $\phi(a \meet b) = \phi(a) \meet
\phi(b)$.  Note that this implies that $a < b$ if and only if $\phi(a)
< \phi(b)$, as $a <b$ if and only if $a \join b = b$.
The original observation that lattices can be embedded into geometric
lattices goes back to Dilworth \cite{CrawleyDilworth}*{Chapter
  14}. See Remark~\ref{r:matroidliterature} for more references.

Recall that given two matroids $\underline{\mathcal M}_1,
\underline{\mathcal M}_2$ on a ground set $\mathcal G$ we say that
$\underline{\mathcal M}_1 < \underline{\mathcal M}_2$ in the weak
order if every basis of $\underline{\mathcal M}_1$ is a basis of
$\underline{\mathcal M}_2$.

\begin{theorem} \label{t:matroidoflattice}
Let $L$ be a lattice with $\hat{0}$ and $\hat{1}$ and a function $r \colon L
\rightarrow \mathbb N$ that satisfies
\begin{enumerate}
	\item (normalization) $r(\hat{0})=0$;
	\item (increasing) If $x < y$ then $r(x)<r(y)$;
	\item (semimodularity) For any $x,y \in L$, we have  $r(x \meet y) + r(x \join y) \leq r(x) + r(y)$.
\end{enumerate}
Then there is a matroid $\underline{ \mathcal M}$ for which $L$ embeds into the
lattice of flats $\mathcal L(\underline{\mathcal M})$ of $\underline{\mathcal M}$ as a
lattice, with $\rk(x)= r(x)$ for all $x \in L$.
		
There is a unique choice of such
matroid $\underline{\mathcal M}$ for which
\begin{enumerate}
	\item the ground set of $\underline{\mathcal M}$ has minimal size, and
	\item for any other matroid $\underline{\mathcal M}'$ whose lattice of flats
          contains $L$ as a meet-semilattice with the same ground set, we have
          $\underline{\mathcal M}'<\underline{\mathcal M}$ in the weak order.
\end{enumerate}

This matroid $\underline{\mathcal M}$ has no loops or parallel elements.
\end{theorem}

\begin{proof}
The main theorem of \cite{Sims} shows that there is a matroid
$\underline{\mathcal M}$ for which $L$ is isomorphic to the lattice of {\em
  cyclic} flats of $\underline{\mathcal M}$.  These are flats (also called ``fully
dependent'') that are unions of circuits.  In addition, $\rk_{\mathcal
  M}(x) = r(x)$ for all $x \in L$.  This shows the existence claim of the theorem.

We write $\mathcal G$ for the ground set of $\underline{\mathcal M}$ and $\phi(x)$
for the flat of $\underline{\mathcal M}$ corresponding to $x \in L$.  In the proof
of \cite{Sims}*{Theorem 1}, Step C shows that the matroid $\underline{\mathcal M}$
constructed has rank function given by $\rk(S) = \min_{x \in L} (r(x)
+ |S \setminus \phi(x)|)$ for all $S \subseteq \mathcal G$.  If $B \subseteq \mathcal G$ satisfies
\begin{equation} \label{eqtn:simsindependent} |B
\cap \phi(x)| \leq r(x) \text{ for all } x \in L,\end{equation} then $\rk(B) = r(y) + |B
\setminus \phi(y)|$ for some $y \in L$, and thus $\rk(B) = r(y) + |B|
- | B \cap \phi(y)| \geq r(y) + |B| - r(y) = |B|$, so all such $B$ are
independent in Sim's matroid $\underline{\mathcal M}$.

To show that there is a minimal choice for $\underline{\mathcal M}$, let $\underline{\mathcal
M}$ be any matroid satisfying the conditions of the theorem.  We now
construct a subset $\mathcal G'$ of the ground set $\mathcal G$ of
$\underline{\mathcal M}$ for which $L$ is still  isomorphic to a sublattice of
the lattice of flats of $\underline{\mathcal M}|_{\mathcal G'}$.  Let $\mathcal J$
be the poset of join irreducibles (elements that are not the join of other subsets of $L$) in $L$.  For $j \in \mathcal J$, set
$s_j = \join_{j' < j, j' \in \mathcal J} j'$.  The set $\mathcal G'$ we
construct will have size $\sum_{j \in \mathcal J} r(j)-r(s_j)$.
Since this will show that every choice of matroid $\underline{\mathcal M}$
satisfying the conditions of the theorem has a subset $\mathcal G'$ of
this size, we can conclude that this is the minimal possible size.

The construction is by induction on $\mathcal J$.  For each minimal
element $j$ of $\mathcal J$, choose a maximal independent subset of
$\phi(j)$, and set $\mathcal G'$ initially to be the union of all such
subsets.  For $j \in \mathcal J$ for which all smaller $j' \in
\mathcal J$ have been considered, choose a complement $W_j \in
\mathcal L(\underline{\mathcal M})$ to $\phi(j)$ with respect to $\phi(s_j)$.
This has the property that $W_j \join \phi(s_j) = \phi(j)$, while $W_j \meet
\phi(s_j) = \hat{0}$; it exists by \cite[Proposition 3.4.4]{White}.
Add a maximal independent subset $\mathcal G'_j$ of $W_j$ to $\mathcal G'$.  We
claim that when this procedure is finished, with all $j \in \mathcal
J$ considered, $L$ is a sublattice of the lattice of flats of
$\underline{\mathcal M}|_{\mathcal G'}$.  Indeed, for $x \in L$ set $\psi(x) =
\phi(x) \cap \mathcal G'$.  Then $\psi(x)$ is a flat of $\mathcal
M|_{\mathcal G'}$, since if there was a circuit in $\underline{\mathcal M}|_{\mathcal G'}$
with exactly one element not in $\psi(x)$, this would be a circuit
with exactly one element not in the flat $\phi(x)$ of $\underline{\mathcal M}$, which is a
contradiction.  In addition, if $y<x$ then $\psi(y) \subsetneq
\psi(x)$.  Note also that this construction removes any loops or parallel elements.

Finally, we note that $\rk(\psi(x)) = r(x)$ for all $x \in L$.  The
proof is by induction on $L$.  When $x \in L$ is minimal, it is a
minimal element of the poset $\mathcal J$ as well, so the rank is
correct by construction.  Suppose now that the claim is true for all
$y<x$.  If $x \in \mathcal J$, then by induction we have $\rk
\psi(s_x) = r(s_x)$.  During the construction of $\mathcal G'$ we add
a set $\mathcal G'_x$ of size $r(x)-r(s_x)$ to $\mathcal G'$.  This
has the property that the closure in $\underline{\mathcal M}$ of $\psi(s_x) \cup
\mathcal G'_x$, which is also the closure of $\phi(s_x) \cup \mathcal
G'_x$, is $\phi(s_x) \join W_x = \phi(x)$, which has rank $r(x)$, so
the union of a maximal independent subset of $\psi(s_j)$ and $\mathcal
G'_j$ is independent, and thus $\rk(\psi(x)) = r(x)$.  If $x \not
\in\mathcal J$, write $x = j_1\join \cdots \join j_{\ell}$ for $j_i
\in \mathcal J$ less than $x$.  By induction $\rk(\psi(j_i)) =
r(j_i)$, and thus the closure in $\underline{\mathcal M}$ of $\psi(j_i)$ is
$\phi(j_i)$ for $1 \leq i \leq \ell$.  Thus the closure in $\underline{\mathcal
M}$ of $\cup_{i=1}^{\ell} \psi(j_i)$ equals the closure of
$\cup_{i=1}^{\ell} \phi(j_i)$, which is $\phi(j_1) \join \cdots \join
\phi(j_\ell) =\phi(j_1 \join \cdots \join j_{\ell}) = \phi(x)$.  This
shows that $\rk(\cup_{i=1}^{\ell} \psi(j_i)) = \rk(\phi(x)) = r(x)$.
As $\rk(\cup_{i=1}^{\ell} \psi(j_i)) \leq \rk(\psi(x)) \leq \rk(\phi(x))$, we conclude that $\rk(\psi(x)) = r(x)$ as required.

Note that when this minimalization procedure is applied to Sim's
matroid $\underline{\mathcal M}$, a set $B \subseteq \mathcal G'$ is still
independent whenever the condition \eqref{eqtn:simsindependent}
holds.  For any matroid satisfying the conditions of the theorem with
ground set $\mathcal G$, and $x \in L$, if a set $B \subseteq \mathcal
G$ is independent then we must have $|B \cap \phi(x)| \leq r(x)$.  In
particular, if $\mathcal G = \mathcal G'$, then every basis of the
matroid is a basis of the minimalization of $\underline{\mathcal M}$, and thus
this minimalization is the largest element in the weak order of
matroids of minimal ground set satisfying the conditions of the theorem.
\end{proof}

\begin{remark} \label{r:matroidliterature}
The first proof that all lattices can be embedded into lattices of
flats of matroids is attributed to earlier unpublished work of
Dilworth in \cite{CrawleyDilworth}*{Chapter 14}.  Sims' proof in
\cite{Sims} follows the general outline given in
\cite{CrawleyDilworth}, which is also present in the minimalization
part of our proof of Theorem~\ref{t:matroidoflattice}.  A different
short proof is given in \cite{BdeM08}.  The fact that there is a
unique choice of matroid with minimal ground set that is largest in the
weak order is essentially shown by Nguyen in \cite{White}*{Chapter
  10}, in the setting that the lattice $L$ is atomic (pointed), and
our proof mimics the one given there.  The case that the lattice is
the lattice of flats of a polymatroid was observed by Helgason
\cite{Helgason} (who calls polymatroids ``hypermatroids'').
Expositions are given in \cite{Lovasz} and \cite[\S12.1]{Oxley}.  See
also \cite{BergmanFanPolymatroid} for a tropical connection.  Note
that the matroid constructed in these latter cases is rarely the
minimal one of Theorem~\ref{t:matroidoflattice}.  See
Remark~\ref{r:matroidrealizable} for the realizable case.
\end{remark}
  
Theorem~\ref{t:matroidoflattice} allows us to associate a
        matroid to a toric vector bundle.
	
\begin{definition} \label{d:matroidoftvb}
Let $\mathcal E$ be a toric vector bundle on a toric variety
$X_{\Sigma}$, with associated filtrations $E^i(j)$.  Let $L'(\mathcal
E)$ be the set of all intersections $\cap_i E^i(j_i)$, ordered by
inclusion, and $L(\mathcal E)$ be the poset consisting of all spans of
subsets of $L'(\mathcal E)$.  Note that $E$ is the maximal element
$\hat{1}$ of this poset, and the empty span $\{\mathbf{0}\}$ is the
minimal element $\hat{0}$.  The poset $L(\mathcal E)$ is a
join-semilattice by construction (as the poset of all subspaces of $E$
is lattice with join equal to span), and thus is a lattice by
\cite{StanleyEC1}*{Proposition 3.3.1}.  Define $r \colon L(\mathcal E)
\rightarrow \mathbb N$ by $r(x) = \dim(x)$.  As join in $L(\mathcal
E)$ is span (join) of the two subspaces, and the meet is contained in
the intersection (which may not be in $L(\mathcal E)$ if the subspaces
are not in $L'(\mathcal E)$) the function $r$ is semimodular.  We say
that $\underline{\mathcal M}$ is a matroid for $\mathcal E$ if $L(\mathcal E)$ is
isomorphic to a join-sublattice of the lattice of flats of $\mathcal E$.
Such a matroid always exists by applying
Theorem~\ref{t:matroidoflattice} to $L(\mathcal E)$.  We say that
$\underline{\mathcal M}$ is a minimal matroid of $\mathcal E$ if the ground set of
$\underline{\mathcal M}$ has the minimal possible size for such a matroid.
\end{definition}

\begin{example} \label{e:trickymatroid}
  Let $\mathcal E$ be the toric vector bundle on $\mathbb P^2$ of
  Example~\ref{e:trickypolymatroid}.  The lattice $L'(\mathcal E)$ of
  Definition~\ref{d:matroidoftvb} is shown in
  Figure~\ref{f:latticetricky}.  The lattice $L(\mathcal E)$ is shown
  in Figure~\ref{f:latticetricky2}.

We give three examples of matroids for $\mathcal E$.

\begin{enumerate}
\item \label{item:e:trickymatroid0} The unique matroid
  $\underline{\mathcal M}_1$ constructed by Theorem~\ref{t:matroidoflattice} for
  $L(\mathcal E)$ with function $r$ given by dimension has ground set
  $\mathcal G = \{ e_0^1,e_0^2,e_1^1,e_2^2 \}$.  The bases of
  $\underline{\mathcal M}_1$ are all subsets of $\mathcal G$ of size $3$, so
  $\underline{\mathcal M}_1$ is the uniform matroid $U(3,4)$.  When constructing a realization of
  $\underline{\mathcal M}_1$, we choose $e_1^1$ to be a multiple of
  $\mathbf{e}_1+\mathbf{e}_3$, and $e_2^1$ to be a multiple of
  $\mathbf{e}_3$.  Since $\{e_0^i, e_1^1,e_2^1\}$ is independent for $i
  =1,2$, we must then choose $e_0^1,e_0^2$ outside
  $\spann(e_1^1,e_2^1)$, so not a multiple of $\mathbf{e}_1$.  One
  such choice is $e_0^1 = \mathbf{e}_1+\mathbf{e}_2$, and $e_0^2=
  \mathbf{e}_2$, so
$$\mathcal G = \{\mathbf{e}_1+\mathbf{e}_2, \mathbf{e}_2,
  \mathbf{e}_1+\mathbf{e}_3, \mathbf{e}_3 \}.$$
\item \label{item:e:trickymatroid} Another choice is the matroid
  $\underline{\mathcal M}_2$ with ground set $\mathcal G =\{ \mathbf{e}_1,
  \mathbf{e}_2, \mathbf{e}_1+\mathbf{e}_3, \mathbf{e}_3 \}$.  We have
  $\underline{\mathcal M}_2 \neq \underline{\mathcal M}_1$, as there is a
  circuit $\{ \mathbf{e}_1, \mathbf{e}_1+\mathbf{e}_3, \mathbf{e}_3
  \}$ of size three in $\underline{\mathcal M}_2$, so it is not the uniform
  matroid, which is the largest rank-$3$ matroid on $4$ elements in the weak order.
			
\item \label{item:e:trickymatroid1}  Let $\underline{\mathcal M}_3$ be the matroid
  with ground set $\mathcal G = \{ \mathbf{e}_1+\mathbf{e}_2,
  \mathbf{e}_1-\mathbf{e}_2, \mathbf{e}_2, \mathbf{e}_1+\mathbf{e}_3,
  \mathbf{e}_3, \mathbf{e}_1+\mathbf{e}_2+\mathbf{e}_3 \}$.  Then
  $L(\mathcal E)$ is identified with the flats
  $\{\spann(\mathbf{e}_1+\mathbf{e}_2,\mathbf{e}_1-\mathbf{e}_2,
  \mathbf{e}_2), \spann(\mathbf{e}_1+\mathbf{e}_3),
  \spann(\mathbf{e}_3), \{0\}, E \}$.
  
\end{enumerate}    
\end{example}

\begin{remark} \label{r:matroidrealizable}
When the lattice $L$ is a join-sublattice of the lattice of all
subspaces of $E$, meaning that the join in $L$ is the join in the
subspace lattice, but the meet in the subspace lattice might not lie
in $L$, then the matroid of Theorem~\ref{t:matroidoflattice} can be
chosen to be realizable.  This is shown in \cite{Ziegler}*{Theorem
  4.9}, and also by \cite{DJS}*{Algorithm 3.2}.
\end{remark}

Note that for each ray $\rho_i$ of $\Sigma$, each filtered piece
$E^i(j)$ is an element of the lattice of flats $\mathcal L(\underline{\mathcal
M})$ of any matroid for $\mathcal E$, and the set of all $E^i(j)$ as $j$
decreases is a chain in $\mathcal L(\underline{\mathcal M})$.

	\begin{remark}
		The matroid associated to a toric vector bundle was
                first introduced in \cite{DJS}.
                Theorem~\ref{t:matroidoflattice} differs from that
                presentation as follows.  Firstly, \cite{DJS} only
                requires an embedding of the meet-semilattice
                $L'(\mathcal E)$; however their embedding actually
                embeds $L(\mathcal E)$, and respects joins.  Adding
                extra elements to the ground set of the matroid in our
                sense does not change the join, but might change the
                meet.  Secondly, note that for the toric vector bundle
                of Example~\ref{e:trickymatroid}, the minimal possible
                size of the ground set is $4$ for an embedding of
                either $L(\mathcal E)$ as a join-semilattice, or
                $L'(\mathcal E)$ as a meet-semilattice.  However both
                $\underline{\mathcal M_1}$ and $\underline{\mathcal
                  M}_2$ have one circuit: $\{
                \mathbf{e}_1+\mathbf{e}_2, \mathbf{e}_2,
                \mathbf{e}_1+\mathbf{e}_3, \mathbf{e}_3 \}$ for
                $\underline{\mathcal M}_1$, and $\{ \mathbf{e}_1,
                \mathbf{e}_1+\mathbf{e}_3, \mathbf{e}_3 \}$ for
                $\underline{\mathcal M}_2$, so the claim in
                \cite{DJS}*{Proposition 3.1} that there is a unique
                matroid with the number of elements in the ground set
                being minimal, and the number of circuits minimal does
                not hold.
		
		The issue here is \cite{DJS}*{Algorithm 3.2} does not construct a
		unique matroid.  Given $L(\mathcal E)$, we initialize the $G$ from that
		algorithm to contain multiples of $\mathbf{e}_1+\mathbf{e}_3$ and
		$\mathbf{e}_3$.  The algorithm then considers
		$V=\spann(\mathbf{e}_1,\mathbf{e}_2)$.  The set $G'$ of elements of
		$G$ lying in $V$ is empty, so $\spann(G') =\{ 0 \}$.  The algorithm
		then tells us to append to $G$ a basis for a complementary subspace to
		$\{0\} = \spann(G')$ in $V$, which is a basis for $V$.  However the
		choice of basis at this stage affects the matroid.  The matroids
		$\underline{\mathcal M}_1$ and $\underline{\mathcal M}_2$ are the results of two different
		choices for a basis for $V$.
		
		Other than the fact that references to ``the matroid'' should be
		replaced by ``a matroid'', we are not aware that this issue
		substantively affects any results in \cite{DJS}.  In particular, as we
		note in Remark~\ref{r:DJSok}, the results about parliaments of
		polytopes (which are no longer unique) encoding global sections still
		hold.

                 As in Theorem~\ref{t:matroidoflattice} we can
                 reimpose uniqueness on the matroid by asking that it
                 be maximal in the weak order on matroids.  However
                 this may not always be geometrically desirable.  For
                 example, the toric vector bundle of
                 Example~\ref{e:trickymatroid} is a direct sum of a
                 rank-$2$ bundle and a line bundle.  This can seen by
                 noting that every subspace in the filtration is a
                 direct sum of a subspace in
                 $\spann(\mathbf{e}_1,\mathbf{e}_3)$ and one in
                 $\spann(\mathbf{e}_2)$.  However the matroid
                 $\underline{\mathcal M}_1$, which is the largest in
                 the weak order, is not a direct sum of matroids of
                 lower ranks, while the matroid $\underline{\mathcal
                   M}_2$, which is smaller in the weak order, is a
                 direct sum.  See Section~\ref{ss:directsum} for more
                 on direct sums in the context of this paper.  A theme
                 of the approach taken here is that the correct choice
                 of matroid depends on the property desired.  See
                 Remark~\ref{r:semistabletropicalize} for more on this in the context of
                 stability of tropical toric vector bundles. 
	\end{remark}

\section{Toric vector bundles as modules over the Cox ring}
\label{s:CoxModule}

In this section we give a description of a toric vector bundle as a module over the Cox ring of the base toric variety.  
Let $\mathcal E$ be a rank-$r$ toric vector bundle on toric variety
$X_{\Sigma}$, and let $\underline{\mathcal M}(\mathcal E)$ be a matroid for $\mathcal E$
in
the sense of Definition~\ref{d:matroidoftvb}, which has a realization
$\mathcal G \subset K^r$.  For each $\mathbf{w} \in \mathcal
G$, set $\mathbf{d}_{\mathbf{w}} \in \mathbb Z^{|\Sigma(1)|}$ to be the
vector with 
\begin{equation} \label{eqtn:dw}
(\mathbf{d}_{\mathbf{w}})_i = \max(j : \mathbf{w} \in E^i(j)).
\end{equation}

Note that the vectors $\mathbf{d}_{\mathbf{w}}$, together with the realization  $\mathcal G$, actually determine the filtrations $E^i(j)$.

\begin{lemma} \label{l:realizableEfromd}
Let $\mathcal E$ be a toric vector bundle on $X_{\Sigma}$ given by
filtrations $\left( E, \{ E^i(j) \}\right) $.  Then
$$E^i(j) = \spann(\mathbf{w} \in \mathcal G : (\mathbf{d}_{\mathbf{w}})_i \geq j).$$
\end{lemma}  

\begin{proof}
By definition, if $(\mathbf{d}_{\mathbf{w}})_i \geq j$, we have
  $\mathbf{w} \in E^i(j)$, so the right-hand side is contained in
  $E^i(j)$.  For the converse, suppose $\mathbf{w} \in \mathcal G$
  lies in $E^i(j)$.  This implies that $(\mathbf{d}_{\mathbf{w}})_i
  \geq j$, so $\mathbf{w}$ is in the right-hand side.  As $E^i(j)$ is a flat in the atomic lattice of flats of the matroid $\underline{\mathcal M}(\mathcal E)$, it is 
  the span of those $\mathbf{w} \in \mathcal G$ it contains, so it
  follows that $E^i(j)$ is contained in the right-hand side.
\end{proof}

\begin{remark} \label{r:KMconnection}
In \cite{KavehManonOlder}, Kaveh and Manon give an equivalent
description of toric vector bundles in terms of piecewise linear
functions, and introduce the {\em diagram} of a toric vector bundle.  This the $|\Sigma(1)| \times |\mathcal G|$ matrix with $(i,\mathbf{w})$th entry $(\mathbf{d}_{\mathbf{w}})_i$.
\end{remark}

Write $G$ for the $r \times m$ matrix with columns the
vectors in $\mathcal G$, and $\mathbf{w}_i$ for the $i$th column of
$G$.
For a collection of vectors $A=\{\mathbf{a}_1,\dots,\mathbf{a}_l\}
\subset \mathbb Z^{|\Sigma(1)|}$, we write $\min A$ for the vector
with $(\min A)_k = \min_j((\mathbf{a}_j)_k)$.
We write $V$ for the $|\Sigma(1)| \times n$ matrix with rows the first
lattice points on the rays of $\Sigma$.

We assume for now that $X_{\Sigma}$ is smooth, so all Weil divisors
are Cartier.  Recall that Cox \cite{CoxHomogeneous} showed that in
this setting there is a correspondence between coherent sheaves on
$X_{\Sigma}$ and class-group-graded modules over a polynomial ring,
now called the Cox ring, with one variable for each ray of the fan
$\Sigma$.  This generalizes the analogous correspondence for
projective space.  When the sheaf is torus equivariant, the module has
a finer $\mathbb Z^{|\Sigma(1)|}$-grading coming from the surjection
$\mathbb Z^{|\Sigma(1)|} \rightarrow A^1(\Sigma)$ of \eqref{eqtn:A1}
(see \cite[Section 4.8]{PerT}).

We write $x_1,\dots,x_s$ for the variables of the Cox ring of $X_{\Sigma}$, and $t_1,\dots,t_n$ for the coordinates on the torus of $X_{\Sigma}$.

\begin{proposition} \label{p:CoxModuleTVB}
  Let $S=K[x_1,\dots,x_s]$ be the Cox ring of a smooth toric variety $X_{\Sigma}$, let
  $\mathcal E$ be a toric vector bundle on $X_{\Sigma}$, and let
  $\mathcal G \subseteq K^r$ be the  ground set of a matroid for
  $\mathcal E$ in the sense of Definition~\ref{d:matroidoftvb}.

For $\mathbf{c} \in \ker(G)$, and $1 \leq i \leq m$, set $\mathbf{u}_i =
  \mathbf{d}_{\mathbf{w}_i} - \min(\mathbf{d}_{\mathbf{w}_l} : c_l
  \neq 0)$.
  Then $\mathcal E$ is the sheafification
  $\widetilde{P}$ of the $\mathbb Z^{|\Sigma(1)|}$-graded module $P$
  given by
  $$P = \bigoplus_{\mathbf{w} \in \mathcal G} S(\mathbf{d}_{\mathbf{w}})/ R,$$ where $R$ is
  the submodule of $\oplus_{\mathbf{w} \in \mathcal G} S(\mathbf{d}_{\mathbf{w}})$ generated by
  $$\left\{\sum_{i=1}^m c_i \mathbf{x}^{\mathbf{u}_i}
  \mathbf{e}_i : (c_1,\dots,c_m) \in \ker(G) \right\}.$$
  \end{proposition}

\begin{proof}
Let $D_{\mathbf{w}}$ be the torus-invariant divisor $\sum
(\mathbf{d}_{\mathbf{w}})_i D_i$ on $X_{\Sigma}$.

  By \cite{DJS}*{Remark 3.6} there is a surjection
  $$\eta \colon \mathcal F := \bigoplus_{\mathbf{w} \in \mathcal G} \mathcal
  O(D_{\mathbf{w}}) \twoheadrightarrow \mathcal E.$$ This follows, via
  Klyachko's equivalence of categories, from the surjection $\eta_e
  \colon K^{|\mathcal G|} \rightarrow E$ given by
  $\mathbf{e}_{\mathbf{w}} \mapsto \mathbf{w}$.

  Let $Q = \oplus_{\alpha \in A^1(X_{\Sigma})} H^0(X_{\Sigma},\mathcal
  E \otimes_{\mathcal O_{X_{\Sigma}}} \mathcal
  O_{X_{\Sigma}}(\alpha))$.  The proof of
  \cite{CoxHomogeneous}*{Theorem 3.2} shows that the sheafification
  $\widetilde{Q}$ of $Q$ equals $\mathcal E$.  The map $\eta$ induces
  a linear transformation
$$\oplus_{\mathbf{w} \in \mathcal G} H^0(X_{\Sigma}, \mathcal
  O(D_{\mathbf{w}})) \rightarrow H^0(X_{\Sigma}, \mathcal E),$$ and
  thus, after taking the sum over all twists by $\mathcal O(\alpha)$,
  we get a homomorphism of
  $S$-modules
    $$\eta' \colon P':=\oplus_{\mathbf{w} \in \mathcal G}
  S(\mathbf{d}_{\mathbf{w}}) \rightarrow Q.$$

 We begin
  by showing that for fixed $\sigma \in \Sigma$ we have $(R_{\mathbf{x}^{\hat{\sigma}}})_{\mathbf{0}} \subseteq
  (P'_{\mathbf{x}^{\hat{\sigma}}})_{\mathbf{0}} \cong \mathcal
  F(U_{\sigma})$  equal to $\ker(\eta(U_{\sigma}))$.

Recall that we have $(S_{\mathbf{x}^{\hat{\sigma}}})_{\mathbf{0}}
\cong K[U_{\sigma}]$, via the map that takes $\mathbf{x}^{\mathbf{v}}$ to $\mathbf{t}^{\mathbf{u}} \in K[U_{\sigma}]$, where $\mathbf{u} \in M$ is the unique element with  $V\mathbf{u} =
\mathbf{v}$.  For each $\mathbf{w} \in \mathcal G$, choose
$\mathbf{x}^{\mathbf{v}_{\mathbf{w}}} \in
S_{\mathbf{x}^{\hat{\sigma}}}$ with
$\deg(\mathbf{x}^{\mathbf{v}_{\mathbf{w}}}) = [D_{\mathbf{w}}] \in
A^1(X_{\Sigma})$ and $(\mathbf{v}_{\mathbf{w}})_i =0$ whenever
$\mathbf{v}_i \in \sigma$.  This is possible since $X_{\Sigma}$ is
smooth.  Then
$(S(\mathbf{d}_{\mathbf{w}})_{\mathbf{x}^{\hat{\sigma}}})_{\mathbf{0}}
\cong \mathbf{x}^{\mathbf{v}_{\mathbf{w}}}
(S_{\mathbf{x}^{\hat{\sigma}}})_{\mathbf{0}}$.  Since
$\deg(\mathbf{x}^{\mathbf{v}_{\mathbf{w}}-\mathbf{d}_{\mathbf{w}}}) =
\mathbf{0}$,  there is $\mathbf{u}_{\mathbf{w}} \in M$ with
$V\mathbf{u}_{\mathbf{w}} =
\mathbf{v}_{\mathbf{w}}-\mathbf{d}_{\mathbf{w}}$.  Thus as an
$M$-graded $K[U_{\sigma}]$ module,
$(S(\mathbf{d}_{\mathbf{w}})_{\mathbf{x}^{\hat{\sigma}}})_{\mathbf{0}}$
is isomorphic to $\mathbf{t}^{\mathbf{u}_{\mathbf{w}}} K[U_{\sigma}]$
\cite{Fulton}*{p62}.

Thus
\begin{equation}
\label{eqtn:localisom}
\mathcal F(U_{\sigma}) = (P'_{\mathbf{x}^{\hat{\sigma}}})_{\mathbf{0}} \cong
\oplus_{\mathbf{w} \in \mathcal G} \mathbf{t}^{\mathbf{u}_{\mathbf{w}}}  K[U_{\sigma}].
\end{equation}
We also have that 
 $\mathcal E$ is a direct sum of line bundles on
 $U_{\sigma}$, so $\mathcal E(U_{\sigma}) \cong \oplus_i \mathbf{t}^{\mathbf{u}'_i}
 K[U_{\sigma}]$ as an $M$-graded module.  We write $\sum_{\mathbf{w} \in
   \mathcal G} \mathbf{t}^{\mathbf{u}_{\mathbf{w}}} g_{\mathbf{w}}
 \mathbf{f}_{\mathbf{w}}$, with $g_{\mathbf{w}} \in K[U_{\sigma}]$,
 for an element of $\mathcal F(U_{\sigma})$, and $\sum_i
 \mathbf{t}^{\mathbf{u}'_i} g_i \mathbf{e}'_i$ for an element of
 $\mathcal E(U_{\sigma})$, where the $\mathbf{f}_{\mathbf{w}}$ and $\mathbf{e}'_i$ are basis vectors.   The morphism $\eta(U_{\sigma})$ is given
 by an $r \times |\mathcal G|$ matrix $A$ with entries in
 $K[U_{\sigma}]$, with $\eta(U_{\sigma})(\mathbf{t}^{\mathbf{u}_{\mathbf{w}}}
  \mathbf{f}_{\mathbf{w}}) = \sum_i A_{i\mathbf{w}}
  \mathbf{t}^{\mathbf{u}'_i} \mathbf{e}'_i$.  Since
 $\eta$ is a morphism of toric vector bundles,
 $\eta(U_{\sigma})$ is equivariant, so each entry of $A$ has the form
 $a\mathbf{t}^{\mathbf{u}}$ for some $a \in K$, $\mathbf{u}
 \in \sigma^{\vee} \cap M$.  The morphism $\eta_e$ is $\eta \otimes
 (K[T]/\langle t_i -1 : 1 \leq i \leq n\rangle)_{\langle t_i-1: 1 \leq i \leq n \rangle}$, so
 $A_{i\mathbf{w}} = w_i \mathbf{t}^{\mathbf{u}}$ for some
 $\mathbf{u}$.  Since
 $\eta(U_{\sigma})$ is a degree-zero
 homomorphism, we have  $\mathbf{u} =
 \mathbf{u}_{\mathbf{w}}-\mathbf{u}'_i$.   The submodule
 $\ker(\eta(U_{\sigma}))$ of $\mathcal F(U_{\sigma})$ is thus
 generated by $$\left\{ \sum_{\mathbf{w} \in \mathcal G} a_{\mathbf{w}} \mathbf{t}^{\mathbf{u}_{\mathbf{w}}} \mathbf{f}_{\mathbf{w}} : (a_{\mathbf{w}}) \in \ker(G)
 \right\}.$$

The compatibility
condition for $\mathcal E$ implies that there is a basis
$B_{\sigma}$ for $\underline{\mathcal M}(\mathcal E)$ with each $E^i(j)$ with $\rho_i \in
\sigma$ spanned by a collection of elements of $B_{\sigma}$.
Fix $k \not \in B_{\sigma}$. Let $\mathbf{c}^k \in \ker(G)$
be the vector whose tropicalization is the fundamental circuit of $k$  over
$B_{\sigma}$, so $c^k_{\mathbf{w}_k}=1$, $c^k_{\mathbf{w}} =0$ for
$\mathbf{w} \not \in B_{\sigma} \cup \{k\}$, and
$$\mathbf{w}_k = - \sum_{\mathbf{w} \in B_{\sigma}} c^k_{\mathbf{w}}
\mathbf{w}.$$ The elements
\begin{equation} \label{eqtn:gensker}
  \left\{ \sum_{\mathbf{w}} c^k_{\mathbf{w}}
\mathbf{t}^{\mathbf{u}_{\mathbf{w}}} \mathbf{f}_\mathbf{w} : \mathbf{c}^k
\text{ is the fundamental circuit of } k \text{ over } \mathcal
B_{\sigma} \text{ for all } k \not \in B_{\sigma}  \right\}
\end{equation}
generate $\ker(\eta(U_{\sigma}))$.

We now check that these elements also generate
$(R_{\mathbf{x}^{\hat{\sigma}}})_{\mathbf{0}}$.  For fixed $i \in
\sigma$, $j \in \mathbb Z$, and $k \not \in B_{\sigma}$, if
$\mathbf{w} \not \in E^i(j)$ for some $\mathbf{w} \in \mathcal
B_{\sigma}$ with $c^k_{\mathbf{w}} \neq 0$, then since $\{ \mathbf{w}
\in B_{\sigma} \} \subseteq K^r$ is linearly independent and
contains a basis for $E^i(j)$, we have $\mathbf{w}_k \not \in E^i(j)$.
This implies that $(\mathbf{d}_{\mathbf{w}_k})_i \leq
(\mathbf{d}_{\mathbf{w}})_i$.  However if all $\mathbf{w} \in \mathcal
B_{\sigma}$ with $c^k_{\mathbf{w}} \neq 0$ lie in $E^i(j)$, then we
also have $\mathbf{w}_k \in E^i(j)$.  Thus
$(\mathbf{d}_{\mathbf{w}_k})_i = \min( (\mathbf{d}_{\mathbf{w}})_i :
\mathbf{w} \in B_{\sigma}, c^k_{\mathbf{w}} \neq 0 )$, so
$(\mathbf{u}_k)_i = 0$.

This means
that $(\mathbf{u}_{k})_i=0$ if $\rho_i \in \sigma$.  Thus
$\mathbf{x}^{-\mathbf{u}_k} \in S_{\mathbf{x}^{\hat{\sigma}}}$, and so $\mathbf{e}_k + \sum_{j \in B_{\sigma}} c_j
\mathbf{x}^{\mathbf{u}_j-\mathbf{u}_k} \mathbf{e}_j \in
P'_{\mathbf{x}^{\hat{\sigma}}}$.  We claim that these elements, as $k$
varies over the complement of $B_{\sigma}$, generate $R_{\mathbf{x}^{\hat{\sigma}}}$.  If not, after subtracting
multiples of these elements from a generator, we would have an element
$\sum_{j \in B_{\sigma}} c'_j \mathbf{x}^{\mathbf{v}_j}
\mathbf{e}_j \in R_{\mathbf{x}^{\hat{\sigma}}}$.  Since the
coefficients of elements of $R$ lie in $\ker(G)$, this would
contradict $B_{\sigma}$ being a basis.

Recall that $\mathbf{v}_{\mathbf{w}_k}$ has $\deg(\mathbf{x}^{
  \mathbf{v}_{\mathbf{w}_k}}) = [D_{\mathbf{w}_k}]$ and $(
\mathbf{v}_{\mathbf{w}_k})_i =0$ when $\rho_i \in \sigma$, so $\mathbf{x}^{\mathbf{v}_{\mathbf{w}_k}}$ is a unit in $S_{\mathbf{x}^{\hat{\sigma}}}$.  Thus the
elements $\mathbf{x}^{\mathbf{v}_{\mathbf{w}_k}}
(\mathbf{e}_{\mathbf{w}_k} + \sum_{\mathbf{w} \in B_{\sigma}}
c_{\mathbf{w}_k} \mathbf{x}^{\mathbf{u}_{\mathbf{w}}-\mathbf{u}_k}
\mathbf{e}_{\mathbf{w}})$ for $k \in B_{\sigma}$ have degree
$\mathbf{0} \in A^1(X_{\Sigma})$, and also generate $R_{\mathbf{x}^{\hat{\sigma}}}$, so generate $(R_{\mathbf{x}^{\hat{\sigma}}})_{\mathbf{0}}$.  Under the isomorphism
\eqref{eqtn:localisom}, these are taken to the elements
\eqref{eqtn:gensker}.  This shows that $(R_{\mathbf{x}^{\hat{\sigma}}})_{\mathbf{0}} \subseteq
(P'_{\mathbf{x}^{\hat{\sigma}}})_{\mathbf{0}} \cong \mathcal
F(U_{\sigma})$ is equal to $\ker(\eta(U_{\sigma}))$ as claimed.

We thus have $\widetilde{R} = \ker(\eta)$, so
$$0 \rightarrow \widetilde{R} \rightarrow \tilde{P'} \cong \mathcal F  \stackrel{\eta}{\rightarrow} \mathcal E \rightarrow 0$$
is exact.  Since the sheafification map is exact \cite{CoxHomogeneous}*{Proposition 3.1}, it follows that
$$\widetilde{P} = \widetilde{P'/R} \cong  \mathcal E$$ as required.
\end{proof}

\begin{example} \label{e:realizableTangentBundle}
	Consider the tangent bundle $\mathscr T$ on $\mathbb
	P^2$.   We use the notation of \cite{DJS}*{Example 3.8}.  The matrix $G$ equals
	$$G = \begin{pmatrix} 1 & 0 & -1 \\ 0 & 1 & -1 \\ \end{pmatrix}.$$ The
	degrees $\mathbf{d}_{\mathbf{e}_i}$ are $(1,0,0)$, $(0,1,0)$,
	and $(0,0,1)$.  The Cox ring of $\mathbb P^2$ is $S = 
	K[x,y,z]$.  We have $\mathscr T$ equal to $\widetilde{P}$,
	where $P$ is the $S$-module $$P = (S(1,0,0) \oplus S(0,1,0) \oplus S(0,0,1)) / \langle (x\mathbf{e}_1+y\mathbf{e}_2+z\mathbf{e}_3) \rangle.$$
	Note that on the chart $x \neq 0$ this becomes
	\begin{multline}
		(P_x)_{\mathbf{0}} =  x K[y/x,z/x] \oplus  yK[y/x,z/x] \oplus z 
		K[y/x,z/x] / \langle (\mathbf{e}_1+y/x\mathbf{e}_2+z/x\mathbf{e}_3)
		\rangle \\ = x K[y/x,z/x] \oplus
		y K[y/x,z/x].\\
	\end{multline}
This is free of rank $2$ as expected.
\end{example}

\begin{remark}
A similar construction occurs in \cites{GeorgeManon1,GeorgeManon2, GeorgeManon3},
where George and Manon construct the Cox rings of some families of
projectivized toric vector bundles.  For those families, the Cox
module described in Proposition~\ref{p:CoxModuleTVB} is a graded
submodule of the associated Cox ring.
\end{remark}

		\section{Tropical toric vector bundles} \label{s:tropicaltoricdefns}
		
		In this section we make the main definition of the paper: in
		Definition~\ref{d:tropicaltvb} we define a {\em tropical} toric vector
		bundle on a tropical toric variety.
		
		\subsection{Tropical toric reflexive sheaves}

		We first define a tropical toric reflexive sheaf on $\trop(X_{\Sigma})$.

\begin{definition} \label{d:tropicalreflexivesheaf}
A tropical toric reflexive sheaf $\mathcal{E}$ on $\trop(X_{\Sigma})$
is given by the data of a simple (no loops or parallel elements)
valuated matroid $\mathcal M$ of rank $r$ on a ground set $\mathcal
G$, together with for each ray $\rho_i$ of $\Sigma$ a collection $\{
E^i(j) : j \in \mathbb Z \}$ of flats of the underlying matroid
$\underline{\mathcal M}$ with the property that $E^i(j) \leq E^i(j')$
if $j >j'$, $E^i(j)= \emptyset$ for $j \gg 0$, and $E^i(j) = \mathcal G$
for $j \ll 0$.
			
Moreover, we call a tropical toric reflexive sheaf {\em tropically
  minimal} if the lattice of flats of $\underline{\mathcal M}$ does
not contain a proper subposet containing the $E^i(j)$ that is the
lattice of flats of a matroid.
			
We denote the tropical toric reflexive sheaf $\mathcal{E}$ by \(\left(
\mathcal M, \mathcal G, \{E^i(j)\}\right) \).
		\end{definition}

For $\mathbf{w} \in \mathcal G$, set $\mathbf{d}_{\mathbf{w}} \in
\mathbb Z^{|\Sigma(1)|}$ to be the vector with
$(\mathbf{d}_{\mathbf{w}})_i = \max(j : \mathbf{w} \in E^i(j) )$.
This is directly analogous to \eqref{eqtn:dw}, however here each $E^i(j)$ is
  a subset of the set $\mathcal G$, and $\mathbf{w} \in \mathcal G$.
			
Let $S=\Cox(\trop(X_{\Sigma})) = \Rbar[x_1,\dots,x_s]$.  Fix
$m=|\mathcal G|$, and let $R$ be the $S$-subsemimodule of
$\oplus_{\mathbf{w} \in \mathcal G} S(\mathbf{d}_{\mathbf{w}})$ generated by
			$$\left\{ \sum_{k=1}^m c_i \mathbf{x}^{\mathbf{u}_k} \mathbf{e}_k : (c_1,\dots,c_m) \text{ is a vector of the valuated matroid } \mathcal
			M \text{ and }
\mathbf{u}_k = \mathbf{d}_{\mathbf{w}_k} - \min(\mathbf{d}_{\mathbf{w}_l} : c_l \neq \infty)\right\}.$$

The tropical toric reflexive sheaf on $\trop(X_{\Sigma})$ given by
this data is described by the $\mathbb Z^{|\Sigma(1)|}$-graded
$S$-semimodule
\begin{equation} \label{eqtn:CoxModuleTropical}  P = \bigoplus_{\mathbf{w} \in \mathcal G}
                     S(\mathbf{d}_{\mathbf{w}}) / \mathcal B(R),
             \end{equation}
             where $\mathcal B()$ is the Giansiracusa bend congruence;
             see Section~\ref{ss:bendcongruence}.  The semimodule $P$
             is the tropicalization of the module $P$ of
             Proposition~\ref{p:CoxModuleTVB}.

\begin{remark}
	Note that we can recover the filtrations $E^i(j)$ from $\{
	\mathbf{d}_{\mathbf{w}} : \mathbf{w} \in \mathcal G \}$ as in
	Lemma~\ref{l:realizableEfromd}.  Indeed, for $\rho_i \in \Sigma(1)$ and $j \in \mathbb Z$ set
	$$E^i_j= \bigvee\limits_{(\mathbf{d}_{\mathbf{w}})_i \geq j}
        \mathbf{w},$$ where $\join$ is join in the lattice $\mathcal
        L(\underline{\mathcal M})$ of flats of $\underline{\mathcal
          M}$.  We will show that $E^i_j = E^i(j)$.  If
        $(\mathbf{d}_{\mathbf{w}})_i \geq j$ for some $\mathbf{w} \in
        \mathcal G$, then $\mathbf{w} \in E^i(j)$ by the definition of
        $\mathbf{d}_{\mathbf{w}}$, so $E^i_j \leq E^i(j)$.
        Conversely, if $\mathbf{w} \in \mathcal G$ is contained in
        $E^i(j)$ for some $i,j$, then $(\mathbf{d}_{\mathbf{w}})_i
        \geq j$ by the definition of $\mathbf{d}_{\mathbf{w}}$, so
        $\mathbf{w} \in E^i_j$.  This shows that $E^i(j) \leq E^i_j$,
        and thus $E^i(j)= E^i_j$.
\end{remark}

\begin{example}\label{e:reflexive}

\begin{enumerate} 
	
	\item \label{item:trs1} A rank-$1$ tropical reflexive sheaf is
          given by the data \(\left( \mathcal M, \mathcal G,
          \{E^i(j)\}\right) \), where \(\rk(\mathcal
          M)=1\). The only information in the filtrations
          \(\{E^i(j)\}\) is the smallest integer $a_i$ for which
          $E^i(a_i) \neq \emptyset$, so we may identify such a sheaf
          with the vector $\mathbf{a} = (a_1,\dots,a_s) \in \mathbb
          Z^s$.  This is the tropical version of $\mathcal O(\sum_{i=1}^s a_i D_i)$. 
          In view of \cite[Theorem 8.2.3]{CLS}, we can define tropical canonical sheaf
          $\trop(\omega_{X_{\Sigma}})$ as the rank-$1$ sheaf where all $a_i$ equal $-1$.
	
	\item \label{item:trs2} (Tropical tangent sheaf). Let
          \(X_{\Sigma}\) be a complete toric variety. We define the
          tropical tangent reflexive sheaf \(\trop(\mathscr{T})\) on
          $\trop(X_{\Sigma})$ as follows.
          Let
          $\mathcal{G}=\{\mathbf{v}_1,\dots,\mathbf{v}_s\} \subseteq
          N \otimes_{\mathbb Z} K$. Consider the simple representable matroid
          on the ground set $\mathcal{G}$ (i.e., we identify
          \(\mathbf{v}_i\) with \(\mathbf{v}_j\) if they are scalar
          multiples of each other). Let $\mathcal{M}$ be the trivial
          valuated matroid structure on this matroid. Consider the following
          filtration of flats
	
	\[ \mathscr{T}^{i}(j) = \left\{ \begin{array}
		{r@{\quad \quad}l}
	\mathcal{G} & i \leq 0 \\ 
		\{\mathbf{v}_i\} & j=1 \\
		\emptyset & j > 1
	\end{array} \right. .\] 
	Then \(\trop(\mathscr{T})=(\mathcal{M}, \mathcal{G},
        \{\mathscr{T}^{i}(j)\})\) is a tropical reflexive sheaf on
        $\trop(X_{\Sigma})$.  See
        Example~\ref{e:tropicalizationegs} \eqref{item:tropicalizationtangentsheaf} for an interpretation of this as a tropicalization of the tangent sheaf.
	
	\item \label{item:trs3} (Uniform matroid as tangent bundle on
          \(\mathbb{P}^n\)).  Consider the uniform matroid \(U_{n,
            n+1}\) on the ground set \(\mathcal{G}=\{\mathbf{w}_0, \ldots,
          \mathbf{w}_n\}\), where any proper subset of $\mathcal{G}$ is
          independent, and let $\mathcal M$ be the trivial valuated
          matroid structure on this uniform matroid.  Let
          \(X_{\Sigma}=\mathbb{P}^n\) be projective space. Let \(\mathbf{e}_1,
          \ldots, \mathbf{e}_n\) denote the standard basis of
          $\mathbb{Z}^n$. The fan $\Sigma$ consists of rays
          \(\mathbf{v}_1=\mathbf{e}_1, \ldots,
          \mathbf{v}_n=\mathbf{e}_n\) and
          \(\mathbf{v}_0=-\mathbf{e}_1- \ldots -\mathbf{e}_n\) and the
          maximal cones generated by proper subsets of
          $\{\mathbf{v}_0, $\ldots$, \mathbf{v}_n\}$. Consider the
          family of filtrations
	\begin{center}
		\begin{tabular}{ccc}
			$E^{\mathbf{v}_i}(j) = \begin{cases} \mathcal{G}  & j \leq 0 \\ \{\mathbf{w}_i\} & j =1 \\  \emptyset & j > 1,  \\ \end{cases} \text{ for } i=0, \ldots, n$ 
		\end{tabular}.
	\end{center}

	This defines a tropical reflexive sheaf on $\trop(\mathbb{P}^n)$, which is the tangent sheaf by Example \ref{e:reflexive} \eqref{item:trs2}.
	
	\item \label{item:trs4} Consider the singular toric surface
          \(X_{\Sigma}\), where the rays of the fan are given by
          \(\mathbf{v}_1=(1,0), \mathbf{v}_2=(1,2)\) and
          \(\mathbf{v}_3=(-1,1)\). The maximal cones of $\Sigma$ are
          given by
	\begin{equation*}
		\sigma_1=\cone(\mathbf{v}_1, \mathbf{v}_2), \, \sigma_2=\cone(\mathbf{v}_2, \mathbf{v}_3) \text{ and } \sigma_3=\cone(\mathbf{v}_1, \mathbf{v}_3).
	\end{equation*}
	Let \(\mathbf{e}_1, \mathbf{e}_2\) denote the standard basis of
        \(E=K^2\). Consider the representable matroid
        \(\underline{\mathcal M}\) of rank \(2\) on the ground set $$\mathcal G=\{\mathbf{e}_1, \mathbf{e}_2, \mathbf{e}_1+\mathbf{e}_2\}.$$ Consider
        the following family of filtrations
	
		$$E^{1}(j) = \begin{cases} \mathcal G &  j\leq 1 \\
			\{\mathbf{e}_1\} &  j =2 \\
			\emptyset  &  j > 2 \\
		\end{cases}, \, 
		E^{2}(j) = \begin{cases} \mathcal G &  j\leq 1 \\
		\{\mathbf{e}_2\} &  j =2 \\
			\emptyset  &  j > 2 \\
		\end{cases},  \, 
		E^{3}(j) = \begin{cases} \mathcal G &  j\leq 1 \\
			\{\mathbf{e}_1+\mathbf{e}_2\} &  j =2 \\
			\emptyset  &  j > 2 \\
		\end{cases}.$$

	\item \label{item:trs6} Let $X_{\Sigma} = \mathbb P^1$.  Let
          $\mathcal M$ be the valuated matroid (with the trivial
          valuation) corresponding to the five vectors in $K^3$:
          $\mathcal G = \{\mathbf{w}_1=(1,0,0), \mathbf{w}_2=(1,1,0),
          \mathbf{w}_3=(0,1,0), \mathbf{w}_4=(0,0,1),
          \mathbf{w}_5=(0,1,1) \}$.  This has the lattice of flats
          shown in Figure~\ref{f:Matroideg}.
	
	\begin{figure}
		\centering
		\begin{minipage}{.5\textwidth}
			\centering
			\begin{tikzpicture}
				[scale=.5,auto=left]
				\node (n0) at (3,0) {$\emptyset$};                                
				\node (n1) at (-1,2)  {1};
				\node (n2) at (1,2)  {2};
				\node (n3) at (3,2) {3};
				\node (n4) at (5,2)  {4};
				\node (n5) at (7,2)  {5};
				\node (n6) at (-2,4) {123};
				\node (n7) at (0,4)  {14};
				\node (n8) at (2,4) {15};
				\node (n9) at (4,4)  {24};
				\node (n10) at (6,4)  {25};
				\node (n11) at (8,4)  {345};
				\node  (n12) at (3, 6) {$\mathcal G$};                                
				\foreach \from/ \to in {n5/n0,n4/n0,n3/n0,n2/n0,n1/n0,n6/n1,n6/n2,n6/n3,n7/n1,n7/n4,n8/n1,n8/n5,n9/n2,n9/n4, n10/n2, n10/n5,n11/n3, n11/n4, n11/n5, n12/n6, n12/n7, n12/n8, n12/n9, n12/n10, n12/n11}
				\draw (\from) -- (\to);
			\end{tikzpicture}
			\captionof{figure}{}
			\label{f:Matroideg}
		\end{minipage}%
		\begin{minipage}{.4\textwidth}
			\centering
			\begin{tikzpicture}
				[scale=.5,auto=left]
				\node (n0) at (3,0) {$\emptyset$};                                
				\node (n1) at (1,2)  {1};
				\node (n3) at (3,2) {3};
				\node (n4) at (5,2)  {4};
				\node (n6) at (1,4) {123};
				\node (n7) at (3,4)  {14};
				\node (n11) at (5,4)  {345};
				\node  (n12) at (3, 6) {$\mathcal G$};                                
				\foreach \from/ \to in {n4/n0,n3/n0,n1/n0,n6/n1,n6/n3,n7/n1,n7/n4,n11/n3, n11/n4, n12/n6, n12/n7, n12/n11}
				\draw (\from) -- (\to);
			\end{tikzpicture}
			\captionof{figure}{}
			\label{f:Matroideg1}
		\end{minipage}
	\end{figure}
	    
	The fan of $\mathbb P^1$ has two rays: $\rho_0$ and $\rho_1$.  Consider the filtrations
	
	\begin{tabular}{ccc}
	  $E^0(j) = \begin{cases}
            \mathcal G
            & j \leq 0 \\ \{\mathbf{w}_1 ,\mathbf{w}_2,
            \mathbf{w}_3\} & j=1 \\ \{\mathbf{w}_1\} & j=2
            \\ \emptyset & j > 2 \\ \end{cases}$ & \hspace{1cm} &
	  $E^1(j) = \begin{cases}
            \mathcal G
            & j \leq 0 \\  \{\mathbf{w}_3, \mathbf{w}_4, \mathbf{w}_5\} & j=1 \\ \{\mathbf{w}_3\} & j=2 \\ \emptyset & j > 2 .\\ \end{cases}$\\
	\end{tabular}
	
	Note that $\mathcal E$ is not tropically minimal, as the subposet of Figure~\ref{f:Matroideg1} is the lattice of flats of the
	uniform matroid $U(3,3)$.  Replacing $\mathcal M$ by the (trivially
	valuated) valuated matroid corresponding to the vectors
	$\{\mathbf{w}_1=(1,0,0),\mathbf{w}_3=(0,0,1), \mathbf{w}_4=(0,0,1)\}$ also gives a tropical reflexive sheaf.

\end{enumerate}

\end{example}

\subsection{Tropical toric vector bundles}

We now generalize Definition~\ref{d:tropicalreflexivesheaf} to define  tropical toric vector bundles.

\begin{definition} \label{d:tropicaltvb}

A rank-$r$ tropical toric vector bundle on $\trop(X_{\Sigma})$ is a
rank-$r$ tropical toric reflexive sheaf \(\left( \mathcal{M}, \mathcal{G}, \{E^i(j)\}\right) \) that satisfies the following compatibility
condition:

For any $\sigma \in \Sigma$ there is a multiset $\mathbf{u}(\sigma)
\subseteq M$, and a basis $B_{\sigma} = \{ \mathbf{w}_{\mathbf{u}} :
\mathbf{u} \in \mathbf{u}(\sigma)\}$ of $\underline{\mathcal M}$ such
that for every ray $\rho_i \in \sigma$ and $j \in \mathbb Z$ we have
\begin{equation} \label{eqtn:tropicalcompatibility}
	E^i(j) = \bigvee\limits_{  \mathbf{u} \in B_{\sigma}, \mathbf{u} \cdot \mathbf{v}_i \geq j } \mathbf{w}_{\mathbf{u}}.
\end{equation}
\end{definition}

When the toric variety $X_{\Sigma}$ is smooth, we can simplify the
compatibility condition \eqref{eqtn:tropicalcompatibility}, as in the
case of usual toric vector bundles.  This should be compared with \cite{Kly}*{Remark 2.2.2}.

\begin{lemma}  \label{l:compatibilitymeansbasis}
	If $X_{\Sigma}$ is a smooth toric variety then condition~\eqref{eqtn:tropicalcompatibility} is equivalent to
	the condition that
	\begin{equation}  \label{cond:basis} 
		\begin{split}
			&\text{for every } \sigma \in \Sigma \text{ there is a
		basis } B_{\sigma} \text{ of }\underline{\mathcal M} \text{ such
            that} \text{ each } E^i(j) \text{  for } i \in \sigma \\
        &  \text{  is the join of
                elements of } B_{\sigma} .
		\end{split}
             \end{equation}
\end{lemma}

\begin{proof}
Suppose that $X_{\Sigma}$ is smooth, and the data $\mathcal E = \{ \mathcal M,
\mathcal G, E^i(j) \}$ defines a tropical toric reflexive sheaf.  We first
suppose that the tropical reflexive sheaf satisfies
\eqref{eqtn:tropicalcompatibility}.
Condition~\ref{cond:basis} follows immediately, as it is an a
priori weaker condition than \eqref{eqtn:tropicalcompatibility}, and
the same basis $B_{\sigma}$ works.

Now suppose that the tropical reflexive sheaf $\mathcal E$ satisfies
Condition~\ref{cond:basis}.  Fix $\sigma \in \Sigma$, and a basis
$B_{\sigma}$ of $\underline{\mathcal M}$ such that
Condition~\ref{cond:basis} holds.  For $\mathbf{w} \in B_\sigma$,
consider $(\mathbf{d}_{\mathbf{w}})_i = \max(j: \mathbf{w} \in
E^i(j))$.  We then have $\mathbf{w} \in B_{\sigma} \cap E^i(j)$ if and
only if $j \leq (\mathbf{d}_{\mathbf{w}})_i$.  Since $\sigma$ is a
unimodular cone, for each $\mathbf{w} \in B_{\sigma}$ there is
$\mathbf{u}_{\mathbf{w}} \in M$ with $\mathbf{u}_{\mathbf{w}} \cdot
\mathbf{v}_i = (\mathbf{d}_{\mathbf{w}})_i$ for all $i$ with $\rho_i
\in \sigma$.  By construction $\mathbf{w} \in B_{\sigma} \cap E^i(j)$
if and only if $\mathbf{u}_{\mathbf{w}} \cdot \mathbf{v}_i \geq j$.
Since $E^i(j)$ is the join of elements of $B_{\sigma}$, and
$B_{\sigma}$ is a basis and thus independent, we have that $E^i(j)$ is
the join of all elements of $B_{\sigma}$ that it contains.  Thus the
multiset $\mathbf{u}(\sigma):=\{ \mathbf{u}_{\mathbf{w}} : \mathbf{w}
\in B_{\sigma} \} \subseteq M$, and the basis $B_{\sigma}$ satisfies
\eqref{eqtn:tropicalcompatibility}.
\end{proof}  

As an application of the above lemma we determine which of the tropical reflexive sheaves in Example \ref{e:reflexive} are in fact  vector bundles.

\begin{example}
	\begin{enumerate}
		\item \label{item:tvb1} When $X_{\Sigma}$ is smooth,
                  then using Lemma \ref{l:compatibilitymeansbasis}, we
                  see that the tropical reflexive sheaves in Example
                  \ref{e:reflexive} \eqref{item:trs2} and
                  \eqref{item:trs3} are examples of tropical vector
                  bundles. In particular, the tropical tangent sheaf
                  on $\trop(\mathbb{P}^n)$ in Example
                  \ref{e:reflexive} \eqref{item:trs3} is a tropical
                  toric vector bundle. More generally, observe that
                  any choice of filtrations of flats the uniform
                  matroid \(U_{n, n+1}\) will give rise to a tropical
                  vector bundle on $\trop(\mathbb{P}^n)$ as the
                  compatibility condition will always hold using Lemma
                  \ref{l:compatibilitymeansbasis}.
		
		\item \label{item:tvb2} The filtrations
                  in Example \ref{e:reflexive} \eqref{item:trs4} do
                  not satisfy the compatibility condition for the cone
                  $\sigma_1$.  If they did, there would be weights $\mathbf{u}_{\mathbf{w}_1} = (a_1,b_1)$ and
                   \(\mathbf{u}_{\mathbf{w}_2} = (a_2,b_2)\) in $\mathbb{Z}^2$. From the
                  filtration \(E^{\mathbf{v}_1}(j)\), we have
		\begin{equation*}
			a_1=2 \text{ and } a_2=1.
		\end{equation*}
		Similarly, from the filtration \(E^{\mathbf{v}_2}(j)\), we have
		\begin{equation*}
			a_1+2 b_1=1 \text{ and } a_2 + 2 b_2=2.
		\end{equation*}
		Note that these equations do not have an integer solution. Hence, this is an example of tropical reflexive sheaf that is not a vector bundle.

	      \item \label{item:trs5}
          Consider the representable matroid
          \(\underline{\mathcal M}\) of rank-\(2\) on the ground set $$\mathcal G=\{\mathbf{w}_1=\mathbf{e}_1, \mathbf{w}_2=\mathbf{e}_2,
          \mathbf{w}_3=\mathbf{e}_1+\mathbf{e}_2\}.$$

          We define a tropical toric reflexive sheaf on the tropical
          toric variety $\trop(\mathbb A^3)=\Rbar^3$, whose fan has three rays
          that span a single cone.  Consider the filtrations on these rays:
		$$E^{1}(j) = \begin{cases} \mathcal G &  j\leq 1 \\
		\{\mathbf{w}_1\} &  j =2 \\
		\emptyset  &  j > 2 \\
	\end{cases}, \, 
	E^{2}(j) = \begin{cases} \mathcal G &  j\leq 1 \\
		\{\mathbf{w}_2\} &  j =2 \\
		\emptyset  &  j > 2 \\
	\end{cases},  \, 
	E^{3}(j) = \begin{cases} \mathcal G &  j\leq 1 \\
		\{\mathbf{w}_3\} &  j =2 \\
		\emptyset  &  j > 2 \\
	\end{cases}.$$
        These filtrations define a tropical toric reflexive sheaf, but
        not a tropical toric vector bundle.  If they defined a
        tropical toric vector bundle, there would be a basis for
        $\underline{\mathcal M}$ for which $E^1(2)=\{\mathbf{w}_1 \}$,
        $E^2(2) = \{\mathbf{w}_2 \}$, and $E^3(2) =\{\mathbf{w}_3 \}$
        are all joins of elements of this basis.  But this would mean
        that the three elements $\{ \mathbf{w}_1, \mathbf{w}_2,
        \mathbf{w}_3 \}$ are part of a basis for the rank-$2$ matroid
        $\underline{\mathcal M}$.
        
	\item  \label{item:tvb4} The tropical reflexive sheaf in Example \ref{e:reflexive} \eqref{item:trs6} is a tropical toric vector bundle as Lemma \ref{l:compatibilitymeansbasis} is satisfied with the basis $\{\mathbf{w}_1,\mathbf{w}_3,\mathbf{w}_4\}$ for both cones.
		\end{enumerate}
\end{example}

\begin{remark}\label{dist_latt_not_sat}
  For  toric vector bundles,
Klyachko \cite{Kly}*{Remark 2.2.2} shows that 
  Condition~\ref{cond:basis} is also
equivalent to the fact that the subspaces $E^i(j)$ generate a {\em
	distributative lattice}.  Recall that a lattice $L$ is {\em
	distributative} if $x \meet (y \join z) = (x \meet y) \join (x
\meet z)$ for all $x,y,z \in L$, or equivalently $x \join (y \meet
z) = (x \join y) \meet (x \join z)$ for all $x,y,z \in L$.  This is no longer equivalent in this setting.

To see this, let $\mathcal M$ be the Vamos matroid with the trivial valuated matroid structure.  This is the rank-$4$ matroid with ground set $\mathcal G = \{1,\dots,8\}$, and all subsets of size at most $4$ independent except
$$\{ \{1,2,3,4\}, \{1,4,5,6\}, \{ 1,4,7,8\}, \{2,3,5,6\},
\{2,3,7,8\} \}.$$ Let $\mathcal E$ be the tropical toric reflexive
sheaf on $\trop(\mathbb A^2)$ with
$$E^1(j) = \begin{cases} \mathcal G & j \leq 0 \\ \{1,2\} & j=1 \\ \emptyset & j >1 \\ \end{cases}, \, \, \, \, \,
E^2(j) = \begin{cases} \mathcal G & j \leq 0 \\ \{3,4\} & j=1 \\ \emptyset & j >1 \\ \end{cases}.$$

The sublattice of $\mathcal L(\underline{ \mathcal M})$ generated by $\mathcal L(\mathcal E)$ is shown in Figure~\ref{f:Vamoseg}.
	\begin{figure}
	\centering
	\begin{minipage}{.5\textwidth}
		\centering
		\begin{tikzpicture}
			[scale=.5,auto=left]
		\node (n0) at (3,0) {$\emptyset$};
		\node (n1) at (1,2)  {12};
		\node (n4) at (5,2)  {34};
		\node (n7) at (3,4)  {1234};
		\node  (n12) at (3, 6) {$\mathcal G$};
		\foreach \from/ \to in {n4/n0,n1/n0,n7/n1, n12/n7, n7/n4}
		\draw (\from) -- (\to);
		\end{tikzpicture}
		\captionof{figure}{}
		\label{f:Vamoseg}
	\end{minipage}%
	\begin{minipage}{.4\textwidth}
		\centering
		\begin{tikzpicture}
				[scale=.5,auto=left]
			\node (n0) at (3,0) {$\emptyset$};
			\node (n1) at (-1,2)  {1};
			\node (n2) at (1,2)  {2};
			\node (n3) at (3,2) {3};
			\node (n4) at (5,2)  {4};
			\node (n5) at (7,2)  {5};
			\node (n6) at (-2,4) {124};
			\node (n7) at (0,4)  {13};
			\node (n8) at (2,4) {15};
			\node (n9) at (4,4)  {235};
			\node (n10) at (6,4)  {34};
			\node (n11) at (8,4)  {45};
			\node  (n12) at (3, 6) {$\mathcal G$};
			\foreach \from/ \to in {n5/n0,n4/n0,n3/n0,n2/n0,n1/n0,n6/n1,n6/n2,n6/n4,n7/n1,n7/n3,n8/n1,n8/n5,n9/n2,n9/n3,n9/n5, n10/n3, n10/n4, n11/n4, n11/n5, n12/n6, n12/n7, n12/n8, n12/n9, n12/n10, n12/n11}
			\draw (\from) -- (\to);
		\end{tikzpicture}
		\captionof{figure}{}
		\label{eg1}
	\end{minipage}
\end{figure}

Note that this lattice is distributative.  However there is no basis
$B$ for $\underline{\mathcal M}$ with all $E^i(j)$ a join of elements
of elements of $B$, so $\mathcal E$ is not a tropical toric vector bundle.

This contrasts with the case of usual toric vector bundles, where
every toric reflexive sheaf on a surface is a toric vector bundle (see \cite[Example 2.3 (4)]{Kly}).  In
Corollary~\ref{c:surfacecompatible} we recover a tropical version of that result by placing an additional assumption on the reflexive sheaf.
\end{remark}

A toric vector bundle on $X_{\Sigma}$ gives rise to a tropical toric vector bundle on $\trop(X_{\Sigma})$  as follows.

\begin{definition}  \label{pd:tropE}
Let $X_{\Sigma}$ be a toric variety and let $\mathcal E$ be a toric
vector bundle of rank $r$ on $X_{\Sigma}$.  Let $\underline{\mathcal
	M}$ be a realizable matroid for $\mathcal E$ in the sense of
Definition~\ref{d:matroidoftvb}, with realization $\mathcal G
\subseteq K^r$.  The tropicalization $\trop(\mathcal E)$ of the pair
$(\mathcal E, \mathcal G)$ is the tropical toric vector bundle
corresponding to the valuated matroid $\mathcal M$ induced by
$\mathcal G$.  This has underlying matroid $\underline{\mathcal M}$.  The
subspaces $E^i(j)$ of the filtration defining $\mathcal E$ correspond
to flats of $\underline{\mathcal M}$ by construction, and the compatibility
condition \eqref{eqtn:tropicalcompatibility} is satisfied for
$\trop(\mathcal E)$ because it is satisfied for $\mathcal E$. Tropical toric vector bundles of the form $\trop(\mathcal E)$ are called realizable.

The module \eqref{eqtn:CoxModuleTropical} is the tropicalization of the Cox module for $\mathcal E$ as given in Proposition~\ref{p:CoxModuleTVB}.
\end{definition}

\begin{example} \label{e:tropicalizationegs}	
	\begin{enumerate}
		\item \label{item:tropicalizationtangentsheaf} The tropical tangent sheaf in Example \ref{e:reflexive} \eqref{item:trs2} is the tropicalization of the tangent sheaf on the toric variety \(X_{\Sigma}\). See \cite[Example 2.3 (5)]{Kly} for the filtration of the tangent sheaf.

	        \item \label{item:tropeg2}
 Consider the representable matroid \(\underline{\mathcal M}\) of rank
 \(3\) on the ground set $$\mathcal G=\{\mathbf{e}_1, \mathbf{e}_2,
 \mathbf{e}_3, \mathbf{e}_1-\mathbf{e}_2,
 \mathbf{e}_2-\mathbf{e}_3\} \subseteq K^3.$$ The lattice of flats of
 \(\underline{\mathcal M}\) is given in Figure \ref{eg1}.

	Consider the following filtrations of flats.
	
	{\scriptsize
		$$E^1(j) = \begin{cases} \mathcal G &  j\leq -1 \\
			\{\mathbf{e}_1, \mathbf{e}_2, \mathbf{e}_1-\mathbf{e}_2\} &   j = 0 \\
			\{\mathbf{e}_1\} & 0 < j \leq 4\\
			\emptyset  &  j > 4 \\
		\end{cases}, \, 
		E^2(j) = \begin{cases} \mathcal G &  j\leq -2 \\
			\{\mathbf{e}_2, \mathbf{e}_3, \mathbf{e}_2-\mathbf{e}_3\} &  -2 < j \leq 0 \\
			\{\mathbf{e}_3\} & 0 < j \leq 3\\
			\emptyset  &  j > 3 \\
		\end{cases},  \, 
		E^0(j) = \begin{cases} \mathcal G &  j\leq -1 \\
			\{ \mathbf{e}_1-\mathbf{e}_2, \mathbf{e}_2-\mathbf{e}_3\} &  -1 < j \leq 2 \\
			\{\mathbf{e}_2-\mathbf{e}_3\} & j = 3\\
			\emptyset  &  j > 3 \\
		\end{cases}.$$}

	Then \(\left( \mathcal{M}, \mathcal{G}, \{E^i(j)\}\right) \)
        defines a tropical toric vector bundle on
        $\trop(\mathbb{P}^2)$. This tropical sheaf is obtained by
        tropicalizing the toric vector bundle in \cite[Example
          4.2]{DJS}.
		\end{enumerate}
\end{example}

\begin{remark}
  Note that the tropicalization of a toric vector bundle depends not only on the choice of matroid $\underline{\mathcal M}$ for $\mathcal E$, but also on the choice of realization of that matroid, as that affects the valuated matroid.
\end{remark}

\begin{example} \label{e:trickymatroidrevisited}
    Example~\ref{e:trickymatroid} gives three different choices of
  matroid for a rank-$3$ toric vector bundle $\mathcal E$ on $\mathbb P^2$.  This
  gives rise to three different choices of tropicalization of $\mathcal E$.  
  We emphasize that  the choice of matroid
is part of the data of the tropical vector bundle.
\end{example}

\begin{remark} \label{r:tropicalizepresentation}
  Given a choice of a realization $\mathcal G$ for a matroid for a
  rank-$r$ toric vector bundle $\mathcal E$ on a toric variety
  $X_{\Sigma}$, following \cite{DJS}*{Remark 3.6}
  we can construct a surjection
  \begin{equation} \label{eqtn:tropicalpresentation} \oplus_{\mathbf{w} \in \mathcal G} \mathcal O(D_{\mathbf{w}})
  \rightarrow \mathcal E \end{equation} as follows.  Recall that
  $\oplus_{\mathbf{w} \in \mathcal G} \mathcal O(D_{\mathbf{w}})$ is a
  toric vector bundle, with corresponding filtration on $K^{|\mathcal
    G|}$ the direct sum of the filtrations corresponding to each
  $\mathcal O(D_{\mathbf{w}})$ as given in Example~\ref{e:reflexive}
  \eqref{item:trs1}; the vector $\mathbf{a}$ there records the
    coefficients of torus invariant divisors in $D_{\mathbf{w}}$.  Consider the linear map $K^{|\mathcal G|}
    \rightarrow K^r$ by $\mathbf{e}_i \mapsto \mathbf{w}_i$.  The
definition of $D_{\mathbf{w}}$ implies that this map respect the
filtrations on $K^{|\mathcal G|}$ and $K^r$, so by Klyachko's
equivalence of categories induces a morphism $\oplus_{\mathbf{w} \in
  \mathcal G} \mathcal O(D_{\mathbf{w}}) \rightarrow \mathcal E$ as
required.  Since each $E^i(j)$ is a flat of the matroid, the linear
map is surjective on the filtration, so the morphism is surjective.
The tropicalization of $\mathcal E$ can thus be thought of as
tropicalizing the presentation~\ref{eqtn:tropicalpresentation}.
\end{remark}

\begin{remark}
  The semimodule corresponding to $\trop(\mathcal E)$ is
\begin{equation} \label{eqtn:tropicalM}
	\bigoplus_{i =1}^m S(\mathbf{d}_{\mathbf{w}_i})/ \mathcal B(\trop(R)),\end{equation}
where $\mathcal B()$ denotes the bend locus of \cite{Giansiracusa2},
$R$ is as in Proposition~\ref{p:CoxModuleTVB}, and $\trop(R)$ is the
submodule obtained by taking $\sum_{k=1}^m c_k \mathbf{x}^{\mathbf{u}_k}$ to
$\tplus_{k=1}^m \val(c_k)\ttimes \mathbf{x}^{\mathbf{u}_k} \in S^m$.
\end{remark}

\begin{example}
   For the tropicalization of the tangent bundle to $\mathbb P^2$, as in Example~\ref{e:reflexive} \eqref{item:trs2}, 
the semimodule defining $\trop(\mathscr{T})$ over $\Cox(\trop(\mathbb
P^2)) = \Rbar[x,y,z]$ is
$$(\Rbar[x,y,z](1,0,0) \oplus \Rbar[x,y,z](0,1,0)
\oplus\Rbar[x,y,z](0,1,0))/\mathcal B (\langle
x\mathbf{e}_1+y\mathbf{e}_2+z\mathbf{e}_3 \rangle) $$ over
$\Cox(\trop(\mathbb P^2)) = \Rbar[x,y,z]$.  Note that on the affine
chart $x \neq 0$, this becomes
$$ x\Rbar[y/x,z/x] \oplus y\Rbar[y/x,z/x] \oplus z\Rbar[y/x,z/x]/\mathcal B \langle   \mathbf{e}_1 + y/x \mathbf{e}_2+z/x\mathbf{e}_3 \rangle. $$
Note that this tropicalization, and presentation of the Cox module, comes from tropicalizing the presentation for $\trop(\mathscr{T})$ coming from the Euler sequence as discussed in Remark~\ref{r:tropicalizepresentation}.
\end{example}

However, there are tropical toric vector bundles that are not the tropicalization of a toric vector bundle on some toric variety, as the following example shows.

\begin{example} \label{e:Fanoeg}
  Let $\mathcal M$ be the trivial
  valuated matroid structure on the Fano matroid on $\mathcal G = \{1,\dots,7\}$.  The Fano matroid is the rank-$3$ matroid whose dependent sets
of size three are the lines in Figure~\ref{f:Fano}.

\begin{figure}
	\centering
	\begin{minipage}{.4\textwidth}
		\centering
		\begin{tikzpicture}
		[scale=1.8,auto=left]
			\draw (0,0) circle (0.5);
			\draw (90:1) -- (-30:1)--(210:1)--cycle;
			
			\draw (90:1)--(0,0);
			\draw (210:1)--(0,0);
			\draw (-30:1)--(0,0);
			
			\draw (30:0.5)--(0,0);
			\draw (150:0.5)--(0,0);
			\draw (270:0.5)--(0,0);
			
			\fill (-0.866,-0.5) circle (1.5pt);
			\node at (-0.866,-0.65) { \( 1 \)};
			\fill (0.866,-0.5) circle (1.5pt);
			\node at (0.866,-0.65) { \( 3 \)};
			\fill (0,-0.5) circle (1.5pt);
			\node at (0,-0.65) { \( 2 \)};
			\fill (0,1) circle (1.5pt);
			\node at (0,1.2) { \( 7\)};
			\fill (0,0) circle (1.5pt);
			\node at (.1,-0.2) { \( 5 \)};
			\fill (0.433,0.25) circle (1.5pt);
			\node at (0.53,0.3) { \( 6 \)};
			
			\fill (-0.433,0.25) circle (1.5pt);
			\node at (-0.53,0.3) { \( 4 \)};
			
		\end{tikzpicture}
		\captionof{figure}{Fano Matroid}
		\label{f:Fano}
	\end{minipage}%
	\begin{minipage}{.4\textwidth}
		\centering
			\begin{tikzpicture}[scale=1.94]
			

			\coordinate (B1) at (0,0);
			\coordinate (B2) at (1.1,0);
			\coordinate (B3) at (0,1);
			\coordinate (B4) at (-1,1);
			\coordinate (B5) at (0,-1);
			\coordinate (B6) at (1,1);
			\coordinate (B7) at (-1,-1);
			\coordinate (B8) at (1,-1);
			\coordinate (B9) at (-1.1,0);

			\definecolor{c1}{RGB}{0,129,188}
			\definecolor{c2}{RGB}{252,177,49}
			\definecolor{c3}{RGB}{35,34,35}

			\draw[c3,->] (B1) -- (B2);
			\draw[c3,->] (B1) -- (B3);
			\draw[c3,->] (B1) -- (B4);
			\draw[c3,->] (B1) -- (B5);
			\draw[c3,->] (B1) -- (B9);
			\draw[c3,->] (B1) -- (B7);
			\draw[c3,->] (B1) -- (B6);
			
			\node at (-1.2,0) { \textbf{\(5\)}};
			
			\node at (1.1,1.1) { \( 2 \)};
			\node at (.5,0) {  \(\bullet \)};
			\node at (.5,-.1) {  \((1,0)\)};
			\node at (-.5,0) {  \(\bullet \)};
			\node at (-.5,-.13) {  \((-1,0)\)};
			\node at (0,-.5) {  \(\bullet \)};
			\node at (.3,-.5) {  \((0,-1)\)};
			\node at (0,.5) {  \(\bullet \)};
			\node at (.22,.5) {  \((0,1)\)};
			\node at (.5,.5) {  \(\bullet \)};
			\node at (.8,.5) {  \((1,1)\)};
			
			\node at (-.5,.5) {  \(\bullet \)};
			\node at (-.9,.5) {  \((-1,1)\)};
			\node at (-.5,-.5) {  \(\bullet \)};
			\node at (-.9,-.5) {  \((-1,-1)\)};
			
			\node at (-1.1,-1.1) { \(6\)};

			\node at (1.2,0) { \(1 \)};
			\node at (0,1.2) { \(3 \)};
			\node at (0,-1.2) { \(7 \)};
			\node at (-1,1.2) { \(4\)};
		\end{tikzpicture}	
		\captionof{figure}{Fan for \(X_{\Sigma}\)}
		\label{f:Fano1}
	\end{minipage}
\end{figure}

Let $X_{\Sigma}$ be the smooth toric
surface whose rays are the columns of the matrix
$$\left( \begin{array}{rrrrrrr}  1 & 1 & 0 & -1 & -1 & -1 & 0 \\
	0  & 1 & 1 & 1 & 0 & -1 & -1 \\\end{array} \right).$$
The corresponding fan is shown on the right of Figure~\ref{f:Fano1}.

We define a tropical toric vector bundle by setting

{\scriptsize
	$$E^1(j) = \begin{cases} \mathcal G &  j\leq 0 \\
		\{1, 2, 3\} &   j=1 \\
		\{1\} & j =2\\
		\emptyset  &  j > 2 \\
	\end{cases}, \, 
	E^2(j) = \begin{cases} \mathcal G &  j\leq 0 \\
		\{2, 4, 6\} &   j=1 \\
		\{2\} & j =2\\
		\emptyset  &  j > 2 \\
	\end{cases},  \, 
	E^3(j) = \begin{cases} \mathcal G &  j\leq 0 \\
		\{3, 4, 5\} &   j=1 \\
		\{3\} & j =2\\
		\emptyset  &  j > 2 \\
	\end{cases}, \,
	E^4(j) = \begin{cases} \mathcal G &  j\leq 0 \\
		\{1, 4, 7\} &   j=1 \\
		\{4\} & j =2\\
		\emptyset  &  j > 2 \\
	\end{cases}, \, $$}
{\scriptsize
	$$
	E^5(j) = \begin{cases} \mathcal G &  j\leq 0 \\
		\{2, 5, 7\} &   j=1 \\
		\{5\} & j =2\\
		\emptyset  &  j > 2 \\
	\end{cases},  \, 
	E^6(j) = \begin{cases} \mathcal G &  j\leq 0 \\
		\{1, 5, 6\} &   j=1 \\
		\{6\} & j =2\\
		\emptyset  &  j > 2 \\
	\end{cases}, \,
	E^7(j) = \begin{cases} \mathcal G &  j\leq 0 \\
		\{3, 6, 7\} &   j=1 \\
		\{7\} & j =2\\
		\emptyset  &  j > 2 \\
	\end{cases}.$$}

Numbering the rays of the fan $\Sigma$ from $1$ to $7$, using that $X_{\Sigma}$ is smooth, we can check the compatibility condition \eqref{cond:basis} as follows:
\begin{equation*}
	\begin{split}
		&	\text{For } E^1(j) \text{ and } E^2(j) \text{ the basis is given by } B=\{1,2,4\}\\
		&	\text{For } E^2(j) \text{ and } E^3(j) \text{ the basis is given by } B=\{2,3, 4\}\\
		&	\text{For } E^3(j) \text{ and } E^4(j) \text{ the basis is given by } B=\{1,4,5\}\\
		&	\text{For } E^4(j) \text{ and } E^5(j) \text{ the basis is given by } B=\{4,5,7\}\\
		&	\text{For } E^5(j) \text{ and } E^6(j) \text{ the basis is given by } B=\{2,5,6\}\\
		&	\text{For } E^6(j) \text{ and } E^7(j) \text{ the basis is given by } B=\{1,6,7\}\\
		&	\text{For } E^7(j) \text{ and } E^1(j) \text{ the basis is given by } B=\{1,3,7\}.
	\end{split}
\end{equation*}
We claim that $\mathcal E$ is not realizable over any field of
characteristic not equal to $2$.  This is immediate from the fact that
$\underline{\mathcal M}$ is not realizable over any field of
characteristic not equal to $2$, since the Fano matroid is the
only rank-$3$ matroid on $7$ elements containing the given flats of
rank one and two in our choice of $E^i(j)$, so this failure of
realizability cannot be fixed by choosing another matroid containing
$\mathcal L(\mathcal E)$.
\end{example}

\subsection{Partially modular tropical toric vector bundles}

We now introduce a subclass of tropical toric vector bundles that behave more like the tropicalizations of toric vector bundles.

\begin{definition} \label{d:partiallymodular}
A tropical toric reflexive sheaf or tropical toric vector bundle is {\em partially modular} if the
following condition holds for all $1 \leq i,j \leq s$, and  $k,l \in \mathbb Z$:
\begin{equation} \label{eqtn:partiallymodular}
	\begin{split}
		\rk((E^i(k-1) & \meet E^j(l)) \join (E^i(k) \meet E^j(l-1))) \\ & = \rk(E^i(k-1) \meet E^j(l)) + \rk(E^i(k) \meet E^j(l-1)) - \rk(E^i(k) \meet E^j(l)).
	\end{split}
	\end{equation}
 Condition \eqref{eqtn:partiallymodular} requires that $(E^i(k-1)
 \meet E^j(l), E^i(k) \meet E^j(l-1))$ is a modular pair for all $1
 \leq i,j \leq s$ and $k,l \in \mathbb Z$; see \S
 \ref{ss:latticesmodularity}.
\end{definition}

The following lemma is then a consequence of the fact that the lattice
of all subspaces of a vector space is modular.

\begin{lemma} \label{l:partiallymodular}
Realizable tropical toric vector bundles are partially modular.
\end{lemma}

\begin{proof}
For any matroid constructed from a toric vector bundle the meet of two
flats of the form $E^i(k)$ is their intersection, and the join is
their sum.  The rank of a flat is the dimension of the corresponding subspace.  We then have
\begin{align*}& \dim  ((E^i(k-1)  \cap E^j(l)) + (E^i(k) \cap  E^j(l-1))) \\ &
	= \dim(E^i(k-1) \cap E^j(l)) + \dim(E^i(k) \cap E^j(l-1)) - 
	\dim((E^i(k-1) \cap E^j(l)) \cap (E^i(k) \cap  E^j(l-1))) \\
	& = \dim(E^i(k-1) \cap E^j(l)) + \dim(E^i(k) \cap E^j(l-)) - 
	\dim(E^i(k) \cap E^j(l)).
\end{align*}  
\end{proof}

\begin{lemma} \label{l:partiallymodularmeansbasis}
Let $\mathcal E$ be a partially modular tropical toric reflexive sheaf on a
tropical toric variety.  Fix $1 \leq i < j \leq s$.  Then there is a basis $B$ of
the matroid $\underline{\mathcal M}(\mathcal E)$ of $\mathcal E$
with each $E^i(k)$ or $E^j(l)$ for $k,l \in \mathbb Z$ the join 
of a subset of $B$.
\end{lemma}

\begin{proof}
Write $a_0 \leq \dots \leq a_q$ for the distinct $E^i(k)$, where $a_0
= \emptyset$, and $a_q = \mathcal G$, and $b_0 \leq \dots \leq b_p$ for
the distinct $E^j(l)$, where again $b_0=\emptyset$ and $b_p = \mathcal G$.  For $1 \leq i \leq q$ and $1 \leq j \leq p$ we have
$$(a_{i-1} \meet b_j) \join (a_i \meet b_{j-1}) \leq a_i \meet b_j.$$
If the inequality is strict, then since lattice of flats of matroids
are relatively complemented (see \cite[Proposition 3.4.4]{White})
  there is $w_{ij} \in \mathcal L(\underline{\mathcal{M}})$ with
$$w_{ij} \meet ((a_{i-1} \meet b_j) \join (a_i \meet b_{j-1})) =
\emptyset, \text{ and } w_{ij} \join ((a_{i-1} \meet b_j) \join (a_i
\meet b_{j-1})) = a_i \meet b_j,$$
and $\rk(w_{ij}) = \rk(a_i \meet
b_j) - \rk((a_{i-1} \meet b_j) \join (a_i \meet b_{j-1}))$.  If the inequality is not strict, we set $w_{ij}=\emptyset$.

We now show that for any $1 \leq i \leq q$ and $1 \leq j \leq p$ we have
\begin{equation} \label{eqtn:spans} a_i \meet b_j = \join_{k \leq i, l \leq j} w_{kl}, \text{ and }
	\rk(a_i \meet b_j) = \sum_{k \leq i, l \leq j} \rk(w_{kl}).\end{equation}
The proof is by induction on $i+j$.  The base case is when $i=j=1$, when  $w_{11} = a_1 \meet b_1$, so the claim is true.   Suppose the claim is true for smaller $i+j$.  We then have
\begin{align*} a_i \meet b_j & = 
	w_{ij} \join ((a_{i-1} \meet b_j) \join (a_i \meet b_{j-1}))\\
	& = w_{ij} \join (\join_{k \leq i-1, l \leq j} w_{kl}) \join (\join_{k \leq i, l \leq j-1} w_{kl}) \\
	& = \join_{k \leq i, l \leq j} w_{kl}.\\
\end{align*}  
In addition,
\begin{align*} \rk(a_i \meet b_j) & = \rk((a_{i-1} \meet b_j) \join (a_i \meet b_{j-1})) + \rk(w_{ij})\\
	& = \rk(a_{i-1} \meet b_j) + \rk(a_i \meet b_{j-1}) - \rk(a_{i-1} \meet b_{j-1})+ \rk(w_{ij}) \\
	& = \sum_{k \leq i-1, l \leq j} \rk(w_{kl})  + \sum_{k \leq i, l \leq j-1} \rk(w_{kl}) - \sum_{k \leq i-1, l \leq j-1} \rk(w_{kl}) + \rk(w_{ij}) \\
	& = \sum_{k \leq i, l \leq j} \rk(w_{kl}).\\
\end{align*}  

The case $i=q, j=p$ of \eqref{eqtn:spans} states that $\mathcal G$ is
the join of all the $w_{kl}$, and the sum of their ranks is the rank
of $\underline{\mathcal M}(\mathcal E)$.  For each $1 \leq i \leq q$
and $1 \leq j \leq p$, choose a maximal independent subset of
$\mathcal G$ whose join is $w_{ij}$, and let $B$ be the union of all
such subsets.  Then $|B| = \sum \rk(w_{ij}) = \rk \underline{\mathcal
  M}(\mathcal E)$, and the join of the elements of $B$ is $\mathcal
G$, so $\rk(B) = \rk \underline{\mathcal M}(\mathcal E)$.  Thus $B$ is a
basis for $\underline{\mathcal M}(\mathcal E)$.  Since each $a_i$
equals $a_i \meet b_p$, and $b_j = a_q \meet b_j$, the lemma follows.
\end{proof}

One consequence of Lemma~\ref{l:partiallymodularmeansbasis} is the following tropical version of \cite{Kly}*{Example 2.3(4)}.

\begin{corollary} \label{c:surfacecompatible}
A partially modular tropical toric reflexive sheaf $\mathcal E$ on the tropicalization of a smooth toric surface $\trop(X_{\Sigma})$ is a tropical toric vector bundle.
\end{corollary}

\begin{proof}
From Lemma~\ref{l:partiallymodularmeansbasis} we know that for all
cones $\sigma \in \Sigma$ there is a basis $B_{\sigma}$ of
$\underline{\mathcal M}(\mathcal E)$ with every $E^i(j)$ for all $\rho_i \in
\sigma$  and $j \in \mathbb Z$ the join
of a subset of $B_{\sigma}$.  Then, Lemma~\ref{l:compatibilitymeansbasis} implies that condition \eqref{eqtn:tropicalcompatibility} holds, so $\mathcal E$ is a tropical toric vector bundle.
\end{proof}

See Example \ref{vbonP1split} for examples and non-examples of
partially modular tropical reflexive sheaves.

\subsection{The total space of a tropical toric vector bundle}

\label{ss:totalspace}

The total space of a rank-$r$ tropical toric vector bundle on
$\trop(X_{\Sigma})$ is a subset of $\trop(X_{\Sigma}) \times \Rbar^{|\mathcal G|}$, which we now define.

Let $\mathbf{z} \in \Rbar^s$ be the Cox coordinates of a point in
$\trop(X_{\Sigma})$.  The fiber over $\mathbf{z}$ of a tropical toric
vector bundle $\mathcal E = (\mathcal M,\mathcal G, \{ E^i(j) \})$ on
$\trop(X_{\Sigma})$ is the tropical linear space in $\Rbar^{|\mathcal G|}$ constructed as follows.  Set $m=|\mathcal G|$.
Recall from Proposition~\ref{p:CoxModuleTVB} that
for a vector $\mathbf{c} = (c_1,\dots,c_m)$ and $1 \leq i \leq m$ we set $\mathbf{u}_i =
\mathbf{d}_{\mathbf{w}_i} - \min(\mathbf{d}_{\mathbf{w}_l} : c_l \neq
\infty)$.
We define $\mathbf{c}(\mathbf{z}) \in \Rbar^{m}$ to be
$$\mathbf{c}(\mathbf{z}) := \sum_{l=1}^m c_l \mathbf{z}^{\mathbf{u}_l}
\mathbf{e}_l,$$ where all operations are tropical, so $c_l
\mathbf{z}^{\mathbf{u}_l}$ is $c_l + \mathbf{z} \cdot \mathbf{u}_l$ in
usual arithmetic.  Note that $\mathbf{z}^{\mathbf{u}_l} = \infty$ if $z_k = \infty$ for some $k$ with $(\mathbf{u}_{l})_k>0$.

\begin{lemma} \label{l:fiberwelldefined}
For any point $\mathbf{z} \in \trop(X_{\Sigma})$ the set $\mathcal C_{\mathbf{z}}$ of elements of minimal nonempty support from $ \{
\mathbf{c}(\mathbf{z}) : \mathbf{c} \text{ is a circuit of } \mathcal M
\}$ is the collection of circuits of a valuated matroid $\mathcal
M_{\mathbf{z}}$.
\end{lemma}

\begin{proof}
We show that the valuated circuit axioms hold for $\mathcal
C_{\mathbf{z}}$ (see \S\ref{ss:valuatedmatroids}).  We have $(\infty,\dots,\infty) \not \in \mathcal
C_{\mathbf{z}}$, and the elements of $\mathcal C_{\mathbf{z}}$ have
minimal support by construction.  The fact that
$\mu \ttimes
\mathbf{c}(\mathbf{z}) \in \mathcal C_{\mathbf{z}}$ for $\mu \in \mathbb R$ 
follows from the
analogous axiom for $\mathcal M$, as $\mu \ttimes
\mathbf{c}(\mathbf{z}) = (\mu \ttimes
\mathbf{c}) (\mathbf{z})$.  Suppose now that
$\mathbf{c}_1(\mathbf{z}), \mathbf{c}_2(\mathbf{z}) \in \mathcal
C_{\mathbf{z}}$ with 
$\mathbf{c}_1(\mathbf{z})_i = \mathbf{c}_2(\mathbf{z})_i < \infty$ and
$\mathbf{c}_1(\mathbf{z})_j < \mathbf{c}_2(\mathbf{z})_j$.

Write $\mathbf{u}_i^{\alpha}$ for the vector $\mathbf{u}_i$ corresponding to
$\mathbf{c}_{\alpha}$ for $\alpha =1,2$.   Since $\mathbf{c}_{\alpha}(\mathbf{z})_i<\infty$ for $\alpha=1,2$, we have we
   have $(\mathbf{u}_i^{\alpha})_k = 0$ whenever $z_k = \infty$,  so $(\mathbf{d}_{\mathbf{w}_i})_k =
   \min((\mathbf{d}_{\mathbf{w}_l})_k : c_{\alpha l} \neq \infty)$ whenever
   $z_k = \infty$.  Thus $ \min((\mathbf{d}_{\mathbf{w}_l}) : c_{1l}
   \neq \infty)_k = \min((\mathbf{d}_{\mathbf{w}_l}) : c_{2l} \neq
   \infty)_k$ for all such $k$.  For any $\mathbf{c} \in \Rbar^m$, set
   $$\lambda_{\mathbf{c}} = \ttimes_{k: z_k \neq \infty}
   z_k^{\min(\mathbf{d}_{\mathbf{w}_l} : c_l \neq \infty)_k}.$$

  Since $\mathbf{c}_1(\mathbf{z})_i =
c_{1i}\mathbf{z}^{\mathbf{u}^1_i} = \mathbf{c}_2(\mathbf{z})_i =
c_{2i}\mathbf{z}^{\mathbf{u}^2_i} < \infty$, we have $\lambda_{\mathbf{c}_{\alpha}} \ttimes \mathbf z^{\mathbf{u}_i^{\alpha}} = \mathbf{z}^{\mathbf{d}_{\mathbf{w}_i}}$ for $\alpha = 1,2$, and thus 
 $c_{1i}
\ttimes \lambda_{\mathbf{c}_2} = c_{2i} \ttimes \lambda_{\mathbf{c}_1}$.  Similarly,
either $c_{1j} \ttimes \lambda_{\mathbf{c}_2} < c_{2j} \ttimes  \lambda_{\mathbf{c}_1}$
or $\mathbf{c}_2(\mathbf{z})_j = \infty$.

By the circuit elimination axiom for $\mathcal M$ applied to the
circuits $\lambda_{\mathbf{c}_2} \ttimes \mathbf{c}_1$ and
$\lambda_{\mathbf{c}_1} \ttimes \mathbf{c}_2$, there is  a circuit
$\mathbf{c}_3$ of $\mathcal M$ with $c_{3i} = \infty$,
$\mathbf{c}_{3j} = \lambda_{\mathbf{c}_2} \ttimes c_{1j}$, and
$\mathbf{c}_3 \geq \lambda_{\mathbf{c}_2} \ttimes \mathbf{c}_1 \tplus
\lambda_{\mathbf{c}_1} \ttimes \mathbf{c}_2$ coordinatewise.  We may
assume that $\mathbf{c}_3$ is chosen to have
$\supp(\mathbf{c}_3(\mathbf{z}))$ minimal among all such choices for
$\mathbf{c}_3$.  Set $\mathbf{c}_4 = (\lambda_{\mathbf{c}_3} -
\lambda_{\mathbf{c}_1} - \lambda_{\mathbf{c}_2}) \ttimes
\mathbf{c}_3$.  We now show that, if its support is minimal,
$\mathbf{c}_4(\mathbf{z})$ is the required circuit.

We have $c_{4i} = \infty$, and since $\supp(\mathbf{c}_4) =
\supp(\mathbf{c}_3)$ we have $\lambda_{\mathbf{c}_4} =
\lambda_{\mathbf{c}_3}$.  In addition, for any $1 \leq t \leq m$,
\begin{equation} \label{eqtn:c4jlarger}
\begin{split} c_{4t} & = \lambda_{\mathbf{c}_3} - \lambda_{\mathbf{c}_1} -
  \lambda_{\mathbf{c}_2} + c_{3t} \\
  & \geq  \lambda_{\mathbf{c}_3} - \lambda_{\mathbf{c}_1} -
  \lambda_{\mathbf{c}_2} + \min(\lambda_{\mathbf{c}_2} + c_{1t},
    \lambda_{\mathbf{c}_1} + c_{2t}) \\
      & = \lambda_{\mathbf{c}_4} + \min(c_{1t}-\lambda_{\mathbf{c}_1},
    c_{2t}-\lambda_{\mathbf{c}_2}),\\
\end{split}
\end{equation}
   with equality if $t = j$.  Thus $c_{4t} - \lambda_{\mathbf{c}_4}
   \geq \min(c_{1t}-\lambda_{\mathbf{c}_1},
   c_{2t}-\lambda_{\mathbf{c}_2})$, with equality if $t=j$.

   We now show that $\mathbf{c}_4(\mathbf{z}) \geq \mathbf{c}_1(\mathbf{z}) \tplus \mathbf{c}_2(\mathbf{z})$. 
Consider $t \neq i$ with $\min(
\mathbf{c}_1(\mathbf{z})_t, \mathbf{c}_2(\mathbf{z})_t) < \infty$.
Write $\mathbf{c}_{\alpha}(\mathbf{z})_t$ for this minimum, where $\alpha \in
\{1,2\}$.  Since $\mathbf{c}_{\alpha}(\mathbf{z})_t < \infty$, it follows
that
\begin{equation} \label{eqtn:dwi}
  (\mathbf{d}_{\mathbf{w}_t})_k =
\min((\mathbf{d}_{\mathbf{w}_l})_k : c_{\alpha l} \neq \infty) =
(\mathbf{d}_{\mathbf{w}_i})_k
\end{equation}for all $k$ with $z_k=\infty$.  As
$(\mathbf{d}_{\mathbf{w}_i})_k$ also equals
$\min((\mathbf{d}_{\mathbf{w}_l})_k : c_{\beta l} \neq \infty)$ for $\beta =
\{1,2 \} \setminus \{\alpha \}$ and such $k$, adding $\ttimes_{k : z_k \neq
  \infty} z_k^{(\mathbf{d}_{\mathbf{w}_t})_k}$ to both sides of \eqref{eqtn:c4jlarger} yields
  \begin{equation} \label{eqtn:c4j}
    c_{4t} + \ttimes_{k : z_k \neq \infty}
    z_k^{(\mathbf{d}_{\mathbf{w}_t})_k} - \lambda_{\mathbf{c}_4} \geq
    \min( \mathbf{c}_1(\mathbf{z})_t, \mathbf{c}_2(\mathbf{z})_t)
  \end{equation}for $t$ where the right-hand side is
  not equal to $\infty$.  We have  equality when $t=j$.
  When
  $c_{4t}<\infty$, since $t \in \supp(\mathbf{c}_4) \subsetneq
  \supp(\mathbf{c}_1) \cup \supp(\mathbf{c}_2)$, the equality
  \eqref{eqtn:dwi} implies that $\min(\mathbf{d}_{\mathbf{w}_l} :
  c_{4l} \neq \infty)_k = (\mathbf{d}_{\mathbf{w}_t})_k$ for all $k$ with
  $z_k =\infty$.  Thus the left-hand side of \eqref{eqtn:c4j} equals
  $\mathbf{c}_4(\mathbf{z})$.  This means that 
  $$\mathbf{c}_4(\mathbf{z})_t \geq \min( \mathbf{c}_1(\mathbf{z})_t,
  \mathbf{c}_2(\mathbf{z})_t)$$ with equality when $t=j$,
  whenever the right-hand side is not equal to $\infty$.

If the right-hand side of \eqref{eqtn:c4j} is equal to $\infty$, so
$\mathbf{c}_1(\mathbf{z})_t = \mathbf{c}_2(\mathbf{z})_t = \infty$,
then either $c_{1t} = c_{2t} = \infty$, so $c_{4t} =\infty$ and thus
$\mathbf{c}_4(\mathbf{z})_t=\infty$, or $(\mathbf{d}_{\mathbf{w}_t})_k
> \min((\mathbf{d}_{\mathbf{w}_l})_k : c_{1l} \neq \infty) =
\min((\mathbf{d}_{\mathbf{w}_l})_k : c_{2l} \neq \infty)=
(\mathbf{d}_{\mathbf{w}_i})_k = (\mathbf{d}_{\mathbf{w}_j})_k$ for
some $k$ with $z_k = \infty$.  However $\mathbf{c}_4(\mathbf{z})_j
<\infty$, so we must also have $\min((\mathbf{d}_{\mathbf{w}_l})_k :
c_{4l} \neq \infty)= (\mathbf{d}_{\mathbf{w}_j})_k$, and so
$(\mathbf{d}_{\mathbf{w}_t})_k >\min((\mathbf{d}_{\mathbf{w}_l})_k : c_{4l} \neq \infty)$.
So we conclude that if
$\mathbf{c}_1(\mathbf{z})_t = \mathbf{c}_2(\mathbf{z})_t = \infty$
then $\mathbf{c}_4(\mathbf{z})_t= \infty$, and so
  $$\mathbf{c}_4(\mathbf{z}) \geq \mathbf{c}_1(\mathbf{z}) \tplus \mathbf{c}_2(\mathbf{z}).$$

  If $\mathbf{c}_4(\mathbf{z}) \in \mathcal C_{\mathbf{z}}$ then we
  are done.  Otherwise there would be a circuit $\mathbf{c}_5$ of
  $\mathcal M$ with $\supp(\mathbf{c}_5(\mathbf{z}))$ minimal and
  $\supp(\mathbf{c}_5(\mathbf{z})) \subsetneq
  \supp(\mathbf{c}_4(\mathbf{z}))$.
After scaling we may assume that  $\mathbf{c}_5(\mathbf{z}) \geq
  \mathbf{c}_4(\mathbf{z})$, and  $\mathbf{c}_5(\mathbf{z})_l =\mathbf{c}_4(\mathbf{z})_l$
  for some $l$.  If we cannot choose $l=j$, 
then repeating the above procedure with
  $\mathbf{c}_4$ and $\mathbf{c}_5$ playing the roles of
  $\mathbf{c}_1$ and $\mathbf{c}_2$, with $i$ replaced by $l$ would yield
  a circuit $\mathbf{c}_6$ with $\supp(\mathbf{c}_6(\mathbf{z})) \subsetneq
  \supp(\mathbf{c}_4(\mathbf{z}))$, $\mathbf{c}_6(\mathbf{z})_i=\infty$, and 
 $\mathbf{c}_6(\mathbf{z})_j= \mathbf{c}_1(\mathbf{z})_j$, which contradicts the minimality assumption on the choice of $\mathbf{c}_3$.  We thus conclude that $\mathbf{c}_5(\mathbf{z})_j$ is the desired elimination.
\end{proof}

\begin{remark} \label{r:othersvectors}
  When $\mathbf{c}(\mathbf{z}) \not \in \mathcal C_{\mathbf{z}}$ for a
  circuit $\mathbf{c}$ of $\mathcal M$, it is a tropical sum of elements
  of $\mathcal C_{\mathbf{z}}$, so a vector of $\mathcal
  M_{\mathbf{z}}$.  This follows from repeated applications of the
  technique of the last paragraph of the proof of
  Lemma~\ref{l:fiberwelldefined} to get circuits $\mathbf{c}_j$ with
  $\mathbf{c}_j(\mathbf{z}) \geq \mathbf{c}(\mathbf{z})$ and
  $\mathbf{c}_j(\mathbf{z})_j = \mathbf{c}(\mathbf{z})_j$ for all $j
  \in \supp(\mathbf{c}(\mathbf{z}))$.  The vector
  $\mathbf{c}(\mathbf{z})$ is then the tropical sum of all such
  $\mathbf{c}_j(\mathbf{z})$.
\end{remark}

In light of Lemma~\ref{l:fiberwelldefined}, we make the following definition.

\begin{definition}
Let $\mathcal E$ be a tropical toric vector bundle on
$\trop(X_{\Sigma})$.  The fiber of $\mathcal E$ over a point
$\mathbf{z} \in \Rbar^s$ giving the Cox coordinates of a point in
$\trop(X_{\Sigma})$ is the {\em tropical linear space} in
$\Rbar^{|\mathcal G|}$ of the valuated matroid $\mathcal
M_{\mathbf{z}}$.  This is the set of all $\mathbf{a} \in
\Rbar^{|\mathcal G|}$ tropically orthogonal to the circuits 
$\mathbf{c}({\mathbf{z}})$ of $\mathcal M_{\mathbf{z}}$:
$$\{ \mathbf{a} \in \Rbar^{|\mathcal G|} : \text{ either
}\tplus_{\mathbf{w} \in \mathcal G}
\mathbf{c}({\mathbf{z}})_{\mathbf{w}} a_{\mathbf{w}} = \infty, \text{
  or the minimum is achieved at least twice} \}.$$
Let
$$\mathcal C = \{
c_1\mathbf{z}^{\mathbf{u}_1} y_1 \tplus \dots \tplus
c_m\mathbf{z}^{\mathbf{u}_m}y_m : \mathbf{c} \text{ is a circuit of }
\mathcal M  \} \subseteq \Rbar[z_1,\dots,z_s,y_1,\dots,y_m].$$
The total space of $\mathcal E$ is
the quotient of
  $$V( \mathcal C ) \setminus V(B_{\Sigma}) \subseteq \Rbar^s \times
\Rbar^m$$ by the tropical dual of the class group as in
\S\ref{preliTTVar}.  This is a subset of $\trop(X_{\Sigma}) \times
\Rbar^m$.  The fiber over a point $\mathbf{z} \in \Rbar^s$ under the
projection to the first coordinates is the fiber defined above.
\end{definition}

\begin{remark}
  The fact that the fiber over a point $\mathbf{z} \in \Rbar^s$ under
  the projection to the first coordinates is the tropical linear space of $\mathcal
  M_{\mathbf{z}}$ follows from the fact that the specializations of
  the elements of $\mathcal C$ at $\mathbf{z} \in \Rbar^s$ are the elements
  $\mathbf{c}(\mathbf{z})$ for circuits $\mathbf{c}$ of $\mathcal M$.  The tropical
  prevariety defined by the collection $\mathcal C_{\mathbf{z}}$ is the
  tropical linear space of $\mathcal M_{\mathbf{z}}$.  When
  $\mathbf{c}(\mathbf{z}) \not \in \mathcal C_{\mathbf{z}}$ for some
  circuit $\mathbf{c}$, it is the tropical sum of circuits of
  $\mathcal M_{\mathbf{z}}$, so the tropical hyperplane corresponding
  to $\mathbf{c}(\mathbf{z})$ contains the fiber.  Thus the tropical
  prevariety defined by all the $c(\mathbf{z})$ is the fiber over
  $\mathbf{z}$.
\end{remark}

\begin{remark}
 When $\mathcal E$ is realizable, so $\mathcal E = \trop(\mathcal F)$
  for a toric vector bundle $\mathcal F$, the fiber of $\mathcal{E}$
  at a point of $\trop(X_{\Sigma})$ with Cox coordinates $\mathbf{z}$
  is the tropicalization of the fiber of $\mathcal F$ at any point of
  $X_{\Sigma}$ with Cox coordinates $\mathbf{x}$ with
  $\val(\mathbf{x}) = \mathbf{z}$.  To see this, note that when $\mathbf{x}$ corresponds to a point in $U_{\sigma}$ for $\sigma \in \Sigma$, the fiber over $\mathbf{x}$ is
  $$(P_{\mathbf{x}^{\hat{\sigma}}})_{\mathbf{0}} \otimes_{(S_{\mathbf{x}^{\hat{\sigma}}})_{\mathbf{0}}}   ((S
  /\mathfrak{m}(\mathbf{x}))_{\mathbf{x}^{\hat{\sigma}}})_{\mathbf{0}},$$
      where $P$ is the Cox-module of $\mathcal F$, as in
      Proposition~\ref{p:CoxModuleTVB}, $S$ is the Cox ring of
      $X_{\Sigma}$, and $\mathfrak{m}(\mathbf{x})$ is the ideal of all
    polynomials in $S$, homogeneous with respect to the grading on
    $S$, that vanish at $\mathbf{x}$.  Given the description of $P$ in
    Proposition~\ref{p:CoxModuleTVB}, this is isomorphic to
    $$K^m/R_{\mathbf{x}},$$
    where $R_{\mathbf{x}}$ is the submodule generated by 
        $$\left\{\sum_{i=1}^m c_i \mathbf{x}^{\mathbf{u}_i}
  \mathbf{e}_i : (c_1,\dots,c_m) \in \ker(G) \right\}.$$

  The valuated matroid associated to the images of standard basis
  vectors of $K^m$ in the $r$-dimensional vector space
  $K^m/R_{\mathbf{x}}$ is $\mathcal M_{\mathbf{z}}$.
  When $\mathbf{x}$ corresponds to the identity of the torus of $X_{\Sigma}$,
  $\mathcal M_{\mathbf{z}}$ is the matroid $\mathcal M$ associated to
  $\mathcal E$.
\end{remark}

        \begin{remark}
When $\mathcal E$ is a realizable tropical toric vector bundle, there
is a subset of the tropicalization of the fiber over the identity of
the torus that does not depend on the choice of tropicalization.
Consider, for example, the case that the valuated matroid $\mathcal M$
has a trivial valuated matroid structure (so the valuation of every
basis is $0$).  Then the tropical linear space of $\mathcal M$ is the
cone over the order complex of the lattice of flats of
$\underline{\mathcal M}$ by \cite{ArdilaKlivans}.  As the poset
$L'(\mathcal E)$ of all intersections of the $E^i(j)$ is a subposet of
this lattice of flats, its order complex is a subcomplex of the order
complex of $\mathcal L(\underline{\mathcal M})$, so the cone over the
order complex of $L'(\mathcal E)$ is a subfan of the fiber.  We call
this the {\em intrinsic subcomplex}.  As noted in the following
examples, this often has smaller dimension than the actual fiber.
This can be regarded as the polymatroid fiber (but should not be
confused with the Bergman fan of the polymatroid in the sense of
\cite{BergmanFanPolymatroid}).
\end{remark}          

\begin{example} \label{e:trivialbundle}
  The trivial rank-$r$ tropical toric vector bundle with valuated
  matroid $\mathcal M$ on a tropical toric variety $\trop(X_{\Sigma})$
  is given by the filtration with $E^i(j) = \mathcal G$ for $j \leq
  0$, and $E^i(j) = \emptyset$ for $j >0$ for $1 \leq i \leq s$.  In
  this case we have $\mathbf{d}_{\mathbf{w}} = \mathbf{0} \in \mathbb
  Z^s$ for all $\mathbf{w} \in \mathcal G$, so $\mathbf{c}(\mathbf{z})
  = \mathbf{c}$ for all $\mathbf{z} \in \Rbar^s$.  This means that
  $\mathcal M_{\mathbf{z}} = \mathcal M$ for all $\mathbf{z}$, and the
  total space of $\mathcal E$ is $\trop(X_{\Sigma}) \times \Gamma$, where
  $\Gamma$ is the tropical linear space of $\mathcal M$.  In this case the intrinsic subcomplex 
  is the
  origin $\mathbf{0} \in \Rbar^m$.
\end{example}

\begin{example}
  When $\mathcal E$ is the tropical tangent sheaf on $\trop(\mathbb
  P^2)$, as in Example~\ref{e:reflexive} \eqref{item:trs2}, the valuated matroid $\mathcal M$ has one
  valuated circuit up to scaling: $(0,0,0)$.  The underlying matroid
  is $U(2,3)$.  At a point $[z_0:z_1:z_2]$ of $\trop(\mathbb P^2)$ with
  $z_0,z_1,z_2 \neq \infty$ the unique (up to scaling) circuit is  $(z_0,z_1,z_2)$, so $\mathcal M_{[z_0:z_1:z_2]}$ is  a projectively  equivalent
  valuated matroid.  The fiber in $\Rbar^3$ has three cones:
  $(-z_0,-z_1,-z_2)+\cone(\mathbf{e}_i) + \spann((1,1,1))$ for $i=0,1,2$.
  The quotient of the fiber  in $\trop(\mathbb P^2)$ is the standard tropical line
  with vertex $[-z_0:-z_1:-z_2]$.  All three cones lie in the intrinsic subcomplex.

  When $z_0=\infty$, but $z_1,z_2 \neq \infty$ we get the valuated matroid
  with unique (up to scaling) circuit $(\infty, z_1,z_2)$.  This has
  underlying matroid $U(1,1) \oplus U(1,2)$.  The fiber in $\Rbar^3$
  is the hyperplane $z_1+y_1=z_2+y_2$.

  When $z_0=z_1=\infty$, the valuated matroid has the unique (up to
  scaling) circuit $(\infty, \infty, z_2)$, so the underlying matroid is
  $U(2,2) \oplus U(0,1)$.  The fiber in $\Rbar^3$ is the hyperplane
  $y_2=\infty$.

  Recall that $\trop(\mathbb P^2)$ is $(\Rbar^3 \setminus \{(\infty,\infty,\infty) \} / \spann((1,1,1))$.  Thus the total space of the tropical tangent sheaf is
  $$(V(z_0y_0\tplus z_1y_1 \tplus z_2y_2) \setminus
  \{(\infty,\infty,\infty,a_1,a_2,a_3): a_i \in
  \Rbar\})/\spann((1,1,1,0,0,0)) \subseteq \trop(\mathbb P^2) \times
  \Rbar^3.$$
\end{example}  

\begin{example} \label{e:trickymatroidtotalspace}
As noted in Example~\ref{e:trickymatroidrevisited}, the three different
matroids given in Example~\ref{e:trickymatroid} for that toric vector
bundle $\mathcal E$ on $\mathbb P^2$ give three different tropicalizations.  The
total spaces are as follows.
\begin{enumerate}
\item For the matroid $\underline{\mathcal M}_1$ with the choice of
  basis $\mathcal G = \{ \mathbf{e}_1+\mathbf{e}_2, \mathbf{e}_2,
  \mathbf{e}_1+\mathbf{e}_3, \mathbf{e}_3 \}$, and the corresponding
  trivial valuated matroid structure, we have a unique circuit
  $\mathbf{c} = (0,0,0,0)$, and
  $\mathbf{d}_{\mathbf{e}_1+\mathbf{e}_2} = (1,0,0)$,
  $\mathbf{d}_{\mathbf{e}_2} = (1,0,0)$,
  $\mathbf{d}_{\mathbf{e}_1+\mathbf{e}_3} = (0,1,0)$, and
  $\mathbf{d}_{\mathbf{e}_3} = (0,0,1)$.
  Thus $\mathbf{c}(\mathbf{z})
  = (z_0,z_0,z_1,z_2)$, so the total space is $V(z_0y_1 \tplus z_0 y_2
  \tplus z_1 y_3 \tplus z_2 y_4) \subseteq \trop(\mathbb P^2) \times
  \Rbar^4$.  The fiber over $z_0=\infty$ is the classical hyperplane
  $z_1+y_3= z_2+y_4$, while the fiber over $z_1=\infty$ is $V(z_0y_1
  \tplus z_0y_2 \tplus z_2 y_4)$, which is the tropical cone over a tropical line times $\mathbb R$.
The intrinsic subcomplex consists of the
  three cones $\cone((1,1,0,0))+\spann(1,1,1,1), \cone((0,0,1,0)+\spann((1,1,1,1))$, and
  $\cone((0,0,0,1))+\spann((1,1,1,1)$.
  
\item   
For the matroid $\underline{\mathcal M}_2$ with the choice of basis
$\mathcal G = \{ \mathbf{e}_1, \mathbf{e}_2,
\mathbf{e}_1+\mathbf{e}_3, \mathbf{e}_3 \}$, and the corresponding
trivial valuated matroid structure, we have a unique circuit
$\mathbf{c} = (0,\infty,0,0)$, and $\mathbf{d}_{\mathbf{e}_1} =
(1,0,0)$.
Thus $\mathbf{c}(\mathbf{z}) = (z_0,\infty,z_1,z_2)$, so the total
space is $V(z_0y_1 \tplus z_1 y_3 \tplus z_2 y_4) \subseteq
\trop(\mathbb P^2) \times \Rbar^3$.  The fiber over $z_0=\infty$ is
again the classical hyperplane $z_1+y_3= z_2+y_4$, but the fiber over
$z_1=\infty$ is the classical hyperplane $z_0+y_1 = z_2+y_4$.
The intrinsic subcomplex is as in the previous part.

\item 
For the matroid $\underline{\mathcal M}_2$ with the choice of basis
$\mathcal G = \{ \mathbf{e}_1+\mathbf{e}_2, \mathbf{e}_1-\mathbf{e}_2,
\mathbf{e}_2, \mathbf{e}_1+\mathbf{e}_3, \mathbf{e}_3,
\mathbf{e}_1+\mathbf{e}_2+\mathbf{e}_3 \}$, and the corresponding
trivial valuated matroid structure, we have $9$ circuits.  We have
$\mathbf{d}_{\mathbf{e}_1-\mathbf{e}_2} = (1,0,0)$, and
$\mathbf{d}_{\mathbf{e}_1+\mathbf{e}_2+\mathbf{e}_3} = (0,0,0)$.  The
total space is thus
\begin{multline*} V(y_1 \tplus y_2 \tplus y_3, 
  z_0y_3 \tplus z_1 y_4 \tplus y_6,
  z_0y_1 \tplus z_2 y_5 \tplus y_6, 
z_0y_1 \tplus z_0 y_2 \tplus z_1 y_4 \tplus z_2y_5,\\
z_0y_1 \tplus z_0 y_2 \tplus z_1 y_4 \tplus y_6,
z_0y_1 \tplus z_0 y_3 \tplus z_1 y_4 \tplus z_2y_5,
z_0y_2 \tplus z_0 y_3 \tplus z_1 y_4 \tplus z_2y_5,\\
z_0y_2 \tplus z_0 y_3 \tplus z_2 y_5 \tplus y_6,
z_0y_2 \tplus z_1 y_4 \tplus z_2 y_5 \tplus y_6
) \subseteq \trop(\mathbb P^2) \times \Rbar^6.\end{multline*}
The intrinsic subcomplex consists of the three cones
  $\cone((1,1,1,0,0,0))+\spann(1,1,1,1,1,1)$,
  $\cone((0,0,0,1,0,0)+\spann((1,1,1,1,1,1))$, and
  $\cone((0,0,0,0,1,0))+\spann((1,1,1,1,1,1)$ in $\Rbar^6$.
The fiber over $z_0=\infty$ is
$$V(y_1 \tplus y_2 \tplus y_3, z_1 y_4 \tplus y_6, z_2 y_5 \tplus y_6, 
z_1 y_4 \tplus z_2y_5),$$
which is the tropical cone over a tropical line times $\mathbb R$,
 and the fiber over $z_1=\infty$ is 
\begin{multline*} V(y_1 \tplus y_2 \tplus y_3, 
z_0y_3 \tplus y_6, z_0y_1 \tplus z_2 y_5 \tplus y_6, 
z_0y_1 \tplus z_0 y_2 \tplus z_2y_5,\\
z_0y_1 \tplus z_0 y_2 \tplus y_6,
z_0y_1 \tplus z_0 y_3 \tplus z_2y_5,
z_0y_2 \tplus z_0 y_3  \tplus z_2y_5,\\
z_0y_2 \tplus z_0 y_3 \tplus z_2 y_5 \tplus y_6,
z_0y_2 \tplus z_1 y_4 \tplus z_2 y_5 \tplus y_6
\end{multline*}
\end{enumerate}

\end{example}

\begin{example}  \label{e:vamostotalspace}
  Let $\mathcal E = (\mathcal M, \mathcal G, \{E^1(j) \})$ be the
  tropical toric vector bundle on $\trop(\mathbb A^1)$ with $\mathcal M$ the
  trivial valuated matroid structure on the Vamos matroid as in
  Remark~\ref{dist_latt_not_sat}, and $E^1(j)= \mathcal G$ for $j \leq
  0$, and $E^1(j) = \emptyset$ for $j >0$.  Then the fiber over a
  point in tropical torus is the tropical cone in $\Rbar^8$ over
  Bergman fan of the Vamos matroid.
\end{example}

	\section{Properties of tropical toric vector bundles} \label{s:properties}
	
In this section we extend some standard properties of vector bundles
to tropical toric vector bundles: decompositions into direct sums,
tensoring by line bundles, and global sections.
	
	\subsection{Decomposition into direct sums} \label{ss:directsum}

	We now give a definition for a tropical toric vector bundle to
	decompose as a direct sum of other tropical toric vector bundles.
	
	We first recall the notion of the direct sums of
	valuated matroids; see, for example~\cite{HLSV}.
	If $\mathcal M_1$ is a valuated matroid of rank $r_1$ on a
        ground set $\mathcal G_1$ with basis valuation function
        $\nu_1$, and $\mathcal M_2$ is a valuated matroid of rank
        $r_2$ on a ground set $\mathcal G_2$ with basis valuation
        function $\nu_2$, then $\mathcal M_1 \oplus \mathcal M_2$ is a
        valuated matroid on the ground set $\mathcal G_1 \sqcup
        \mathcal G_2$ with basis valuation function $\nu(I_1 \cup I_2)
        = \nu_1(I_1)+\nu_2(I_2)$ for all sets $I_1 \subseteq \mathcal
        G_1$ of size $r_1$ and $I_2 \subseteq \mathcal G_2$ of size
        $r_2$.  The lattice of flats $\mathcal L(\mathcal M_1
        \oplus \mathcal M_2)$ is isomorphic to $\mathcal L(\mathcal M_1) \times
        \mathcal L(\mathcal M_2)$.
	
	This lets us define the direct sum of tropical toric vector bundles.  
	
	\begin{definition}
		Let $\mathcal E_1=\left( \mathcal{M}_1, \mathcal{G}_1, \{E_{\mathcal M_1}^i(j)\}\right)$ and $\mathcal E_2=\left( \mathcal{M}_2, \mathcal{G}_2, \{E_{\mathcal M_2}^i(j)\}\right)$ be tropical toric vector bundles
		on $\trop(X_{\Sigma})$.  The tropical toric vector bundle $\mathcal
		E_1 \oplus \mathcal E_2$ is given by the valuated matroid $\mathcal
		M_1 \oplus \mathcal M_2$ on the ground set
		$\mathcal G_1 \sqcup \mathcal G_2$ with $E^i(j) = E^i_{\mathcal M_1}(j) \sqcup
		E^i_{\mathcal M_2}(j)$.  Compatibility for $\mathcal E_1 \oplus
		\mathcal E_2$ follows by setting $B_{\sigma} = B_{\sigma, \mathcal
			M_1} \sqcup B_{\sigma,\mathcal M_2}$, and $\mathbf u(\sigma) =
		\mathbf u(\sigma)_{\mathcal M_1} \cup \mathbf{u}(\sigma)_{\mathcal
			M_2}$.  
		
		A tropical toric vector bundle $\mathcal E$ is {\em decomposable} if
		it is the direct sum of two tropical toric vector bundles of smaller
		rank.  A tropical toric vector bundle is {\em indecomposable} if it is not decomposable.
	\end{definition}

	Recall that a matroid is {\em connected} if it is not a nontrivial
	direct sum of two matroids of lower ranks. The following lemma shows that the Krull-Schmidt theorem (see \cite[Section 4.3]{PerT}) holds for tropical toric vector bundles.
	
	\begin{lemma}\label{l:Krull-Schmidt}
		A tropical toric vector bundle $\mathcal E=\left( \mathcal{M}, \mathcal{G}, \{E^i(j)\}\right)$ is indecomposable if and
		only if the underlying matroid $\underline{\mathcal M}$ is connected.  Every
		tropical toric vector bundle can be written uniquely as the direct
		sum of indecomposable toric vector bundles.
	\end{lemma}
	
	\begin{proof}
		The first claim follows from the fact that a valuated matroid
		$\mathcal M$ is a direct sum of two valuated matroids of smaller ranks
		if and only if the underlying matroid $\underline{\mathcal M}$ is not
		connected.  If the valuated matroid decomposes, then each flat in the
		lattice of flats of the underlying matroid is a union of flats of the
		two constituent matroids, which lets us define the summand tropical
		toric vector bundles.  The second claim then follows from the fact
		that matroids have well-defined connected components.
	\end{proof}
	
	\begin{remark}\label{rmk6.1}
		If $\mathcal E$ is an indecomposable toric vector
                bundle, then $\trop(\mathcal E)$ is also
                indecomposable.  To see this, the key observation is
                that when a realizable matroid is decomposable, it is
                decomposable as a direct sum of {\em realizable}
                matroids.  Thus if $\trop(\mathcal E)$ is
                decomposable, the realizable matroid
                $\underline{\mathcal M}(\mathcal E)$ is a direct sum
                of matroids of smaller ranks, which are themselves
                realizable.  This means that the vector space $E$
                decomposes as a direct sum of subspaces spanned by the
                realizations of each of the summand matroids.  This 
                induces a decomposition of $\mathcal E$ as a direct
                sum of toric vector bundles.
		
		Note, however, that the tropicalization of a decomposable vector
		bundle might not be decomposable; this depends on the choice of
		matroid for the bundle $\mathcal E$.  For example, consider the
		toric vector bundle $\mathcal E$ of Example~\ref{e:trickypolymatroid}.
		This is decomposable: we can write $K^3 =
		\spann(\mathbf{e}_1,\mathbf{e}_3) \oplus \spann(\mathbf{e}_2)$. This
		is reflected in the matroid $\mathcal M_2$ from that example, which
		is the direct sum $U_{2,3} \oplus U_{1,1}$.  The matroids $\mathcal
		M_1$ and $\mathcal M_3$ do not decompose as direct sums, however.
		Given a decomposable toric vector bundle $\mathcal E \cong \mathcal E_1
		\oplus \mathcal E_2$ we can always {\em choose} a valuated matroid ${\mathcal M}$ for $\mathcal E$ for which $\trop(\mathcal E)$ is decomposable.
	\end{remark}

	A tropical toric vector bundle $\mathcal E$ {\em splits as a
          direct sum of line bundles} if there are $\mathcal
        E_1,\dots, \mathcal E_r$ all of rank one with $\mathcal E =
        \mathcal E_1 \oplus \dots \oplus \mathcal E_r$.  If $\mathcal
        E$ is tropically minimal we then have $\underline{\mathcal
          M}(\mathcal E)$ equal to the uniform matroid $U_{r,r}$, as a
        tropically minimal tropical vector bundle of rank one has
        matroid $U_{1,1}$.
	
We now see the use of the partially modular condition of
Definition~\ref{d:partiallymodular} for decomposing into direct sums.

	\begin{proposition}[Tropical version of Grothendieck's theorem]  $($cf. \cite[Example 2.3 (3)]{Kly}, \cite[Example 3.3]{GUZ}$)$\label{p:P1splits}
		Let $\mathcal E$ be a partially modular tropical toric vector bundle on $\trop(\mathbb P^1)$
		that is tropically minimal.  Then $\mathcal E$ splits as a direct
		sum of line bundles.
	\end{proposition}
	
	\begin{proof}
		Since the fan of $\mathbb P^1$ has only two rays, we have two flags of
		flats $E^0(j)$ and $E^1(k)$ for $j,k \in \mathbb Z$.  By
		Lemma~\ref{l:partiallymodularmeansbasis} there is a basis $B$ of the lattice of
		flats $\mathcal L(\underline{\mathcal M})$ for which each $E^0(j)$ and
		$E^1(k)$ is the join of a subset of $B$.  Since  $\mathcal E$ is
		tropically minimal, the lattice of flats must equal the lattice of
		subsets of $B$, so the underlying matroid is the uniform matroid
		$U_{r,r}$, where $r$ is the rank of $\mathcal E$.  This is the direct
		sum $\oplus_{i=1}^r U_{1,1}$, so we conclude that $\mathcal E$ splits
		as the direct sum of line bundles.
	\end{proof}

	\begin{example}\label{vbonP1split}
		\begin{enumerate}
	\item Consider the tropically minimal tropical toric vector bundle $\mathcal
          E$ on $\trop(\mathbb{P}^1)$ of Example \ref{e:reflexive}
          \eqref{item:trs6}.  The lattice of flats is shown in Figure
          \ref{f:Matroideg1}. The underlying matroid decomposes as
          \(U_{3,3}=\{1\} \oplus \{3\} \oplus \{4\}\).  Consider the
          filtrations
			
				$$L_1^{1}(j) = \begin{cases} 
				\{1\} &  j \leq 2 \\
				\emptyset  &  j > 2 \\
			\end{cases}, \, 
			L_2^{1}(j) = \begin{cases} 
				\{3\} &  j \leq 1 \\
				\emptyset  &  j > 1 \\
			\end{cases},  \, 
			L_3^{1}(j) = \begin{cases} 
				\{4\} &  j \leq 0 \\
				\emptyset  &  j > 0 \\
			\end{cases},$$
			$$L_1^{2}(j) = \begin{cases} 
				\{3\} &  j \leq 2 \\
				\emptyset  &  j > 2 \\
			\end{cases}, \, 
			L_2^{2}(j) = \begin{cases} 
				\{4\} &  j \leq 1 \\
				\emptyset  &  j > 1 \\
			\end{cases},  \, 
			L_3^{2}(j) = \begin{cases} 
				\{1\} &  j \leq 0 \\
				\emptyset  &  j > 0 \\
			\end{cases}.$$

			These gives rise to tropical toric line bundles $\mathcal{L}_1$, $\mathcal{L}_2$ and $\mathcal{L}_3$ respectively. Note that $\mathcal{E}= \mathcal{L}_1 \oplus \mathcal{L}_2 \oplus \mathcal{L}_3$.
			
			The tropical toric vector bundle $\mathcal E$  is partially modular; in fact, the entire lattice of flats in Figure \ref{f:Matroideg1} is a modular lattice.

			\item \label{item:P1notdecomposable} Consider the rank-$3$ realizable matroid $\underline{\mathcal
				M}$ with ground set $\mathcal G = \{
			\mathbf{w}_1=(1,0,0), \mathbf{w}_2=(0,1,0),
			\mathbf{w}_3=(0,0,1), \mathbf{w}_4=(1,1,1)\}$.  Let $\mathcal
			E$ be the tropical toric vector bundle on $\trop(\mathbb P^1)$
			given by

			$$E^0(j)= \begin{cases} \mathcal G & j \leq 0 \\
				\{\mathbf{w}_1,\mathbf{w}_2\} &  j = 1 \\ \emptyset & j >1 \\ \end{cases} \, \,\, , \,
			E^1(j)= \begin{cases} \mathcal G & j \leq 0 \\
				\{\mathbf{w}_3,\mathbf{w}_4\} &  j = 1 \\ \emptyset & j >1 \\ \end{cases}.$$
			Then $\mathcal E$ is not decomposable.  It is also not partially modular, as
			\begin{equation*}
				\begin{split}
					3 & =   \rk((E^0(0) \meet E^1(1)) \join (E^0(1) \meet E^1(0))\\
				& \leq \rk( E^0(0) \meet E^1(1))+ \rk(E^0(1) \meet E^1(0)) - \rk(E^0(1) \meet E^1(1)) \\
				& =4\\
				\end{split}
			\end{equation*}
			Note that while $\underline{\mathcal M}$ is
                        realizable, $\mathcal E$ is not realizable.
                        This is because in the realizable case we
                        impose the condition that all intersections of
                        $E^i(j)$ are flats of the matroid, which is
                        not the case here, as $\spann(E^0(1)) \cap \spann(E^1(1)) = \spann((1,1,0))$, and a multiple of $(1,1,0)$ 
                        is not in the ground set of
                        $\underline{\mathcal M}$.  This shows that a
                        condition like partially modularity is
                        necessary for Proposition~\ref{p:P1splits}.
			\item  Consider $\mathcal O(a D_0+bD_1)$ on $\mathbb P^1$ as a toric vector bundle.  This has
			$$E^0(j) = \begin{cases} K& j \leq a \\ 0 & j>a \\ \end{cases} \, \, \, ,
			E^1(j) = \begin{cases} K & j \leq b \\ 0 & j>b \\ \end{cases}. \, \, \,
			$$ for $i = 0,1$.  Let $\mathcal E = \mathcal O(D_0+2D_1) \oplus
			\mathcal O(2D_0+D_1)$, and choose the matroid $\mathcal G = \{ (1,0),
			(0,1), (1,1) \}$ for $\mathcal E$.  Then $\trop(\mathcal E)$ is not
			decomposable, so not a direct sum of line bundles.  In this case
			$\trop(\mathcal E)$ is partially modular, but not tropically minimal;
			the subposet given by the matroid of $\{ (1,0), (0,1) \}$ contains the
			$E^i(j)$, so this shows the necessity of the tropically minimal
			condition for Proposition~\ref{p:P1splits}.
		\end{enumerate}      
	\end{example}  
	
	\begin{remark}
		Note that Example~\ref{vbonP1split} \eqref{item:P1notdecomposable}  shows that a tropical vector bundle whose underlying matroid is realizable need not be a realizable tropical vector bundle.
	\end{remark}

\subsection{Tensor product with line bundles.}\label{tensor prod}

We now define the tensor product of a tropical toric vector bundle by
a tropical line bundle.

Recall that for a toric vector bundle $\mathcal{E}$ on a toric variety
$X_{\Sigma}$ with associated filtrations $(E , \{ E^{i}(j) \})$, the
tensor product of $\mathcal{E}$ by a torus-invariant line bundle is
 obtained by shifting the filtration.
Specifically, consider the line bundle $\mathcal{O}_X(D)$,
for some torus-invariant Cartier divisor $D=\sum_{i \in
  \Sigma(1)}a_{i}D_{i}$ with the associated filtrations $(L, \{
L^{i}(j) \})$. Then the filtrations associated to the tensor product
$\mathcal{E} \otimes \mathcal{O}_X(D)$ are given by $(E , \{
E^{i}(j-a_{i}) \})$ (see \cite[Example II.9.]{gonzalez}).  
Motivated by
this we define the tensor product by a line bundle as follows.

\begin{definition}
Let \(\mathcal{E}=\left( \mathcal{M}, \mathcal{G}, \{E^i(j)\}\right) \) be
a tropical toric vector bundle on tropical toric variety
$\trop(X_{\Sigma})$.  A tropical line bundle $\mathcal{L}$ is
determined by a sequence of integers \(a_1, \ldots, a_s\) (see Example \ref{e:reflexive}
\eqref{item:trs1}).  We define the tensor product \(\mathcal{E}
\otimes \mathcal{L}\) by the data \((\mathcal{M}, \mathcal{G},
\{E^i(j-a_i)\})\): for the ray $\rho_i$ the $j$th piece of the filtration for $\mathcal E \otimes \mathcal L$ is the $(j-a_i)$th piece $E^i(j-a_i)$ of the filtration for $\mathcal E$.
\end{definition}

We now define the notion of isomorphism of tropical toric vector bundle motivated by the following result on toric vector bundles.
	\begin{proposition}\cite[Corollary 1.2.4]{Kly}
		Let $\mathcal{E}$ and $\mathcal{F}$ be indecomposable toric vector bundles on a complete toric variety \(X_{\Sigma}\). If $\mathcal{E}$ and $\mathcal{F}$ are isomorphic as ordinary bundles, then $\mathcal{E}$ is equivariantly isomorphic to $\mathcal{F} \otimes \divv(\chi^m)$ for some character \(m \in M\).
	\end{proposition}
	Note that in view of Lemma \ref{l:Krull-Schmidt} it suffices to define isomorphism of tropical toric vector bundle for indecomposable bundles.
	\begin{definition}\label{Dtropiso}
		Two indecomposable tropical toric vector bundles $\mathcal{E}_1=\left( \mathcal{M}_1, \mathcal{G}_1, \{E_1^i(j)\}\right) $ and $\mathcal{E}_2=\left( \mathcal{M}_2, \mathcal{G}_2, \{E_2^i(j)\}\right) $ on a tropical toric variety $\trop(X_{\Sigma})$ are said to be isomorphic if the valuated  matroids \(\left( \mathcal{M}_1, \mathcal{G}_1\right)\) and \(\left( \mathcal{M}_2, \mathcal{G}_2\right)\) are isomorphic and under the identification of the lattice of flats $\mathcal{L}(\underline{\mathcal{M}}_1)$ and $\mathcal{L}(\underline{\mathcal{M}}_2)$, we have
		\begin{equation*}
			E^i_1(j)=E^i_2(j+ \mathbf{u} \cdot \mathbf{v}_i)
		\end{equation*}
		for some \( \mathbf{u} \in M\) and for all \(i\) and \(j\).
	\end{definition}

\begin{definition}
  The Picard group of a tropical toric variety $\trop(X_{\Sigma})$,
  denoted by $\text{Pic}(\trop(X_{\Sigma}))$, is the set of
  isomorphism classes of tropical toric line bundles under tensor
  product.
Any tropical line bundle on a
tropical toric variety $\trop(X_{\Sigma})$ is realizable (see
Definition \ref{pd:tropE}) as any rank 1 matroid is realizable. Hence,
we have $\text{Pic}(X_{\Sigma})=\text{Pic}(\trop(X_{\Sigma}))$, as expected from \cite{JMT19}.
\end{definition}

	\begin{example}  
	  Consider the toric variety $\mathbb{P}^2$.
          Any tropical toric line bundle $\mathcal{L}$ is determined by a vector \((a_0, a_1, a_2) \in \mathbb{Z}^3\). Let \(\mathbf{u}=(-a_1, -a_2) \in M=\mathbb{Z}^2\). Then $\mathcal{L}$ is isomorphic to the vector \((a_0+a_1+a_3, 0, 0)\) under the isomorphism defined in Definition \ref{Dtropiso}. Hence, the group $\text{Pic}(\trop(\mathbb P^2))$ is isomorphic to $\mathbb Z$.
          	\end{example}

        \begin{remark}
There is not a simple definition of the tensor product of tropical
toric vector bundles.  This is because there is not a tensor product
of matroids in general.  In \cite{LasVergnas} Las Vergnas showed that,
given two matroids $\mathcal M_1$, $\mathcal M_2$, there is in general
no definition of a matroid with ground set the product $\mathcal G_1
\times \mathcal G_2$ of the two ground sets, rank the product of the
ranks, the product of two flats a flat, and the restriction of the
matroid to sets of the form $\{x\} \times \mathcal M_2$ or $\mathcal
M_1 \times \{y\}$ is a flat.  These are all natural requirements for the
tensor product of matroids, and thus for the tensor product of
tropical toric vector bundles, so any definition of ``the'' or ``a''
tensor product needs to relax one of them.

When $\mathcal E$ and $\mathcal F$ are toric vector bundles on a toric
variety $X_{\Sigma}$, so define realizable tropical vector bundles, 
the choices of realizations of  matroids for $\mathcal E$ and $\mathcal F$ give rise to a
matroid for $\trop(\mathcal E \otimes \mathcal F)$.
However this matroid may depend on the choice of realization of the two matroids.  
For example, suppose that the matroid of both $\mathcal E$ and
$\mathcal F$ is the uniform matroid $U(2,4)$, with realizations
\begin{center}
	$\left( \begin{array}{rrrr}  1 &  0 & 1 & 1 \\
		0  & 1 & 1 & a  \\\end{array} \right) \text{ and }$
	$\left( \begin{array}{rrrr}  1 &  0 & 1 & 1 \\
		  0  & 1 & 1 & b  \\\end{array} \right).$
\end{center}
We write $\mathbf{v}_1,\dots,\mathbf{v}_4$ for the columns of the first matrix, and $\mathbf{v}'_1,\dots,\mathbf{v}'_4$ for the columns of the second matrix.  The tensor product of these two vector configurations defines a matroid on ground set $\{ \mathbf{v}_i \otimes \mathbf{v}'_j : 1 \leq i,j \leq 4 \}$ of size $16$.  The subconfiguration $\{ \mathbf{v}_1 \otimes \mathbf{v}'_1,
\mathbf{v}_2 \otimes \mathbf{v}'_2,   \mathbf{v}_3 \otimes \mathbf{v}'_3,   \mathbf{v}_4 \otimes \mathbf{v}'_4 \}$ has matrix
$$\left( \begin{array}{rrrr} 1 & 0 & 1 & 1 \\ 0 & 0 & 1 & b \\ 0 & 0 &
  1 & a \\ 0 & 1 & 1 & ab \\ \end{array} \right),$$ which has
determinant $a-b$.  Thus this subset is a basis for the matroid of $\trop(\mathcal E \otimes \mathcal F)$ if and only if $a \neq b$, so the tropicalization of the tensor product depends on the choice of realization.

In this case there is a generic choice for the matroid, which is the
matroid of the tensor product when none of $\{a-b, ab-1, ab-a-b,
ab-a+1, ab-b+1,a+b-1 \}$ are zero.  However this is a consequence of
the realization space of the matroid $U(2,4)$ being irreducible.  When
the realization spaces of the matroids for $\mathcal E$ and $\mathcal F$  are reducible there may be no
choice of generic matroid for the tensor product.
\end{remark}

	\subsection{Parliaments of polytopes}
	\label{s:parliament}
	
	In \cite{DJS} Di Rocco, Jabbusch, and Smith associate a collection of
	polytopes to a toric vector bundle that encode the global sections of
	the bundle.  We now give a tropical formulation of this, which lets us
	{\em define} $h^0(\mathcal E)$ for a tropical toric vector bundle.  For the rest of this section we assume that $X_{\Sigma}$ is smooth and complete.
	
	 Recall that for a tropical toric vector bundle $\mathcal E$
         with associated filtrations \( \{E^i(j)\} \) and $\mathbf{w}$
         in the ground set of $\underline{\mathcal M}(\mathcal E)$ we
         have $\mathbf{d}_{\mathbf{w}} \in \mathbb Z^{|\Sigma(1)|}$
         given by $(\mathbf{d}_{\mathbf{w}})_i = \max(j : \mathbf{w}
         \in E^i(j) )$.
	
	\begin{definition} \label{d:parliament}
		Let $\mathcal E$ be a tropical toric vector bundle. For each $\mathbf{w}$ in the ground set $\mathcal G$ of $\underline{\mathcal M}(\mathcal E)$ we set
		$$P_{\mathbf{w}} = \{ \mathbf{u} \in M_{\mathbb R} : \mathbf{u} \cdot \mathbf{v}_i \leq (\mathbf{d}_{\mathbf{w}})_i \text{ for all } 1 \leq i \leq s \}.$$
		We call $\{P_{\mathbf{w}} : \mathbf{w} \in \mathcal G\}$ the {\em parliament of polytopes} associated to $\mathcal E$.
		For $\mathbf{u} \in M$, we define
		\begin{equation} \label{eqtn:h0} h^0(\mathcal E)_{\mathbf{u}} =\rk\left( \bigvee\limits_{\mathbf{u} \in P_{\mathbf{w}}} \mathbf{w} \right),\end{equation}
		and $$h^0(\mathcal E) = \sum_{\mathbf{u} \in M} h^0(\mathcal E)_{\mathbf{u}}.$$
	\end{definition}  
	Note that the assumption that $\Sigma$ is complete implies that each
	$P_{\mathbf{w}}$ is bounded, so a polytope.  Thus $h^0(\mathcal
	E)_{\mathbf{u}}$ is zero for all but finitely many $\mathbf{u}$, and
	so the last sum is finite.
	
	\begin{example} \label{e:parliaments}
		We consider the tropicalization of the toric vector bundle $\mathcal E$ of Example~\ref{e:trickymatroid}.
		\begin{enumerate}
			\item  \label{item:firstparliament} For the matroid with ground set
                          $\{\mathbf{e}_1+\mathbf{e}_2,
                          \mathbf{e}_2,
                          \mathbf{e}_1+\mathbf{e}_3,
                          \mathbf{e}_3 \}$ in
                          Example~\ref{e:trickymatroid}
                          \eqref{item:e:trickymatroid0}, we have
			\begin{equation*}
				\mathbf{d}_{\mathbf{e}_1+\mathbf{e}_2} = \mathbf{d}_{\mathbf{e}_2} = (1, 0, 0), \mathbf{d}_{\mathbf{e}_1+\mathbf{e}_3} = (0,1,0) \text{ and } \mathbf{d}_{\mathbf{e}_3} = (0,0,1).
			\end{equation*}
			The parliament of polytopes is 
			\begin{equation*}
				\begin{split}
				&	P_{\mathbf{e}_1+\mathbf{e}_2} = P_{\mathbf{e}_2} = \{ (u_1,u_2) \in M_{\mathbb R} : -u_1-u_2 \leq 1, u_1 \leq 0, u_2 \leq 0 \},\\
				&	P_{\mathbf{e}_1+\mathbf{e}_3} = \{(u_1,u_2)  \in M_{\mathbb R} : -u_1-u_2 \leq 0, u_1 \leq 1, u_2 \leq 0 \} \text{ and }\\ 
				&  P_{\mathbf{e}_3} = \{(u_1,u_2)  \in M_{\mathbb R} : -u_1-u_2 \leq 0, u_1 \leq 0, u_2 \leq 1 \}.
				\end{split}
			\end{equation*}

			 These are shown in Figure~\ref{f:parliaments}.
			 	\begin{figure}
			 	\caption{The parliament of polytopes for Example~\ref{e:parliaments}\eqref{item:firstparliament} \label{f:parliaments}}
			 	\begin{tikzpicture}[scale=1]
\coordinate (A1) at (0,0);
\coordinate (A3) at (0,2);
\coordinate (A2) at (2,0);
\coordinate (A4) at (-2,2);
\coordinate (A5) at (-2,0);
\coordinate (A6) at (0,-2);
\coordinate (A7) at (2,-2);
\coordinate (B2) at (3,0);
\coordinate (B3) at (0,3);
\coordinate (B5) at (-3,0);
\coordinate (B6) at (0,-3);

\definecolor{c1}{RGB}{0,129,188}
\definecolor{c2}{RGB}{252,177,49}
\definecolor{c3}{RGB}{35,34,35}

	\draw[c3, <->, very thick] (B2) -- (B5);
	\draw[c3, <->, very thick] (B3) -- (B6);
	
	\fill [black!30] (A1) -- (A5) --(A6)  -- cycle;
	\fill [black!30] (A1) -- (A3) --(A4)  -- cycle;
		\fill [black!30] (A1) -- (A2) --(A7)  -- cycle;
		\filldraw[black] (0,0) circle (2pt);
		\filldraw[black] (2,0) circle (2pt);
		\filldraw[black] (2,-2) circle (2pt);
		\filldraw[black] (0,2) circle (2pt);
		\filldraw[black] (-2,0) circle (2pt);
		\filldraw[black] (-2,2) circle (2pt);
		\filldraw[black] (0,2) circle (2pt);
		\filldraw[black] (0,-2) circle (2pt);
			\node at (-1.5,-1.3) { \( P_{\mathbf{e}_2}\)};
						\node at (-1.5,2.3) { \( P_{\mathbf{e}_3}\)};
												\node at (1.5,-2.3) { \( P_{\mathbf{e}_1+\mathbf{e}_3}\)};
			 		\end{tikzpicture}
			 \end{figure}
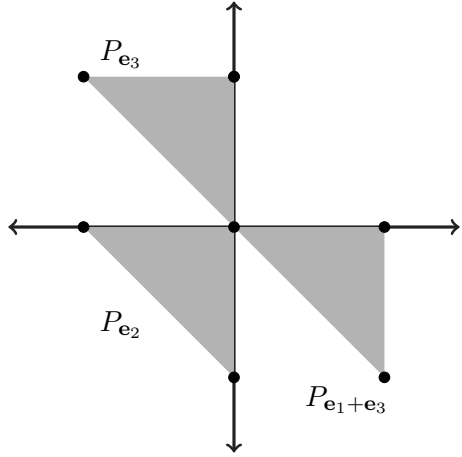 
			 
			   Note that
			 \begin{equation*}
			 	\begin{split}
			 		h^0(\trop(\mathcal E)) = &	h^0(\trop(\mathcal E))_{(0,0)}+h^0(\trop(\mathcal E))_{(1,0)} + h^0(\trop(\mathcal E))_{(1,-1)}+h^0(\trop(\mathcal E))_{(0,1)} \\
			 	& + h^0(\trop(\mathcal E))_{(-1,1)} + h^0(\trop(\mathcal E))_{(-1,0)}+ h^0(\trop(\mathcal E))_{(0,-1)}\\
			 	& =3+1+1+1+1+2+2=11.
			 	\end{split}
			 \end{equation*}
			\item For the matroid with ground set $\{\mathbf{e}_1, \mathbf{e}_2,
			\mathbf{e}_1+\mathbf{e}_3, \mathbf{e}_3 \}$ in Example~\ref{e:trickymatroid} \eqref{item:e:trickymatroid}, we have
			$\mathbf{d}_{\mathbf{e}_1} = (1, 0,0)$, so $P_{\mathbf{e}_1} =
			P_{\mathbf{e}_2}$, $P_{\mathbf{e}_1+\mathbf{e}_3}$ and $P_{\mathbf{e}_3}$ are as above.  Thus $h^0(\trop(\mathcal E))$ is again $11$.  
			
			\item For the matroid with ground set $\{\mathbf{e}_1+\mathbf{e}_2, \mathbf{e}_1-\mathbf{e}_2, \mathbf{e}_2, \mathbf{e}_1+\mathbf{e}_3, \mathbf{e}_3,\mathbf{e}_1+\mathbf{e}_2+\mathbf{e}_3 \}$ in Example~\ref{e:trickymatroid} \eqref{item:e:trickymatroid1},   $P_{\mathbf{e}_1-\mathbf{e}_2} = P_{\mathbf{e}_2} = P_{\mathbf{e}_1+\mathbf{e}_2}$.  We have $\mathbf{d}_{\mathbf{e}_1+\mathbf{e}_2+\mathbf{e}_3} = (0, 0, 0)$, so 
			  $$P_{\mathbf{e}_1+\mathbf{e}_2+\mathbf{e}_3} = \{ (u_1,u_2) \in M_{\mathbb R} : -u_1-u_2 \leq 0, u_1 \leq 0, u_2 \leq 0\} = \{ \mathbf{0} \}.$$
                          We then get
\begin{equation*}                     \begin{split}h^0(\trop(\mathcal E)) = h^0(\trop(\mathcal E))_{(0,0)} +
h^0(\trop(\mathcal E))_{(1,0)} + h^0(\trop(\mathcal E))_{(1,-1)}+h^0(\trop(\mathcal E))_{(0,1)} \\
			 	 + h^0(\trop(\mathcal E))_{(-1,1)} + h^0(\trop(\mathcal E))_{(-1,0)}+ h^0(\trop(\mathcal E))_{(0,-1)}\\
			 	 =3+1+1+1+1+2+2=11.
  \end{split}
  \end{equation*}
		\end{enumerate}    
		Note that in this example dimension of the space of global sections does not depend on the choice of the matroid. The following lemma shows that this is true in general.
	\end{example}

	\begin{lemma} \label{l:tvbh0}
		When $\mathcal E = \trop(\mathcal F)$ for a toric vector bundle $\mathcal F$ on $X_{\Sigma}$ we have
		$$h^0(\mathcal E)_{\mathbf{u}} = h^0(\mathcal F)_{\mathbf{u}} \text{ for $\mathbf{u} \in M$ and thus }
		h^0(\mathcal E) = h^0(\mathcal F).$$
	\end{lemma}
	
	\begin{proof}
		 The proof of Proposition 1.1 of \cite{DJS} gives an
                 isomorphism $H^0(X,\mathcal E)_{\mathbf{u}} \cong
                 \bigcap_{i=1}^s E^i(\mathbf{u} \cdot \mathbf{v}_i)$,
                 essentially following \cite[Corollary 4.13]{Kly}.
                 Since a spanning set of $\bigcap_{i=1}^s
                 E^i(\mathbf{u} \cdot \mathbf{v}_i)$ belongs to the
                 lattice of flats of the underlying matroid of
                 $\mathcal E = \trop(\mathcal F)$ (see Definition
                 \ref{d:matroidoftvb}), the lemma follows from the
                 proof of \cite[Proposition 1.1]{DJS}.
	\end{proof}  
	
	\begin{remark} \label{r:DJSok}
		A version of the formula~\eqref{eqtn:h0} is found in the work of
		Kaveh and Manon connecting toric vector bundles with piecewise
		linear valuations; see \cite{KMTVBvalTrop23}*{Proposition 3.12}.  Note that
		one corollary of Lemma~\ref{l:tvbh0} is that the dimension of
		$h^0(\mathcal E)$ for a toric vector bundle as given in
		\cite{DJS}*{Proposition 1.1} does not depend on the choice of
		matroid of $\mathcal E$.
	\end{remark}

\section{Stability} \label{s:stability}
	
In this section we give a definition of (semi)stability for a tropical
toric reflexive sheaf, and construct Jordan-H\"older and
Harder-Narisimhan filtrations.  The concept of modularity from lattice
theory plays an essential role here; see
\S\ref{ss:latticesmodularity}.  Throughout this section we assume that the fan $\Sigma$ defines a {\em projective} toric variety.  For convenience we also assume that $\Sigma$ is simplicial, though this is only needed for defining the polarization \eqref{eqtn:polarization}.

\subsection{Definition of stability}

Recall that a reflexive sheaf $\mathcal E$ on a nonsingular projective $n$-dimensional variety with distinguished very ample divisor $H$ is  {\em stable} if the slope 
	$$\mu(\mathcal F) = \frac{c_1(\mathcal F) \cdot H^{n-1}}{\rk(\mathcal F)}$$ 
of a torsion-free subsheaf $\mathcal F$ of $\mathcal E$ with \(0 < \rk(\mathcal F) < \rk(\mathcal E) \) and first Chern class $c_1(\mathcal F)$ 
satisfies
$$\mu(\mathcal F)  < \mu(\mathcal E).$$
The sheaf $\mathcal E$ is {\em semistable} if the condition holds with
$<$ replaced by $\leq$.
	
Note that for a toric variety  it suffices to consider only equivariant reflexive
subsheaves of $\mathcal E$ (see \cite[Remark 2.3.2]{DDK-Stab}) to
check (semi)stability. We now formulate a tropical version of this
concept. This first requires the definition of a reflexive subsheaf of
a tropical reflexive sheaf. We are motivated by the following
proposition which establishes a correspondence between equivariant
subsheaves and subspaces.

\begin{proposition}[{\cite[Corollary 0.0.2]{ErSEVB}}] \label{p:subsheafsubspace}
	Let $ \mathcal{E}$ be an equivariant reflexive sheaf on a toric variety $X_{\Sigma}$ with associated filtrations $\left( E, \{E^{i}(j) \} \right)$. Then  equivariant saturated subsheaves of $\mathcal{E}$ are in one-to-one correspondence with induced subfiltrations $\left( F, \{F^{i}(j) \} \right)$ of $\left( E, \{E^{i}(j) \} \right)$, where $F$ is a subspace of $E$ and $F^{i}(j) = E^{i}(j) \cap F$.
\end{proposition}
See also \cite[Proposition 2.3]{HNS}.

	\begin{definition}\label{sheaf_asso_F}
	  Let $\mathcal E=\left( \mathcal{M}, \mathcal{G}, \{E^i(j)\}\right) $ be a tropical reflexive sheaf on a tropical toric variety $\trop(X_{\Sigma})$ as in Definition~\ref{d:tropicalreflexivesheaf}.
          We associate to a flat $F \in \mathcal{L}(\underline{\mathcal{M}})$ a tropical reflexive sheaf as follows.  Let $\mathcal M|_F$ be the restriction
          of the valuated matroid $\mathcal{M}$ (see \S \ref{ss:valuatedmatroids}).   Note that the lattice of flats $\mathcal{L}(\underline{\mathcal{M}}|_F)$ is the interval \([\emptyset, F]\) in $\mathcal{L}(\underline{\mathcal{M}})$.  For any ray $\rho_i \in \Sigma(1)$ and $j \in \mathbb Z$, set 
		$$F^i(j) = F \wedge E^i(j) \in
          \mathcal{L}(\underline{\mathcal M}) . $$ The data
          \(\mathcal{E}_F:=(\mathcal{M}|_F, F, \{F^i(j)\})\) defines a
          tropical toric reflexive sheaf, which we call a reflexive subsheaf,
          or just a subsheaf,  of $\mathcal{E}$.
	\end{definition}

		\begin{remark} In  \cite[Proposition 2.5]{GUZ}, the authors showed that every subbundle \(\mathcal F\) of a vector bundle
		\(\mathcal{E}\) on a metric graph $\Gamma$ is split:
		there exists another subbundle \(\mathcal H\) of
		$\mathcal{E}$ such that \(\mathcal{E} \cong
		\mathcal{F} \oplus \mathcal{H}\). This phenomenon does not occur for tropical toric vector bundles. Consider
		the Hirzebruch surface \(X=\mathbb F_1\), whose fan
		has rays spanned by  \(\mathbf{v}_1=(1,0), \,
		\mathbf{v}_2=(0,1), \, \mathbf{v}_3=(-1, 1) \) and
		\(\mathbf{v}_4=(0, -1)\). Consider the tangent bundle
		$\mathcal{E}=\mathscr{T}_X$ with filtrations given by
		\(\left( E, \{E^i(j)\} \right) \), where
		\[ E^{i}(j) = \left\{ \begin{array}
			{r@{\quad \quad}l}
			E=\spann(\mathbf{v}_1, \mathbf{v}_2) & j \leq 0 \\ 
			\spann(\mathbf{v}_i) & j=1 \\
			0 & j > 1
		\end{array} \right. .\] 
		Then $\mathcal{E}$ has as realizable matroid the matroid 
		$\underline{\mathcal{M}}(\mathcal{E})$ with ground set
		$\mathcal{G}=\{\mathbf{v}_1, \mathbf{v}_2,
		\mathbf{v}_3\}$.   Note that we do not include $\mathbf{v}_4$ as $\spann(\mathbf{v}_4) = \spann(\mathbf{v}_2)$. Consider the flat
		\(F=\spann(\mathbf{v}_2)\). Then, by
		\cite[Proposition 4.1.1]{KD} \(F\) corresponds to a
		subbundle $\mathcal{F}$ with associated filtrations
		\((F, \{F \cap E^i(j)\})\). If there is a subsheaf
		\(\mathcal H\) of $\mathcal{E}$ such that
		\(\mathcal{E} \cong \mathcal{F} \oplus \mathcal{H}\),
		we must have \(H^1(1)=\spann(\mathbf{v}_1)\).
		From the direct sum we also need
		\(H^3(1)=\spann(\mathbf{v}_3)\), which is not
		possible as $\mathcal{H}$ has rank one. In \cite[Section 3.2]{Lucie}, the author also constructed a similar example. Now tropicalizing, we get the tropical toric vector bundle $\trop(\mathcal E)=\left( \underline{\mathcal{M}}(\mathcal{E}), \mathcal E, \{E^i(j)\}\right) $ on $\trop(X_{\Sigma})$. The flat \(F\) corresponds to a tropical toric subsheaf $\mathcal{E}_F=\trop(\mathcal{F})$, which is not split as the above argument suggests.
		\end{remark}

	Recall that the first Chern class of a toric vector bundle $\mathcal{E}$ associated with Klyachko data \((E, \{E^i(j)\})\) has the following combinatorial description  (see \cite[Remark 3.2.4]{Kly}):
	\begin{equation} \label{eqtn:toricChernclass}
		c_1{\mathcal{(E)}}= \sum\limits_{\mathbf{v}_i \in \Sigma(1), j \in \mathbb{Z}} j \text{ dim} \left( \frac{E^{i}(j)}{E^{i}(j+1)}\right)  D_{i},
	\end{equation}
	We thus {\em define} the first Chern class of a tropical toric reflexive sheaf as follows.
	
	\begin{definition} \label{d:firstChernclass}
		Let $\mathcal E$ be a tropical reflexive sheaf on a tropical toric variety $\trop(X_{\Sigma})$ given by the data \((\mathcal{M}, \mathcal{G}, \{E^i(j)\})\). Recall that we have the exact sequence  \eqref{eqtn:A1}
		\begin{equation*}
		0 \longrightarrow M \longrightarrow \mathbb{Z}^s \stackrel{\phi}{\longrightarrow} A^1(\Sigma) \longrightarrow 0.
		\end{equation*} 
                Let $\mathbf{c} \in \mathbb Z^s$ be the vector with $i$th component
		\begin{align}  \label{eqtn:Chernclass}
                  c_i= \sum\limits_{j \in \mathbb{Z}} j  \left( \rk(E^{i}(j)) - \rk (E^{i}(j+1))\right).
		\end{align}
We define the {\em first Chern class} $c_1(\mathcal E)$ as the image of
$\mathbf{c}$ in \(A^1(\Sigma)\) under the map $\phi$.  We now define
the degree of a tropical reflexive sheaf $\mathcal E$. We fix a
polarization \(H\), i.e., an ample line bundle on the toric variety
\(X_{\Sigma}\). This gives rise to a vector
		\begin{equation} \label{eqtn:polarization}
			\mathbf{h}:=(h_1, \ldots, h_s) \in \mathbb{N}^s, \text{ where } h_i=D_i \cdot H^{n-1}.
		\end{equation}
		We define the degree as follows:
		\begin{equation*}
			\deg(\mathcal{E}):= \sum\limits_{\rho_i \in \Sigma(1), j \in \mathbb{Z}} j  \left( \rk (E^{i}(j)) - \rk (E^{i}(j+1)) \right) \, h_i \in \mathbb N.
		\end{equation*}
	We fix such a vector \(\mathbf{h}\) throughout the rest of the
        section, and all definitions implicitly depend on the choice
        of $\mathbf{h}$.  We define the slope as the quotient
		\begin{equation*}
			\mu(\mathcal{E}):=\frac{\deg(\mathcal{E})}{\mathrm{rank}(\mathcal{E})}.
		\end{equation*}
	\end{definition}
	
        \begin{example}
When $\mathcal E$ is a rank-$1$ toric vector bundle, given by $\mathbf{a} \in \mathbb Z^s$ as in Example~\ref{e:reflexive} \eqref{item:trs1}, we have $c_1(\mathcal E)$ equal to the image of $\mathbf{a}$ in $A^1(\Sigma)$.
\end{example}          
	
	\begin{remark}
		Let $\mathcal{E}$ be a toric vector bundle on a 
                toric variety \(X_{\Sigma}\). Consider the
                tropicalization $\trop(\mathcal{E})$ defined in
                Definition \ref{pd:tropE}. Then $c_1(\mathcal E) =
                c_1(\trop(\mathcal E))$, as the rank of the flat
                $E^i(j)$ in the tropicalization equals the dimension
                of the subspace that it corresponds to, so
                \eqref{eqtn:toricChernclass} agrees with \eqref{eqtn:Chernclass}.
                Since the
                definition of degree in both cases only depends on a
                class in $A^1(X_{\Sigma}) \cong A^1(\Sigma)$, we also
                have $\deg(\mathcal E) = \deg(\trop(\mathcal E))$, and
                so $\mu(\mathcal{E})=\mu(\trop(\mathcal{E}))$.  Note
                that this does not depend on the choice of matroid for
                $\mathcal E$, and thus on the precise tropicalization.
	\end{remark}
	
	\begin{definition}
		Let $\mathcal E$ be a tropical reflexive sheaf on a tropical toric variety $\trop(X_{\Sigma})$ given by the data \((\mathcal{M}, \mathcal{G}, \{E^i(j)\})\).  We say that $\mathcal E$ is {\em stable} if 
		$$\mu(\mathcal E_F)  < \mu(\mathcal E) $$                
		for all proper nonempty flats $F$ of the lattice of flats of the underlying matroid $\underline{\mathcal M}$.  The sheaf $\mathcal E$ is {\em semistable} if the condition holds with $<$ replaced by $\leq$.  
                
	\end{definition}  
	
	\begin{example}
		Let $\mathcal{E}$ be a rank-\(1\) tropical toric reflexive sheaf on a tropical toric variety $\trop(X_{\Sigma})$ given by the data \((\mathcal{M}, \mathcal{G}, \{E^i(j)\})\). Since the underlying matroid $\underline{\mathcal M}$ does not contain any proper flats, $\mathcal{E}$ is vacuously stable.
	\end{example}
	
	\begin{example} \label{eg_tangent_bdl} 
	  Let \(X_{\Sigma}=\mathbb{P}^n\) be projective space, and
          consider the tropical tangent bundle \(\trop(\mathscr{T})\) defined in Example \ref{e:reflexive}\eqref{item:trs3}.  The only $j$ which contributes to the first Chern class computation is $j=1$, for which $j(\rk(E^i(j)-\rk(E^i(j+1))=1$, so the first Chern class $c_1(\trop(\mathscr T))$ is the image of $(1,\dots,1)$ in $A^1(\mathbb P^n) \cong \mathbb Z$, which is $1 \in \mathbb Z$.  For the polarization $\mathbf{h} = (1,\dots,1)$, corresponding to $\mathcal O(1)$ on $\mathbb P^n$, the degree is then $n+1$, so the slope of  \(\trop(\mathscr{T})\) is given by
	  \[\mu(\trop(\mathscr{T}))=\frac{n+1}{n}.\]
          
		Let \(F\) be a proper flat of rank \(l < n\). Then \(|F|=l\) as any proper subset of
                the ground set of the matroid of $\trop(\mathscr{T})$
                is linearly independent. Hence, the slope of the
                subsheaf corresponding to \(F\) is
                \(\mu(\mathscr{T}_F)=1 <
                \mu(\trop(\mathscr{T}))\). Thus,
                \(\trop(\mathscr{T})\) is stable.  Note that this is also true for the 
                tangent bundle on (nontropical) projective space (see \cite[Theorem 1.3.2]{Okonek}).
	\end{example}

	\begin{example} Let \(\mathcal{E}=\left( \mathcal{M}, \mathcal{G}, \{E^i(j)\}\right) \) be a tropical toric vector bundle on a tropical toric variety $\trop(X_{\Sigma})$ and let $\mathcal{L}$ be a tropical line bundle, determined by $\mathbf{a} \in \mathbb Z^s$ as in Example~\ref{e:reflexive}\eqref{item:trs1}. 
          We have
	  \begin{equation*}
			\begin{split}
 \text{deg}(\mathcal{E} \otimes \mathcal{L})
 &= \sum\limits_{\rho_i \in \Sigma, j \in \mathbb{Z}} j
 \left( \rk (E^{i}(j-a_i)) - \rk (E^{i}(j-a_i+1))\right) h_{i}\\
 &= \sum\limits_{\rho_i \in \Sigma} \left(  \sum\limits_{j \in
     \mathbb{Z}} (j-a_i) \left( \rk (E^{i}(j-a_i)) - \rk (E^{i}(j-a+1))
 \right) h_i \right. \\
& \hspace{1cm} \left. + a_i \sum_{j \in \mathbb Z}  \left( \rk (E^{i}(j-a_i)) - \rk (E^{i}(j-a+1)) \right) h_i \right) \\
& = \deg(\mathcal E) + \sum_{\rho_i \in \Sigma}  a_i  \sum_{j \in \mathbb Z} \left( \rk (E^{i}(j-a_i)) - \rk (E^{i}(j-a+1)) \right) h_i\\
   & = \deg(\mathcal E) + \sum_{\rho_i \in \Sigma} a_i \rk(\mathcal G) h_i  \\
   & = \deg(\mathcal E) + \rk(\mathcal G) \deg(\mathcal L).\\
			\end{split}
		\end{equation*}
		Thus we have \(\mu(\mathcal{E} \otimes \mathcal{L})=\mu(\mathcal{E}) + \mu (\mathcal{L})\). Hence, a tropical toric vector bundle $\mathcal{E}$ is (semi)stable if and only if \(\mathcal{E} \otimes \mathcal{L}\) is (semi)stable for any tropical line bundle $\mathcal{L}$.
	\end{example}

	\begin{proposition} \label{Sprop1.1}
		Stable tropical toric vector bundles are indecomposable.
	\end{proposition}
	
	\begin{proof}
		Let \(\mathcal{E}=\left( \mathcal{M}, \mathcal{G},
                \{E^i(j)\}\right) \) be a stable tropical toric vector bundle of rank
                \(r\). Suppose that $\mathcal{E}$ is decomposable.
                Then $\mathcal M = \mathcal M_1 \oplus \mathcal M_2$, for valuated matroids $\mathcal M_1$ and $\mathcal M_2$ of ranks $r_1$, $r_2$ respectively on ground sets
                $\mathcal{G}_1$ and $\mathcal{G}_2$
with $\mathcal{G}=
\mathcal{G}_1 \sqcup \mathcal{G}_2$.
Note that                 their lattice of flats are related as follows:
		\begin{equation*}
			\mathcal{L}(\underline{\mathcal M})=\mathcal{L}(\underline{\mathcal M}_1) \times \mathcal{L}(\underline{\mathcal M}_2).
		\end{equation*}
		Then, \(E^i(j)=E^i_1(j) \times E^i_2(j) \), where \(E^i_l(j) \in \mathcal{L}(\underline{\mathcal M}_l) \) for \(l=1, 2\). So we can consider the tropical reflexive sheaves \(\mathcal{E}_l=\left( \mathcal{M}_l, \mathcal{G}_l, \{E_l^i(j)\}\right) \) for \(l=1, 2\). Note that we have
		\begin{equation*}
			\rk(E^i(j))=\rk(E^i_1(j)) + \rk(E^i_2(j)).
		\end{equation*}
		Hence, from the definition of the first Chern class (Definition~\ref{d:firstChernclass}), we see that
		\begin{equation*}
			c_1(\mathcal{E})=c_1(\mathcal{E}_1) + c_1(\mathcal{E}_2).
		\end{equation*}
                This implies that
                $$\deg(\mathcal E) = \deg(\mathcal E_1) + \deg(\mathcal E_2).$$
		Without loss of generality we may assume that \(\mu(\mathcal{E}_1) \leq \mu(\mathcal{E}_2)\). Then
		\begin{equation*}
			\begin{split}
				\mu(\mathcal{E})=& \frac{\deg(\mathcal{E}_1)+\deg(\mathcal{E}_2)}{r_1+r_2}\\
				=& \frac{r_1}{r} \frac{\deg(\mathcal{E}_1)}{r_1} + \frac{r_2}{r} \frac{\deg(\mathcal{E}_2)}{r_2}\\
				\leq & \left( \frac{r_1}{r} + \frac{r_2}{r}\right) \mu(\mathcal{E}_2) =  \mu(\mathcal{E}_2).
			\end{split}
		\end{equation*}
		Since \(\underline{\mathcal{M}}_2= \underline{\mathcal{M}}|_F\), where \(F\) is the flat \(\emptyset \times \mathcal{G}_2 \in \mathcal{L}(\underline{\mathcal M})\), this contradicts the stability of $\mathcal{E}$. Thus $\mathcal{E}$ must be indecomposable. 
	\end{proof}  
	
	\begin{remark}
	Note that any matroid $\underline{\mathcal{M}}$ can be written
        as a direct sum $\underline{\mathcal{M}} \cong
        \underline{\mathcal{M}}' \oplus \left( U_{0,1}\right)^{\oplus
          l} \oplus \left( U_{1, 1}\right)^{\oplus c}$, where
        \(\underline{\mathcal{M}}'\) has neither loops nor coloops
        (see \cite[Proposition 5.3]{Katz}). Proposition
        \ref{Sprop1.1} shows that a necessary condition for a
        tropical toric vector bundle to be stable is that the
        underlying matroid should, in addition to being loopfree as required in Definition~\ref{d:tropicalreflexivesheaf}, also be coloopfree.
	\end{remark}

	\begin{example} (Behavior of (semi)stability under tropicalization). \label{stabUtrop}
		\begin{enumerate}
			\item \label{item:stabUtrop0} Example \ref{eg_tangent_bdl} shows that tangent bundle on $\mathbb{P}^n$ remains stable under tropicalization.
			
			\item  \label{item:stabUtrop} Let $\mathcal{E}$ be the toric vector bundle on $\mathbb{P}^2$ considered in Example~\ref{e:trickymatroid}. The slope is given by
			\begin{equation*}
				\mu(\mathcal{E})= \frac{2+1+1}{3}=1 + \frac{1}{3}.
			\end{equation*}
			Consider the subspace \(F=\spann(\mathbf{e}_1, \mathbf{e}_3) \subseteq E\). The corresponding subsheaf is given by the following filtrations:
			
			{\footnotesize $$F^{0}(j) = \left\{ \begin{array}
					{r@{\quad \quad}l}
					F & j \leq 0 \\ 
					\spann(\mathbf{e}_1)  & j=1 \\
					0 & j > 1
				\end{array} \right. ,
				E^{1}(j) = \left\{ \begin{array}
					{r@{\quad \quad}l}
					F & j \leq 0 \\ 
					\spann(\mathbf{e}_1+\mathbf{e}_3) & j=1 \\
					0 & j > 1
				\end{array} \right.,
				E^{2}(j) = \left\{ \begin{array}
					{r@{\quad \quad}l}
					F & j \leq 0 \\ 
					\spann(\mathbf{e}_3) & j=1 \\
					0 & j > 1
				\end{array} \right. .$$}
			
			The slope is given by
			\begin{equation*}
				\mu(\mathcal{F})= \frac{1+1+1}{2}=1 + \frac{1}{2}.
			\end{equation*}
			Thus $\mathcal{E}$ is not semistable. 
			  
Now we consider the tropicalization from Example
\ref{e:trickymatroid}\eqref{item:e:trickymatroid0} and show that it
is stable.  Let $\mathcal{M}$ be the rank-\(3\) matroid on the ground set $\mathcal{G}=\{\mathbf{e}_1+\mathbf{e}_2,\mathbf{e}_2,
\mathbf{e}_1+\mathbf{e}_3, \mathbf{e}_3\}$. The lattice of flats is
given in Figure \ref{eg8}.
			\begin{figure}
				\caption{\label{eg8}}
				\begin{tikzpicture}
					[scale=.6,auto=left]
					\node (n0) at (2,0) {$\emptyset$};
					\node (n1) at (-1,1.5)  {1};
					\node (n2) at (1,1.5)  {2};
					\node (n3) at (3,1.5) {3};
					\node (n4) at (5,1.5)  {4};
					\node (n5) at (-2.5,4)  {12};
					\node (n6) at (-1,4) {23};
					\node (n7) at (1,4)  {13};
					\node (n8) at (3,4) {24};
					\node (n9) at (5,4)  {14};
					\node (n10) at (6.6,4)  {34};
					\node  (n12) at (2, 6) {$\mathcal G$};
					\foreach \from/ \to in {n4/n0,n3/n0,n2/n0,n1/n0,n5/n1,n5/n2,n6/n2, n6/n3, n7/n1,n7/n3,n8/n2, n8/n4,n9/n1,n9/n4, n10/n3, n10/n4, n12/n6, n12/n7, n12/n8, n12/n9, n12/n5, n12/n10}
					\draw (\from) -- (\to);
				\end{tikzpicture}
			\end{figure} 
			The tropicalization is the tropical toric vector bundle $\trop(\mathcal{E})=\left(\mathcal{M}, \mathcal{G}, \{E^i(j)\} \right) $, where the filtrations of the flats are given as follows:
			
			{\footnotesize $$E^{0}(j) = \left\{ \begin{array}
					{r@{\quad \quad}l}
					\mathcal{G} & j \leq 0 \\ 
					\{\mathbf{e}_1+\mathbf{e}_2, \mathbf{e}_2\}& j=1 \\
					\emptyset & j > 1
				\end{array} \right.,
				E^{1}(j) = \left\{ \begin{array}
					{r@{\quad \quad}l}
					\mathcal{G} & j \leq 0 \\ 
					\{\mathbf{e}_1+\mathbf{e}_3\} & j=1 \\
					\emptyset & j > 1
				\end{array} \right.,
				E^{2}(j) = \left\{ \begin{array}
					{r@{\quad \quad}l}
					\mathcal{G} & j \leq 0 \\ 
					\{\mathbf{e}_3\} & j=1 \\
					\emptyset & j > 1
				\end{array} \right. .$$}
			
			The slope is given by
			\begin{equation*}
				\mu(\trop(\mathcal{E}))= 1 + \frac{1}{3}.
			\end{equation*}
			All the rank-$1$ and rank-$2$ flats have
                       slope $1$. Hence, $\trop\mathcal{(E)}$
                        is stable.

  The issue here is that the destabilizing subsheaf corresponding to
  the subspace $F = \spann(\mathbf{e}_1,\mathbf{e_3})$ does not
    tropicalize to a subsheaf of $\trop(\mathcal E)$.  We have $F \cap
    \mathcal G = \{ \mathbf{e}_1+\mathbf{e}_3, \mathbf{e}_3\}$, which
    is a flat of $\mathcal L(\underline{\mathcal M})$, but the
      restriction of $\mathcal M$ to $F$ has lattice of flats
      $[\emptyset,F]$, which just has two elements of rank $1$.  However
      the subsheaf of $\mathcal E$ determined by $F$ has three
      one-dimensional flats: in addition to
      $\spann(\mathbf{e}_1+\mathbf{e}_3)$ and $\spann(\mathbf{e}_3)$
      in $[\emptyset,F]$, there is also $\spann(\mathbf{e}_1) = E^0(2)
      \cap F$.  This is a consequence of the lattice of flats not
      being modular: we have $E^0(2) \meet (E^1(3) \join E^2(1)) =
      \emptyset$, and $E^0(2) \join (E^1(3) \join E^2(1)) = \mathcal G$,
      so the sum of these ranks is $0+3=3$, which is less than
      $\rk(E^0(2)) + \rk(E^1(3) \join E^2(1)) = 2+2=4$.
     
\item If we choose the tropicalization of the toric vector bundle
$\mathcal{E}$ in Example \ref{stabUtrop}\eqref{item:stabUtrop} to be
given by the matroid of Example \ref{e:trickymatroid}\eqref{item:e:trickymatroid}, then we see $\trop\mathcal{(E)}$ is not
semistable as well.  Note that here modularity assumption is satisfied
by Proposition \ref{propmodflat}\eqref{item:propmodflat2}.

		\end{enumerate}
		
	\end{example}

\begin{lemma} \label{l:stabilitypreserved} Let $\mathcal E$ be a
stable, or semistable, toric vector bundle.  Then $\trop(\mathcal E)$
is stable, respectively semistable.  \end{lemma}

\begin{proof} Write $\mathcal F = (\mathcal M,\mathcal G, \{ E^i(j)
\})$ for $\trop(\mathcal E)$, where the ground set $\mathcal G$ is a
subset of the vector space $E$.  If $\mathcal E$ is stable then
for every subsheaf $\mathcal E_F$ induced by a subspace $F \subseteq
E$ in the sense of Proposition~\ref{p:subsheafsubspace} or \cite[Proposition 2.3]{HNS} we have $\mu(\mathcal E_F) < \mu(\mathcal E)$.     In particular this
is true for the subspaces spanned by subsets of $\mathcal G$, which
correspond to flats of $\underline{\mathcal M}$.  

 By construction $c_1(\mathcal F) = c_1(\mathcal E)$, so $\mu(\mathcal
 F) = \mu(\mathcal E)$.  For a flat of $\underline{\mathcal M}$ we
 write $F$ for the subspace and $\tilde{F} \subseteq \mathcal G$ for
 the set; in particular we also use this notation to make a
 distinction between $E^i(j) \subseteq E$ and $\tilde{E}^i(j)
 \subseteq \mathcal G$.  To show that $\mathcal F$ is stable it
 suffices to show that $\mu(\mathcal F_{\tilde{F}}) \leq
   \mu(\trop(\mathcal E_F))$, as that will show that slope of every
   subsheaf induced by a flat of $\underline{\mathcal M}$ is less than
   $\mu(\mathcal F)$.  Since the rank of the flat $\tilde{F}$ equals
   the dimension of the subspace $F = \spann(\tilde{F})$, and $h_i
   \geq 0$ for all $1 \leq i \leq s$, to show this it suffices to show
   that $c_1(\mathcal F_F)_i \leq c_1(\mathcal E_F)_i$ for all $i$.
   This follows, using \cite[Lemma 2.5]{HNS}, from the observation
   that $\rk(\tilde{E}^i(j) \meet \tilde{F}) \leq \dim( E^i(j) \cap
   F)$. 
\end{proof}

\begin{remark}
\label{r:semistabletropicalize}
  The converse of Lemma~\ref{l:stabilitypreserved} is not
true: if $\mathcal E$ is unstable, then $\trop(\mathcal E)$ might not
be unstable.  One example is given by Example \note{7.13(2)}, where
the toric vector bundle $\mathcal E$ is not stable, but the first
tropicalization is.

There are two ways in which the tropicalization of an unstable toric
vector bundle can fail to be unstable.  The first, which is what
occurs in Example \note{7.13(2)}, is that the subspace of $E$
corresponding to a destabilizing subsheaf may not be a flat of
$\underline{\mathcal M}$.  The second is that the subspace may be a
flat of $\underline{\mathcal M}$, but the inequality
$\rk(\tilde{E}^i(j) \meet \tilde{F}) \leq \dim( E^i(j) \cap F)$ may be
strict in a way that means that the slope of the $\trop(\mathcal
E)_{\tilde{F}}$ is less than the slope of $\mathcal E_F$.  This
does not happen when $\tilde{E}^i(j)$ and $\tilde{F}$ are a modular pair.

However, if $\mathcal E$ is toric vector bundle that is not
(semi)stable, it is always possible to choose a matroid for $\mathcal
E$ in such a way that the tropicalization $\trop(\mathcal E)$ is not
(semi)stable.  This is inspired by \cite{Lucie}*{Theorem 2.18}, though
see Remark~\ref{r:rank3} for some differences.  If $\mathcal E$ is not
stable, then by \cite[Theorem 2.1]{BDGP}) there is an equivariant
destabilizing subsheaf $\mathcal F$.  Such a sheaf is given by
filtrations $\{F^{\sigma}_{\mathbf{u}} \}_{\sigma \in \Sigma,
  \mathbf{u} \in M}$ \cite[Chapter 4]{PerT}.  By \cite[Corollary
  3.18]{Kool} the first Chern class of $\mathcal F$ only depends on
the $F^{\rho_i}_{\mathbf{u}}$ and the formula is the same as for the
reflexive sheaf with filtrations $F^i(j) := F^{\rho_i}_\mathbf{u}$
when $j = \mathbf{u} \cdot \mathbf{v}_i$.  Here we use that
$F^{\rho_i}_{\mathbf{u}} = F^{\rho_i}_{\mathbf{u}'}$ if $\mathbf{u}
\cdot \mathbf{v}_i = \mathbf{u}' \cdot \mathbf{v}_i$.  Let
$\mathcal{F}'$ be this toric reflexive sheaf.  By construction we have
$\mu(\mathcal F) = \mu(\mathcal F')$, so $\mathcal F'$ is also
destabilizing.  After passing to its saturation if necessary, we have
$F^i(j) = E^i(j) \cap F$ for some subspace $F$ of $E$.  Given any
matroid for $\mathcal E$, we can augment $\mathcal G$ by adding a
basis for each $F \cap E^i(j)$ to $\mathcal G$.  This ensures that
$\trop(\mathcal F) = \trop(\mathcal E)_{\tilde{F}}$, where $\tilde{F}$
is the flat of the matroid $\underline{\mathcal M}$ for $\mathcal E$
corresponding to $F$.  This means that $\mu(\trop(\mathcal
E)_{\tilde{F}}) = \mu(\mathcal F)$, so the flat $\tilde{F}$ is
destabilizing for $\trop( \mathcal E)$.
          \end{remark}

\begin{remark} \label{r:rank3}
  The first tropicalization in Example \ref{stabUtrop}\eqref{item:stabUtrop} shows that the whether a rank-$3$ toric
  vector bundle is stable can not always be detected from its
  tropicalization.  This contrasts with \cite[Theorem 2.23]{Lucie},
  where Devey shows that (semi)stability of a toric\ vector bundle
  $\mathcal{E}$ of rank at most $ 3$ can be determined by considering
  flats of {\em any} choice of the associated matroid
  $\underline{\mathcal{M}}$.  The distinction is that in \cite{Lucie}
  the flat is identified with the subspace: then content of
  \cite[Theorem 2.23]{Lucie} is that if there is a destabilizing
  sheaf, it is $\mathcal E_F$ for a subspace $F \subseteq E$
  corresponding to a flat of the chosen matroid.  However the
  stability of $\trop(\mathcal E)$ depends on $\trop(\mathcal
  E)_{\tilde{F}}$, which may not equal $\trop(\mathcal E_F)$.
\end{remark}

	\subsection{Jordan-H\"older and Harder-Narasimhan filtrations}

        In this section we construct Jordan-H\"older and
        Harder-Narasimhan filtrations for tropical toric reflexive
        sheaves.  This allows us to reduce the study of tropical toric
        vector bundles to (semi)stable ones.  In usual algebraic
        geometry these justify the study of moduli spaces of stable vector bundles, and we hope that the same will be true in the tropical setting.

	We defined the notion of a subsheaf of a tropical reflexive
        sheaf in Definition \ref{sheaf_asso_F}. We now formalise the
        notion of the quotient of a tropical reflexive sheaf.  This uses the notion of the {\em quotient} of valuated matroids; see \S\ref{ss:valuatedmatroids}.
        
	\begin{definition}
	  Let \(\mathcal{E}=\left( \mathcal{M}, \mathcal{G},
          \{E^i(j)\}\right) \) be a tropical toric reflexive sheaf. A
          flat \(F \in \mathcal{L}(\underline{\mathcal M})\)
          corresponds to a subsheaf \(\mathcal{E}_F:=(\mathcal{M}|_F,
          F, \{F^i(j)\})\) (see Definition \ref{sheaf_asso_F}).
          Consider the contraction of the valuated matroid \(\mathcal
          M\) by the flat \(F\), denoted by \(\mathcal M/ F\), as
          recalled in \S\ref{ss:valuatedmatroids}.  This is a matroid
          on the ground set \(\mathcal{G} \setminus F\) with lattice
          of flats \(\mathcal{L}(\underline{\mathcal M}/ F)\)
          isomorphic to the interval \([F, \mathcal{G}]\) in
          $\mathcal{L}(\underline{\mathcal M})$.   Note that \(\left(
          \mathcal M/ F, \mathcal{G} \setminus F, \{
          E'^i(j)\}\right) \) is a tropical toric reflexive sheaf,
          where
		$$ E'^{i}(j) = (E^i(j) \vee F) .$$ 
          Here we are using the isomorphism $\mathcal L(\underline{\mathcal M}/F) \cong [F,\mathcal G]$ to regard ${E'}^i(j)$ as a subset of $\mathcal G \setminus F$.
                We call this the quotient of $\mathcal{E}$ by the subsheaf $\mathcal{E}_F$ and denote it by \(\mathcal{E}/F\). Note that 
		\begin{equation}\label{egqt1}
			\begin{split}
				& \rk(\mathcal M/ F)=\rk(\underline{\mathcal M})- \rk(F), \text{ and }\\
				&\rk_{\mathcal M/ F}((E^i(j) \vee F) )=\rk_{\underline{\mathcal M}} (E^i(j) \vee F)-\rk_{\underline{\mathcal M}} (F)
			\end{split}
		\end{equation}
		(see \cite[Proposition 7.4.3]{White}). Here we have used the subscripts $\mathcal{M}$ and $\mathcal M/ F$ to denote the corresponding rank functions. We usually omit the subscript for the rank function of $\mathcal{M}$. Let \(F' \in \mathcal{L}(\underline{\mathcal M})\) be another flat containing the flat \(F\). We have \(F \in \mathcal{L}(\underline{\mathcal M}|_{F'})\) and we call the contraction
		\[\left( (\mathcal M|_{F'})/ F,  F'\setminus F, \{((E^i(j) \meet F') \vee F) \}\right) ,\]
		the quotient of \(\mathcal E_{F'}\) by \(\mathcal
                E_F\).  This is denoted by \(\mathcal{E}_{F'/
                  F}\). 
                Note that \((\underline{\mathcal M}|_{F'})/ F\) has
                rank \(\rk(F')-\rk(F)\) and
		\begin{equation}\label{egqt1.1}
			\begin{split}
				\rk_{\mathcal{M}_{F'/ F}}(((E^i(j) \meet F') \vee F) )= \rk((E^i(j) \meet F') \vee F)-\rk(F).
			\end{split}
		\end{equation}
	\end{definition}
	
	\begin{remark}[Comparison with realizable tropical vector bundles]
		The definition of matroid quotient is consistent with
                the realizable case. Let $\mathcal{E}$ be a toric
                vector bundle with associated filtrations $\left(E,
                \{E^i(j)\} \right)$ and let $\mathcal{F}$ be a subsheaf
                corresponding to the subspace $F \subseteq E$. Then the
                filtrations associated to the quotient
                $\mathcal{E}/\mathcal{F}$ are given by
                $\left(\frac{E}{F}, \{\frac{E^i(j)+F}{F}\}
                \right)$.

Suppose $\underline{\mathcal M}$ is a realizable
matroid for $\mathcal E$, with ground set $\mathcal G \subseteq E$,
and $F$ is the span of a subset of $\mathcal G$, so corresponds to a
flat of $\underline{\mathcal M}$, which we will also call $F$.  Then
$E^i(j) + F$ corresponds to the flat $E^i(j) \join F$.  Quotienting
$E$ by $F$ corresponds to contracting $\underline{\mathcal M}$ by $F$.  The
quotient $\underline{\mathcal M}/F$ has lattice of flats isomorphic to the
interval $[F,\mathcal G]$ in $\mathcal L(\underline{\mathcal{M}})$.
Thus $\{E^i(j) \join F\}$ is the induced filtration on the quotient.
	\end{remark}
	
	\begin{remark}\label{rmkforSlem1.1}
		Using \eqref{egqt1}, we have
		\begin{equation*}
			\begin{split}	
			\deg(\mathcal{E}/F)& = \sum\limits_{\rho_i \in \Sigma(1), j \in \mathbb{Z}} j  \left( \rk_{\mathcal M/ F} (E'^{i}(j)) - \rk_{\mathcal M/ F} (E'^{i}(j+1))\right)  h_{i}\\
				& = \sum\limits_{\rho_i \in \Sigma(1), j \in \mathbb{Z}} j  \left( \rk (E^{i}(j) \vee F) - \rk (E^{i}(j+1) \vee F) \right)  h_{i}
			\end{split}
		\end{equation*}
		Note that the lattice of flats of
                \(\underline{\mathcal M}':=(\underline{\mathcal
                  M}|_{F'})/ F\) is isomorphic to the interval \([F, F']\) in
                \(\mathcal{L}(\underline{\mathcal M})\). Using
                \eqref{egqt1.1}, we have
		\begin{equation*}
			\deg(\mathcal E_{F'/F})=\sum\limits_{\rho_i \in \Sigma(1), j \in \mathbb{Z}} j  \left(  \rk((E^i(j) \meet F') \vee F) - \rk((E^i(j+1) \meet F') \vee F) \right)  h_{i}.
		\end{equation*}
	\end{remark}

	\begin{lemma}\label{Slem1.1}
Let \(\mathcal{E}=\left( \mathcal{M}, \mathcal{G}, \{E^i(j)\}\right)
\) be a tropical toric reflexive sheaf of rank \(r\). Let \(F \in
\mathcal{L}(\underline{\mathcal M})\) be a flat of rank \(r_1\) such
$(F,E^i(j))$ is a modular pair for all $1 \leq i \leq s$ and all $j
\in \mathbb Z$, so the following condition holds
\begin{equation}\label{Seq1.1}
	\rk(F) + \rk(E^i(j))=\rk(E^i(j) \vee F) + \rk(E^i(j) \meet F)
\end{equation}
for all \(1 \leq i \leq s\) and \( j \in \mathbb{Z}\). Let
\(r_2=r-r_1\) Then we have
\[\mu(\mathcal{E})= \frac{r_1}{r} \mu(\mathcal{E}_{F}) +  \frac{r_2}{r} \mu(\mathcal{E}/F).\]
	\end{lemma}
	
\begin{proof}
First we show that \(\deg(\mathcal{E})=\deg(\mathcal{E}_{F}) +
\deg(\mathcal{E}/F)\). Consider the condition \eqref{Seq1.1} for
\(j+1\),
		\begin{equation}\label{Seq1.2}
			\rk(F) + \rk(E^i(j+1))=\rk(E^i(j+1) \vee F) + \rk(E^i(j+1) \meet F).
		\end{equation}
		From the equations \eqref{Seq1.1} and \eqref{Seq1.2}, we get
		\begin{equation*}
			\begin{split}
				\rk (E^{i}(j)) - \rk (E^{i}(j+1)) & =\left( \rk (E^{i}(j) \vee F) - \rk (E^{i}(j+1) \vee F)\right) \\
				& + \left( \rk (E^{i}(j) \meet F) - \rk (E^{i}(j+1) \meet F )\right). 
			\end{split}
		\end{equation*}
		Thus, using Remark \ref{rmkforSlem1.1}, we have \(\deg(\mathcal{E})=\deg(\mathcal{E}_{F}) + \deg(\mathcal{E}/F)\). Now consider
		\begin{equation*}
			\begin{split}
				\mu(\mathcal{E})& = \frac{\deg(\mathcal{E})}{r}=\frac{\deg(\mathcal{E}_{F}) + \deg(\mathcal{E}/F)}{r}\\
				&= \frac{r_1}{r} \frac{\deg(\mathcal{E}_{F})}{r_1} + \frac{r_2}{r} \frac{\deg(\mathcal{E}/F)}{r_2}\\
				&= \frac{r_1}{r} \mu(\mathcal{E}_{F}) +  \frac{r_2}{r} \mu(\mathcal{E}/F).
			\end{split}
		\end{equation*}
	\end{proof}
	
	\noindent
	The following immediate corollary, called the see-saw property, is extremely crucial:
	\begin{corollary}\label{Scor1.1}
		Let \(\mathcal{E}=\left( \mathcal{M}, \mathcal{G}, \{E^i(j)\}\right) \) be a tropical toric reflexive sheaf. Let \(F \in \mathcal{L}(\underline{\mathcal M})\) be a flat satisfying condition \eqref{Seq1.1}. Then 
		\begin{equation*}
			\begin{split}
				& 	\mu(\mathcal{E}_F) < \mu(\mathcal{E})  \Longleftrightarrow \mu(\mathcal{E}) <\mu(\mathcal{E}/ F), \text{ and } \\
				& 	\mu(\mathcal{E}_F) > \mu(\mathcal{E})  \Longleftrightarrow \mu(\mathcal{E}) >\mu(\mathcal{E}/ F). \\
			\end{split}
		\end{equation*}
	\end{corollary}
	
	\begin{lemma}\label{Slem1.2}
Let $\mathcal M$ be a valuated matroid of rank \(r\) on the ground set
$\mathcal{G}$ with basis valuation function $\nu$. Let \(F_1 \in
\mathcal{L}(\underline{\mathcal M})\) be a flat of rank \(r_1\). Let
\(F_2 \in \mathcal{L}(\underline{\mathcal M})\) be another flat of
rank \(r_2\) with \(F_1 \subsetneq F_2\). Then we have \(\left( \mathcal
M/ F_1\right) / F_2= \mathcal M/ F_2\). Let \(\mathcal{E}=\left(
\mathcal{M}, \mathcal{G}, \{E^i(j)\}\right) \) be a tropical toric
reflexive sheaf, then we have \((\mathcal{E}/F_1)/F_2=\mathcal
E/F_2\). Moreover, if \(E^i(j)\) are modular flats we have
\(\mathcal{E}'_{F_2} =\mathcal{E}_{F_2/F_1},\) where
$\mathcal{E}'=\mathcal{E} / F_1$.
	\end{lemma}
	
	\begin{proof}
Recall from \S\ref{ss:valuatedmatroids} that the contraction
\(\mathcal M':=\mathcal M/ F_1\) is a valuated matroid of rank
\(r-r_1\) on the ground set \(\mathcal{G} \setminus F_1 \) with the
basis valuation function given as follows:
		\begin{equation*}
			\nu_{\mathcal{M}'} : \binom{\mathcal{G} \setminus F_1}{r-r_1} \longrightarrow \Rbar, ~ B \longmapsto \nu(B \cup A),
		\end{equation*}
where \(A \subseteq F_1\) is such that $|A| = \rk_{\mathcal{M}}(A)=
r_1$ and $B \subseteq \mathcal{G} \setminus F_1$.  Note that \(F_2 \in
\mathcal{L}(\underline{\mathcal M}/ F_1) = [F_1, \mathcal{G}]\) and
$\rk_{\mathcal{M}'}(F_2)=r_2-r_1$. Then the contraction \(\mathcal M'/
F_2\) is a valuated matroid on the ground set \( (\mathcal{G} \setminus
F_1 )\setminus F_2=\mathcal{G} \setminus F_2\) of rank
$r-r_1-(r_2-r_1)=r-r_2$. We can choose an independent subset $A'
\subseteq \mathcal G$ of rank $r_2$ in $\mathcal M$ satisfying $A
\subsetneq  A' \subseteq F_2$. Let $A'':= A' \setminus F_1 \subset F_2
\setminus F_1$. Then, we have \(
\rk_{\mathcal{M}'}(A'')=\rk(A')-r_1=r_2-r_1\). Also, $|A''|=|A'| -
|F_1| \leq r_2-r_1$. Hence, $A''$ must be independent. Then the basis
valuation function for \(\mathcal M'/ F_2\) is given as follows: for
any $B \subseteq \mathcal{G} \setminus F_2$
                \begin{equation*}
                	\begin{split}
                		\nu_{\mathcal{M}'/F_2}(B)= \nu_{\mathcal{M}'}(B \cup A'') = \nu (B \cup A'' \cup A) =\nu(B \cup A').
                	\end{split}
                \end{equation*}
             This shows that $\nu_{\mathcal{M}'/F_2}$ coincides with the basis valuation function for the matroid $\mathcal{M}/F_2$. Hence, we have \(\left( \mathcal M/ F_1\right) / F_2= \mathcal M/ F_2\) as valuated matroids. Recall that the quotient \(\mathcal E/F_2\) is given by \(\left( \mathcal{M}/F_2, \mathcal{G} \setminus F_2, \{E^i(j) \vee F_2\}\right) \). Note that $\mathcal{L}(\underline{\mathcal{M}}/F_2)$ is the interval \([F_2, \mathcal{G}]\) in $\mathcal{L}(\underline{\mathcal{M}})$ and we have 
		$$\rk_{\underline{\mathcal M}/ F_2}(E^i(j) \vee F_2)=\rk(E^i(j) \vee F_2) - \rk(F_2) .$$
		The tropical toric reflexive sheaf $(\mathcal{E}/F_1)/F_2$ is given by
		\[\left( \mathcal{M}/F_2, \mathcal{G} \setminus F_2, \{E^i(j) \vee F_2\}\right) ,\]
		as \(F_1 \subsetneq F_2\). This shows that the tropical toric reflexive sheaves \((\mathcal{E}/F_1)/F_2\) and \(\mathcal E/F_2\) are same.
		
		To see the last sentence of the lemma, note that the tropical toric reflexive sheaves \(\mathcal{E}_{F_2/F_1}\) and \(\mathcal{E}'_{F_2}\) are given by the data
		\begin{equation*}
			\begin{split}
				&	\left( \left( \mathcal{M} |_{F_2}\right)/ F_1 , F_2 \setminus F_1, \{(E^i(j) \meet F_2) \vee F_1\}\right)  \text{ and }\\ 
				&    \left( (\mathcal M/ F_1)|_{F_2}, F_2 \setminus F_1, \{(E^i(j) \vee F_1) \meet F_2 \}\right),
			\end{split}
		\end{equation*}
		respectively. Since \(E^i(j)\) are modular flats, we have 
		\begin{equation*}
			(E^i(j) \meet F_2) \vee F_1=(E^i(j) \vee F_1) \meet F_2. 
		\end{equation*}
		Hence, it suffices to check that both the matroids
                \(\left( \mathcal{M} |_{F_2}\right)/ F_1\) and
                \((\mathcal M/ F_1)|_{F_2}\) have the same basis
                valuation function.  Choose
		\begin{equation}\label{N7.22.3}
			A \subseteq \mathcal{G} \setminus F_2 \text{ such that } \rk(A \cup F_2)=r.
		\end{equation}
	  Then the basis valuation function for $\mathcal{M} |_{F_2}$ is given by
		\begin{equation*}
			\nu_{F_2} : \binom{F_2}{r_2} \longrightarrow \Rbar, B_2 \longmapsto \nu(B_2 \cup A),
		\end{equation*}
		for all $B_2 \subseteq F_2$. Note that $\mathcal{L}( \underline{\mathcal{M}} |_{F_2}) =[\emptyset, F_2]$ and $F_1 \in \mathcal{L}( \underline{\mathcal{M}} |_{F_2}) $ is a flat of rank $r_1$. Choose $A' \subseteq F_1$ independent with $\rk(A')=r_1$. Then the contraction $\left( \mathcal{M} |_{F_2}\right)/ F_1$ is a valuated matroid on the ground set $F_2 \setminus F_1$ with the basis valuation function given by
		\begin{equation}\label{N7.22.4}
			\nu'' : \binom{F_2 \setminus F_1}{r_2-r_1} \longrightarrow \Rbar, B \longmapsto \nu_{F_2}(B \cup A')=\nu(B \cup A' \cup A).
		\end{equation}
		The contraction \(\mathcal M':=\mathcal M/ F_1\) is a valuated matroid of rank \(r-r_1\) on the ground set \(\mathcal{G} \setminus F_1 \)  with the basis valuation function given as follows: 
		\begin{equation*}
			\nu_{\mathcal{M}'} : \binom{\mathcal{G} \setminus F_1}{r-r_1} \longrightarrow \Rbar, ~ B_1 \longmapsto \nu(B_1 \cup A'),
		\end{equation*}
		where \(B_1 \subseteq F_1\). Note that using \eqref{N7.22.3}, we have
		\begin{equation*}
			\begin{split}
				\rk_{\mathcal{M}'} (A \cup F_2)=\rk(A \cup F_2)-r_1=r-r_1= \rk(\mathcal{M}').
			\end{split}
		\end{equation*}
		Now the restriction $\mathcal{M}'|_{F_2}$ is a valuated matroid on the ground set $F_2 \setminus F_1$ with basis valuation function given by
		\begin{equation}\label{N7.22.5}
	\nu_{\mathcal{M}'|_{F_2}} : \binom{F_2 \setminus F_1}{r_2- r_1}  \longrightarrow \Rbar, \, \, \, \, \, \, \,	\nu_{\mathcal{M}'|_{F_2}}(B)  = \nu_{\mathcal{M}'}(B \cup A)=\nu(B \cup A \cup A'),
		\end{equation}
		where $B \subseteq F_2 \setminus F_1$. Thus, from \eqref{N7.22.4} and \eqref{N7.22.5} we get the valuated matroids \(\left( \mathcal{M} |_{F_2}\right)/ F_1\) and \((\mathcal M/ F_1)|_{F_2}\) are same.
	\end{proof}

	\begin{theorem}[Jordan-H\"older filtration] \label{t:JH}
		Let \(\mathcal{E}=\left( \mathcal{M}, \mathcal{G}, \{E^i(j)\}\right) \) be a semistable tropical toric vector bundle on tropical toric variety $\trop(X_{\Sigma})$. Assume that the \(E^i(j)\) are modular flats of  $\mathcal L(\underline{\mathcal M})$. Then, there is an increasing filtration by flats of the underlying matroid \(\underline{\mathcal M}\)
		\begin{equation*}
			\emptyset =F_0 \subsetneq F_1 \subsetneq F_2 \subsetneq \ldots \subsetneq F_k=\mathcal{G}
		\end{equation*}
		for which the quotients \(\text{gr}_i:=\mathcal{E}_{F_i/ F_{i-1}} \) are stable with slope \(\mu(\mathcal{E})\) for $1 \leq i \leq k$.
	\end{theorem}

	\begin{proof}
		If $\mathcal{E}$ is stable, we set \(F_1=\mathcal{G}\)
                and we are done. Suppose that $\mathcal{E}$ is not
                stable. Since $\mathcal E$ is semistable there is
                a flat \(G_1 \in \mathcal{L}(\underline{\mathcal M})\)
                with \(\mu(\mathcal{E})=\mu(\mathcal{E}_{G_1})\). If
                \(\mathcal{E}_{G_1}\) is not stable we can choose a
                flat \(G_2 \in \mathcal{L}(\underline{\mathcal M})\)
                with \(G_2 \subsetneq G_1\) and
                \(\mu(\mathcal{E}_{G_2})=\mu(\mathcal{E}_{G_1})\). Since
                the ranks of these flats are decreasing, this cannot
                continue forever, so we must have a nonempty flat \(F_1 \in
                \mathcal{L}(\underline{\mathcal M})\) with
                \(\mathcal{E}_{F_1}\) stable and
                \(\mu(\mathcal{E})=\mu(\mathcal{E}_{F_1})\). Now we
                show that the quotient \(\mathcal E/ F_1\) is
                semistable with slope $\mu$. That the slope of \(\mathcal E/ F_1\) is
                $\mu$ follows from Lemma \ref{Slem1.1}.
                  Suppose \(\mathcal E/ F_1\) is not
                semistable. Then using Corollary \ref{Scor1.1}, we
                have a flat \(G \in [F_1, \mathcal{G}]\) such that
                \(\mu(\mathcal E/ F_1) > \mu((\mathcal E/
                F_1)/G)\). Thus we have \(\mu(\mathcal{E}) >
                \mu(\mathcal E/G)\) using Lemma \ref{Slem1.2}.  Corollary \ref{Scor1.1} then contradicts that
                $\mathcal{E}$ is semistable, so we conclude that 
                \(\mathcal E/ F_1\) is semistable. We denote by
                \(\mathcal{E}'\) the quotient \(\mathcal E/ F_1\) and repeat the same
                procedure. Let \(F_2 \in
                \mathcal{L}(\underline{\mathcal M}/ F_1)\) be a flat
                such that \(\mathcal{E}'_{F_2} \) is stable with slope
                $\mu$. Then we have $\mathcal{E}'_{F_2} =\mathcal{E}_{F_2/F_1}$ by Lemma \ref{Slem1.2}. Hence, we get \(\emptyset \subsetneq F_1 \subsetneq F_2\) with \(\mathcal{E}_{F_1}\) and \(\mathcal{E}_{F_2/F_1}\) are both stable with same slope $\mu$. Since \(\mathcal E'= \mathcal{E} / F_1\) is semistable and \(\mathcal{E}'_{F_2}\) is stable, we have \(\mathcal{E}'/F_2\) is again semistable,  using Remark \ref{rmkmodflat}(2). Note that \(\mathcal{E}'/F_2=\mathcal{E}/F_2\) by Lemma \ref{Slem1.2}. Hence, we can repeat the procedure together with Remark \ref{rmkmodflat} to get the desired filtration.
	\end{proof}
	
	\begin{lemma}\label{HNFILTgrlem}
		Let \(\mathcal{E}=\left( \mathcal{M}, \mathcal{G}, \{E^i(j)\}\right) \) be a tropical toric reflexive sheaf on tropical toric variety $\trop(X_{\Sigma})$. Assume that  \(E^i(j)\) are modular flats of the lattice of flats of $\underline{\mathcal M}$. Let $F, G \in \mathcal{L}(\underline{\mathcal{M}})$ be flats such that $\mathcal{E}/F$ is semistable. Assume further that one of the flats \(F\) or \(G\) is also modular. Then, we have
		\begin{equation*}
			\mu(\mathcal{E}_{G/ G \meet F}) \leq \mu(\mathcal{E}/F).
		\end{equation*} 
	\end{lemma}
	
	\begin{proof}
		Recall that the tropical toric reflexive sheaf $\mathcal{E}_{G/ G \meet F}$	is given by 
		\begin{equation*}
			\left( (\mathcal{M}|_G)/ (G \meet F), G \setminus (G \meet F), \{(E^i(j) \meet G) \vee(G \meet F)\}\right). 
		\end{equation*}
		Note that using the modularity of the flat \(E^i(j)\) (see Proposition \ref{defmodflat}(2)), we have 
		\begin{equation}\label{HNFILTgrlemeq1}
			\begin{split}
				E'^i(j):=(G \meet F)  \vee (E^i(j) \meet G) &= ((G \meet F)  \vee E^i(j)) \meet G.
			\end{split}
		\end{equation}
		Let \(A=F \vee G \in [F, \mathcal{G}]=\mathcal L(\underline{\mathcal M}/F)\). The tropical reflexive subsheaf of $\mathcal{E}/F$ corresponding to \(A\) is given by
		\begin{equation*}
			\left( (\mathcal{M}/F)|_A, A \setminus F, \{(E^i(j) \vee F) \meet A \}\right). 
		\end{equation*}
		We have \(E'^i(j) \leq (E^i(j) \vee F) \meet A\) and under the additional assumption of modularity, we get
		\begin{equation*}
			\begin{split}
				\rk((\mathcal{M}|_G)/ (G \meet F)) & =\rk(G)- \rk(G \meet F)\\
				& =\rk(F \vee G)-\rk(F)=\rk((\mathcal{M}/F)|_A).
			\end{split}
		\end{equation*}
		Hence, using \cite[Lemma 2.5]{HNS} we have \(\mu(\mathcal{E}_{G/ G \meet F}) \leq \mu((\mathcal{E}/F)_A)\). Finally, by semistability of \(\mathcal{E}/F\) we get \(\mu((\mathcal{E}/F)_A) \leq \mu(\mathcal{E}/F)\). Hence, the lemma follows.
	\end{proof}
	
	\begin{theorem}[Harder-Narasimhan filtration] \label{t:HNfilt}
		Let \(\mathcal{E}=\left( \mathcal{M}, \mathcal{G}, \{E^i(j)\}\right) \) be a tropical toric reflexive sheaf on tropical toric variety $\trop(X_{\Sigma})$. Assume that the \(E^i(j)\) are modular flats of the lattice of flats of $\underline{\mathcal M}$. Then there is an increasing filtration by flats of the underlying matroid \(\underline{\mathcal M}\)
		\begin{equation*}
			\emptyset=F_0 \subsetneq F_1 \subsetneq F_2 \subsetneq \ldots \subsetneq F_k=\mathcal{G}
		\end{equation*}
		where the quotients \(\text{gr}_i:=\mathcal{E}_{F_i/ F_{i-1}}, \, i=1, \ldots, k\) satisfies the following conditions:
		\begin{itemize}
			\item[(i)] the quotient \(\text{gr}_i\) is semistable, and 
			\item[(ii)]  $\mu(\text{gr}_i) > \mu(\text{gr}_{i+1})$ for \(i=1, \ldots, k-1\).
		\end{itemize}
		This filtration is unique under the additional assumption that the flats \(F_1, \ldots, F_k\) are also modular.
	\end{theorem}
	
	\begin{proof}
If $\mathcal{E}$ is semistable, then we set \(k=1\) and
\(F_1=\mathcal{G}\). Now suppose that $\mathcal{E}$ is not
semistable. Since there are only finitely many flats in the lattice of
flats of \(\underline{\mathcal M}\), we may choose a flat \(F_1 \in
\mathcal{L}(\underline{\mathcal M})\) of maximal slope and maximal
rank amongst flats of maximal slope. Here by slope of a flat \(F_1\),
we mean slope of \(\mathcal{E}_{F_1}\). Then $\mathcal{E}_{F_1}$ is
semistable by the choice of the flat \(F_1\). Now consider the
quotient \(\mathcal{E}':=\mathcal E/ F_1\) on the ground set
\(\mathcal{G} \setminus F_1\). As before choose a flat \(F_2 \in
\mathcal{L}(\underline{\mathcal M}/ F_1)\) of maximum slope and
maximum rank amongst flats of maximal slope of \(\underline{\mathcal
  M}/ F_1\).
Note that, as before \(\mathcal E'_{F_2}\) is semistable by
the choice of \(F_2\). Using Lemma \ref{Slem1.2}, we get
$\mathcal{E}'_{F_2} =\mathcal{E}_{F_2/F_1}$. Next we show that
		\begin{equation*}
			\mu\left( \mathcal{E}_{F_1} \right) >  \mu\left( \mathcal{E}_{F_2/ F_1}\right).
		\end{equation*}
Since \(\rk(F_2) > \rk(F_1)\), by the choice of \(F_1\), we have
\(\mu(\mathcal{E}_{F_2}) < \mu(\mathcal{E}_{F_1})\). By Remark \ref{rmkmodflat}(1), $E^i(j) \join F$ is a modular flat of $\mathcal M/F$, so by Corollary \ref{Scor1.1} and 
we get
\(\mu(\mathcal{E}_{F_2}) > \mu(\mathcal{E}_{F_2/ F_1})\). Thus we get
$\mu\left( \mathcal{E}_{F_1} \right) > \mu\left( \mathcal{E}_{F_2/
  F_1}\right)$. To continue consider \(\mathcal{E}'/F_2\). Note that
\(\mathcal{E}'/F_2=\mathcal{E}/F_2\) by Lemma \ref{Slem1.2}. Hence, we
can repeat the procedure
to obtain the desired filtration.
		
		To see the uniqueness, note that if \((F_i)_{i=1, \ldots, k}\) is a filtration which satisfies the desired properties and if \(F'\) is another flat in $\mathcal{L}(\underline{\mathcal{M}})$ then this filtration induces a filtration \((F'_i=F_i \meet F')_{i=1, \ldots, k}\) of \(F'\). Let \(\text{gr}'_i=\mathcal{E}_{F'_i/ F'_{i-1}}\) be the corresponding quotients. If they are all non-zero, by taking $\mathcal{E}=\mathcal{E}_{F_i}$, $F=F_{i-1}$ and \(G=F' \meet F_i\) in Lemma \ref{HNFILTgrlem} together with Remark \ref{rmkmodflat}(1), we have 
		\begin{equation}\label{grvsgr'}
			\mu(\text{gr}'_i) \leq \mu(\text{gr}_i) \leq \mu(\mathcal{E}_{F_1})
		\end{equation}
		Here we have used the additional assumption that the flats \(F_i\) are also modular. Now applying Lemma \ref{Slem1.1} together with Remark \ref{rmkmodflat}(1), we get
		\begin{equation*}
			\begin{split}
				\mu(\mathcal{E}_{F_2'}) &= \frac{r_1}{r'} \, \mu(\mathcal{E}_{F'_{1}}) + \frac{r_2}{r'} \, \mu(\mathcal{E}_{F_2'/F'_{1}}) \\
				& \leq \frac{r_1}{r'} \, \mu(\mathcal{E}_{F_1}) + \frac{r_2}{r'} \, \mu(\mathcal{E}_{F_1}) ~ (\text{using \eqref{grvsgr'} and semistability of } \mathcal{E}_{F_1})\\
				& = \mu(\mathcal{E}_{F_1}),
			\end{split}
		\end{equation*}
		$\text{here } r_1=\rk(F'_{1}), r'=\rk(F_2') \text{ and } r_2=r'-r_1$. Inducting on $k$, we get
		\begin{equation*}
\mu(\mathcal{E}_{F'})=\mu(\mathcal{E}_{F_k'}) \leq \mu(\mathcal{E}_{F_1}).
		\end{equation*}
		 Moreover, if the equality holds, it forces all  \(\text{gr}'_i\) to be zero for all \(i \geq 2\), in other words \(F' \subseteq F_1\). Now if \((G_i)_{i=1, \ldots, k}\) is another filtration with these properties then the slopes of \(\mathcal{E}_{F_1}\) and \(\mathcal{E}_{G_1}\), we have \(F_1=G_1\). Now we continue applying the same argument to \(\mathcal{E}/F_1\). Note that we can apply the same argument in view of Remark \ref{rmkmodflat}. Hence, we conclude the filtration is unique.
	\end{proof}

	\begin{example}\label{Stab_eg1}
		\begin{enumerate}
		\item Consider the tropical tangent sheaf
                  \(\trop(\mathscr{T})\) on $\trop(X_{\Sigma})$ from
                  Example \ref{eg_tangent_bdl}, which is given by the
                  representable matroid $\underline{\mathcal{M}}$ on
                  the ground set
                  $\mathcal{G}=\{\mathbf{v}_1,\dots,\mathbf{v}_s\}$. Note
                  that by Proposition \ref{propmodflat}\eqref{item:propmodflat2}
                  \(\emptyset\), \(\mathcal G\) and atoms in a lattice
                  are always modular flats.

 Hence using Theorem \ref{t:HNfilt} a Harder-Narasimhan filtration of
                  \(\trop(\mathscr{T})\) corresponds to a particular increasing
                  sequence of flats of $\underline{\mathcal{M}}$.
                This can be thought of as a tropical analogue of
                \cite[Corollary 2.1.7]{pangT}).  which shows that a
                Harder-Narasimhan filtration of the tangent bundle on
                a toric variety corresponds to an increasing
                filtration of subspaces of $N_{\mathbb R}$.
			
			\item  Using Proposition \ref{propmodflat}(2), we get that any rank-$2$ tropical toric reflexive sheaf admits a  Harder-Narasimhan filtration as described in Theorem \ref{t:HNfilt}.
			
			\item \label{item:Stab_eg1.1} Let \(\mathcal{E}=\left( \mathcal{M}, \mathcal{G}, \{E^i(j)\}\right) \) be a tropical reflexive sheaf on tropical toric variety $\trop(X_{\Sigma})$. Suppose that there is a modular flat \(F\) of maximal slope such that the quotient is $\mathcal{E} / F$ is semistable and \(\mu (\mathcal{E}) < \mu(\mathcal{E}_F)\). Then \( \emptyset \subsetneq F \subsetneq \mathcal{G}\) is a Harder-Narasimhan filtration. To see this it suffices to check
			\begin{equation*}
				\mu(\mathcal{E}_F) > \mu (\mathcal{E} / F).
			\end{equation*}
			This follows from Corollary \ref{Scor1.1}.  Note that this case does not require the $E^i(j)$ to be modular flats.
		\end{enumerate}
		
	\end{example}

	\begin{example}
		Consider the representable matroid $\mathcal{M}$ on the ground set $\mathcal{G}=\{\mathbf{e}_1, \mathbf{e}_2, \mathbf{e}_1+\mathbf{e}_3, \mathbf{e}_3\}$ introduced in Example \ref{e:trickymatroid} \eqref{item:e:trickymatroid}. The lattice of flats is given in Figure \ref{eg7}.
	\begin{figure}
		\caption{\label{eg7}}
		\begin{tikzpicture}
			[scale=.5,auto=left]
			\node (n0) at (2,0) {$\emptyset$};
			\node (n1) at (-1,2)  {1};
			\node (n2) at (1,2)  {2};
			\node (n3) at (3,2) {3};
			\node (n4) at (5,2)  {4};
			\node (n6) at (-1,4) {12};
			\node (n7) at (1,4)  {23};
			\node (n8) at (3,4) {134};
			\node (n9) at (5,4)  {24};
			\node  (n12) at (2, 6) {$\mathcal G$};
			\foreach \from/ \to in {n4/n0,n3/n0,n2/n0,n1/n0,n6/n1,n6/n2,n7/n2,n7/n3,n8/n1, n8/n3, n8/n4,n9/n2,n9/n4, n12/n6, n12/n7, n12/n8, n12/n9}
			\draw (\from) -- (\to);
		\end{tikzpicture}
	\end{figure}
		Consider the tropical toric vector bundle $\mathcal{E}=\left(\mathcal{M}, \mathcal{G}, \{E^i(j)\} \right) $, where the filtrations of the flats are given as follows:
		
		$$E^{0}(j) = \left\{ \begin{array}
			{r@{\quad \quad}l}
			\mathcal{G} & i \leq 1 \\ 
			\{\mathbf{e}_1, \mathbf{e}_2\}& j=2 \\
			\emptyset & j > 2
		\end{array} \right.,
		E^{1}(j) = \left\{ \begin{array}
			{r@{\quad \quad}l}
			\mathcal{G} & i \leq 2 \\ 
			\{\mathbf{e}_1+\mathbf{e}_3\} & j=3 \\
			\emptyset & j > 3
		\end{array} \right.,
		E^{2}(j) = \left\{ \begin{array}
			{r@{\quad \quad}l}
			\mathcal{G} & i \leq 0 \\ 
			\{\mathbf{e}_3\} & j=1 \\
			\emptyset & j > 1
		\end{array} \right. .$$
		
		The slope is given by
		\begin{equation*}
			\mu(\mathcal{E})= \frac{(1+4)+(4+3)+1}{3}=\frac{13}{3}.
		\end{equation*}
		Note that \(F=\{\mathbf{e}_1,
                \mathbf{e}_1+\mathbf{e}_3, \mathbf{e}_3\}\) is the
                unique flat with maximum slope and maximum rank. Also,
                the flat \(F\) is modular and the quotient
                $\mathcal{E}/ F$ is rank 1 so stable. Hence, by
                Example \ref{Stab_eg1} \eqref{item:Stab_eg1.1} we get
                that \(\emptyset \subsetneq F \subsetneq \mathcal{G}\) is
                the Harder-Narasimhan filtration for $\mathcal{E}$.
	\end{example}
	
	\begin{example}
		Consider the tropical toric sheaf \(\left(\mathcal{E}, \mathcal{G}, \{E^i(j)\} \right) \) from Remark \ref{dist_latt_not_sat}. Note that here the \(E^i(j)\) are not all modular flats. The slope of the bundle is given by 
		\begin{equation*}
			\mu(\mathcal{E})=\frac{2+2}{4}=1.
		\end{equation*}
		Note that the rank-$1$ flats are singleton subsets of $\mathcal{G}$.	For the flats \(\{1\}, \{2\}, \{3\}\) and \(\{4\}\) the slope of the corresponding tropical toric  subsheaf is $1$. All other rank-$1$ flats have slope $0$.		
		The rank-$2$ flats are subsets of $\mathcal{G}$ with cardinality \(2\) as there are no circuits with cardinality 3. All of them have slope less than or equal to 1.
		
		The rank-$3$ flats are \(\{1,2,3,4\}, \{1,4,5,6\}, \{1,4,7,8\}, \{2,3,5,6\}\) and \(\{2,3,7,8\}\) and all subsets of cardinality 3 of $\mathcal{G}$ not contained in any of the previous 4-element flats. Consider the flat \(F=\{1,2,3,4\}\) and the corresponding filtrations are given by 
		$$F^0(j) = \begin{cases} F & j \leq 0 \\ \{1,2\} & j=1 \\ \emptyset & j >1 \\ \end{cases} \, \, \, \, \,,
		F^1(j) = \begin{cases} F & j \leq 0 \\ \{3,4\} & j=1 \\ \emptyset & j >1 \\ \end{cases}.$$
		The slope is given by
		\begin{equation*}
			\mu(\mathcal{E}_F)=\frac{2+2}{3}=1 + \frac{1}{3} > \mu(\mathcal{E}).
		\end{equation*}
		Hence, the bundle $\mathcal{E}$ is not
                semistable. Note that \(F\) is the unique flat with
                maximum slope and maximum rank. Note that \(F\) is not
                a modular flat. Consider the quotient \(\mathcal{E}
                /F=\left( \mathcal{M} / F, \{5,6,7,8\}, \{E^i(j) \vee
                F \}\right)\), where $$E^0(j) \vee F = \begin{cases}
                  \mathcal G & j \leq 0 \\ \{1,2, 3, 4\} & j >
                  0 \end{cases} \, \, \, \, \,, E^1(j) \vee
                F= \begin{cases} \mathcal G & j \leq 0 \\ \{1,2,3,4\}
                  & j >0 \\ \end{cases}.$$ Thus the slope
                \(\mu(\mathcal{E}/F)=0 < \mu(\mathcal{E}_F)\).  Since
                the quotient \(\mathcal{E} /F\) is rank $1$, it is
                stable. Hence, the Harder-Narasimhan filtration is
                given by \(\emptyset \subsetneq F \subsetneq \mathcal{G}\).
                This Harder-Narasimhan filtration is unique, even
                though $E^1(1)$ and $E^2(1)$ are not modular flats.
                Since there are only finitely many possibilities for a
                filtration, this can be checked directly.
	\end{example}

\section{Tropical vector bundles on other varieties} \label{s:nontoric}

In this section we extend the definition of tropical vector bundles from bundles on tropical toric varieties to more general bases.
A subscheme $X$ of a tropical toric variety $\trop(X_{\Sigma})$ is
given by a tropical ideal in the sense of \cite{TropicalIdeals}.
Given a tropical toric vector bundle on $\trop(X_{\Sigma})$ we can
restrict it to $X$.  Algebraically, this means replacing the Cox semiring $S$ in
\eqref{eqtn:tropicalM} by the quotient of $S$ by the bend locus of the
ideal for $X$.

We first recall the definition of a subscheme of a tropical toric variety.

\begin{definition} \cite{TropicalIdeals}*{\S 4}
  Let $\Sigma$ be a rational polyhedral fan in $N_{\mathbb R}$.  Let
  $\sigma$ be a cone of $\Sigma$.  An ideal $J \subseteq
  \Rbar[\sigma^{\vee} \cap M]$ is a {\em tropical ideal} if for every
  finite set $E$ of monomials in $\Rbar[\sigma^{\vee} \cap M]$ the set
  of polynomials with support in $E$ is the collection of vectors of a
  valuated matroid with ground set $E$.  A ideal $I \subseteq
  \Cox(\trop(X_{\Sigma})) = \Rbar[x_1,\dots,x_s]$ that is homogeneous
  with respect to the $A^1(\Sigma)$-grading is a {\em locally tropical
    ideal} if for every $\sigma \in \Sigma$ the ideal $(I
  \Cox(\trop(X_{\Sigma}))_{\mathbf{x}^{\hat{\sigma}}})_{\mathbf{0}}
    \subseteq \Rbar[\sigma^{\vee} \cap M]$ is a tropical ideal.

    A subscheme $X$ of a tropical toric variety $X_{\Sigma}$ is
    determined by a locally tropical ideal $I \subseteq
    \Cox(\trop(X_{\Sigma}))$ as follows.  On the local chart
    $\trop(U_{\sigma})$ the subscheme $X \cap \trop(U_{\sigma})$ has
    coordinate semiring $$\Rbar[\sigma^{\vee} \cap M]/\mathcal
    B({I\Cox(\trop(X_{\Sigma})_{\mathbf{x}^{\hat{\sigma}}}})_{\mathbf{0}}),$$
    where $\mathcal B()$ is the Giansiracusa bend congruence of \S
    \ref{ss:bendcongruence}.
\end{definition}

\begin{example}
Let $X \subseteq \trop(\mathbb P^4)$ be the tropicalization of the
embedding of $\mathbb P^1$ into $\mathbb P^4$ given by $\phi \colon
[x:y] \mapsto [x:y:x-y:x-2y:x-5y]$.  Then $X$ is given by the
tropical ideal $\trop(J)$ for $J = \langle u-x+y, v-x+2y,w-x+5y
\rangle \subseteq K[x,y,u,v,w]$.  On the chart $x \neq \infty$ the coordinate ring of $X$ is
$$\Rbar[y/x,z/x,u/x,v/x,w/x] / \mathcal R,$$ where $\mathcal R$ is the
bend congruence of the ideal $\trop(\langle u/x-1+y/x, v/x -1 +2y/x,
w/x -1 +5y/x \rangle)$.  This contains relations such as $$u/x \tplus
0 \tplus y/x \sim u/x \tplus 0 \sim u/x \tplus y/x \sim  0 \tplus
y/x.$$
\end{example}

We then define tropical vector bundles on subschemes of tropical toric
varieties to be the restriction of tropical vector bundles on the
ambient toric variety.

\begin{definition}
Let $X$ be a tropical subscheme of a tropical toric variety
$\trop(X_{\Sigma})$.  A tropical toric vector bundle on $X$ is given by the
data of a tropical toric vector bundle on the ambient variety $\trop(X_{\Sigma})$.
\end{definition}

\begin{example}
Let $Y = V(x^3+y^3+z^3) \subseteq \mathbb P^2$, and let $X = \trop(Y)
\subseteq \trop(\mathbb P^2)$.  Let $\mathcal E$ be the tropical
tangent bundle on $\mathbb P^2$.  The restriction of $\mathcal E$ to
$X$ is described by the $\mathbb Z$-graded semimodule
$$(\Rbar[x,y,z]/\mathcal X)^3/R,$$ where $\mathcal X$ is the bend
congruence $\{ x^3\tplus y^3 \tplus z^3 \sim x^3 \tplus y^3
, \dots \}$ of $X$, and $R$ is the congruence generated by $\{x \mathbf{e}_1
\tplus y \mathbf{e}_2 \sim x \mathbf{e}_1 \tplus z \mathbf{e}_3 \sim y
\mathbf{e}_2 \tplus z \mathbf{e}_3 \}$.
\end{example}

We can tropicalize vector bundles on projective varieties as follows.

\begin{definition}
  Let $\mathcal F$ be a vector bundle on a projective variety $Y$.  Suppose

  \begin{equation} \label{q:pullback}
   \begin{split}
     \text{there is an embedding }\phi \colon Y \rightarrow X_{\Sigma} \text{ of }Y \text{  into a toric variety }X_{\Sigma}, \\ \text{  and a toric vector bundle }   \mathcal E \text{ on }X_{\Sigma} \text{ for which } \mathcal F = \phi^*(\mathcal E).
\end{split}     
\end{equation}
Then a tropicalization of $\mathcal F$ is given by the tropicalization
of the subscheme $Y \subseteq X_{\Sigma}$, and the data of a
tropicalization of $\mathcal E$ in the sense of
Definition~\ref{pd:tropE}.
\end{definition}

Condition \ref{q:pullback} does not hold for all embeddings of $Y$ into a toric variety, as the following example shows.

\begin{example} \label{e:notTVB}
  Let $\mathcal F$ be the coherent sheaf on $\mathbb P^1$ given by $\mathcal F = \tilde{P}$ for the module
  $$P = K[x,y]^3/\langle (x-y)\mathbf{e}_1 + (x-2y)
  \mathbf{e}_2 + (x-5y)\mathbf{e}_3 \rangle.$$ over $K[x,y]$.
  Note that $\mathcal F$ is locally free, as $\mathbb P^1 = D(x-y)
  \cup D(x-2y) \cup D(x-5y)$ and on each of these open sets $\mathcal
  F$ is free.  For example,
  \begin{align*} \mathcal
  F(D(x-y)) & = (P_{x-y})_{\mathbf{0}} \\ &  = K[x/(x-y), y/(x-y)]^3/\langle
    \mathbf{e}_1 + (x-2y)/(x-y) \mathbf{e}_2 + (x-5y)/(x-y)
    \mathbf{e}_3 \rangle \\ & \cong K[x/(x-y),y/(x-y)]\mathbf{e}_2
    \oplus K[x/(x-y),y/(x-y)] \mathbf{e}_3.
    \end{align*}
Thus $\mathcal F$ defines a rank-$2$ vector bundle on $\mathbb P^1$.
However $\mathcal F$ is not presented as an equivariant vector bundle, so the
embedding $\phi \colon \mathbb P^1 \rightarrow \mathbb P^1$ does not
satisfy \ref{q:pullback}.
\end{example}  

\begin{question} \label{q:vbpullback}
  Given a vector bundle $\mathcal E$ on a projective variety $Y$, does
  there exist an embedding $i \colon Y \rightarrow X_{\Sigma}$ of $Y$
  into a toric variety $X_{\Sigma}$ with the property that $\mathcal E
  = i^* \mathcal F$ for a toric vector bundle $\mathcal F$ on
  $X_{\Sigma}$?
\end{question}

The vector bundles for which the answer to Question~\ref{q:vbpullback}
is yes are those that we can currently tropicalize.

It is important to note that Question~\ref{q:vbpullback} does not fix
the ambient toric variety.  Given one embedding of $Y$ into a toric variety, we may want 
to reembed $Y$ into a larger toric variety to achieve success.
We illustrate this with Example~\ref{e:notTVB}.

\begin{example}
Consider the example of Example~\ref{e:notTVB}.  We embed $\mathbb P^1$ into $\mathbb P^4$ by $\phi \colon [x:y] \mapsto
[x:y:x-y:x-2y:x-5y]$.  We then have $\mathbb P^1$ is the subscheme defined by
the ideal $J = \langle u-x+y, v-x+2y,w-x+5y \rangle \subseteq S = K[x,y,u,v,w]$.  We have
$$P \cong  (S/J)^3/ \langle u \mathbf{e}_1 + v\mathbf{e}_2 + w \mathbf{e}_3 \rangle.$$
Let $X_{\Sigma}$ be $\mathbb P^4$ with all torus invariant planes removed; this has fan the $1$-skeleton of the fan of $\mathbb P^4$, and contains $\phi(\mathbb P^1)$.
The lift
$$\tilde{P} = S^3/\langle u \mathbf{e}_1+v\mathbf{e}_2 + w \mathbf{e}_3 \rangle$$
defines a coherent sheaf $\mathcal F$ on $\mathbb P^4$, and thus on $X_{\Sigma}$.  Note that it is not a vector bundle on $\mathbb P^4$, but is on $X_{\Sigma}$, and $\mathcal E$ is the restriction of $\mathcal F$ to $\phi(\mathbb P^1) \cong \mathbb P^1$.  
\end{example}

\begin{remark} \label{r:allpullbacks}
  \begin{enumerate}
\item The answer to Question~\ref{q:vbpullback} is yes almost by
  definition when $\mathcal E$ is a very ample line bundle: in that
  case $\mathcal E = \phi^*(\mathcal O(1))$ for the embedding $\phi
  \colon Y \rightarrow \mathbb P^{h^0(\mathcal F)-1}$ given by $\mathcal E$.
  This extends easily to all line bundles.
\item \label{item:brion} However the answer to Question~\ref{q:vbpullback} is not
  uniformly positive, at least if we assume that the toric variety
  $X_{\Sigma}$ is smooth.  The Chow ring of a smooth toric variety is
  generated by the classes of divisors, so in particular all Chern
  classes of a toric vector bundle $\mathcal F$ are polynomials in
  classes of divisors on $X_{\Sigma}$.  Since Chern classes of vector
  bundles pullback, any vector bundle for which the answer to
  Question~\ref{q:vbpullback} is yes must have its Chern classes in
  the subring of the Chow ring of $Y$ generated by divisors.  As this
  is not the case for most tautological bundles of Grassmannians (see, for example, \cite{3264}*{Theorem 5.26}), we see that the answer to Question~\ref{q:vbpullback} is
  not always yes.  We thank Michel Brion for pointing this out to us.
\end{enumerate}
\end{remark}



\begin{bibdiv}
	
	\begin{biblist}


\bib{Allermann}{article}{
   author={Allermann, Lars},
   title={Chern classes of tropical vector bundles},
   journal={Ark. Mat.},
   volume={50},
   date={2012},
   number={2},
   pages={237--258},
   issn={0004-2080},
}
            
\bib{ArdilaKlivans}{article}{
   author={Ardila, Federico},
   author={Klivans, Caroline J.},
   title={The Bergman complex of a matroid and phylogenetic trees},
   journal={J. Combin. Theory Ser. B},
   volume={96},
   date={2006},
   number={1},
   pages={38--49},
   issn={0095-8956},
}	  
		\bib{Birkhoff}{book}{
			author={Birkhoff, Garrett},
			title={Lattice theory},
			series={American Mathematical Society Colloquium Publications},
			volume={Vol. XXV},
			edition={3},
			publisher={American Mathematical Society, Providence, RI},
			date={1967},
			pages={vi+418},
		}
		
		\bib{BDGP}{article}{
			author={Biswas, Indranil},
			author={Dey, Arijit},
			author={Genc, Ozhan},
			author={Poddar, Mainak},
			title={On stability of tangent bundle of toric varieties},
			journal={Proc. Indian Acad. Sci. Math. Sci.},
			volume={131},
			date={2021},
			number={2},
			pages={Paper No. 36, 21},
			issn={0253-4142},
		}
		
		\bib{BdeM08}{article}{
			author={Bonin, Joseph E.},
			author={de Mier, Anna},
			title={The lattice of cyclic flats of a matroid},
			journal={Ann. Comb.},
			volume={12},
			date={2008},
			number={2},
			pages={155--170},
			issn={0218-0006},
		}

\bib{CDPR}{article}{
   author={Cools, Filip},
   author={Draisma, Jan},
   author={Payne, Sam},
   author={Robeva, Elina},
   title={A tropical proof of the Brill-Noether theorem},
   journal={Adv. Math.},
   volume={230},
   date={2012},
   number={2},
   pages={759--776},
   issn={0001-8708},
}                


\bib{CoxHomogeneous}{article}{
   author={Cox, David A.},
   title={The homogeneous coordinate ring of a toric variety},
   journal={J. Algebraic Geom.},
   volume={4},
   date={1995},
   number={1},
   pages={17--50},
   issn={1056-3911},
}
                
		\bib{CLS}{book}{
			author={Cox, David A.},
			author={Little, John B.},
			author={Schenck, Henry K.},
			title={Toric varieties},
			series={Graduate Studies in Mathematics},
			volume={124},
			publisher={American Mathematical Society, Providence, RI},
			date={2011},
			pages={xxiv+841},
			isbn={978-0-8218-4819-7},
		}
		


		\bib{CrawleyDilworth}{book}{
			author={Crawley, Peter },
			author={Dilworth, Robert P. },
			title={Algebraic Theory of Lattices},
			series={the University of Michigan, },
			publisher={Prentice-Hall},
			date={1973},
			pages={201},
			isbn={9780130222695},
		}

		
                \bib{BergmanFanPolymatroid}{unpublished}{
label={CHLSW22}                  
      title={The Bergman fan of a polymatroid}, 
      author={Colin Crowley}
      author ={June Huh}
      author={Matt Larson}
      author={Connor Simpson}
      author={Botong Wang},
      year={2022},
      note={arXiv:2207.08764},
}		
		
		\bib{DDK-Stab}{article}{
			author={Dasgupta, Jyoti},
			author={Dey, Arijit},
			author={Khan, Bivas},
			title={Stability of equivariant vector bundles over toric varieties},
			journal={Doc. Math.},
			volume={25},
			date={2020},
			pages={1787--1833},
			issn={1431-0635},
			
		}
		
		\bib{ErSEVB}{article}{
			author={Dasgupta, Jyoti},
			author={Dey, Arijit},
			author={Khan, Bivas},
			title={Erratum for ``Stability of equivariant vector bundles over toric
				varieties''},
			journal={Doc. Math.},
			volume={26},
			date={2021},
			pages={1271--1274},
			issn={1431-0635},
		}

		\bib{Lucie}{unpublished}{
                        label={Dev22}
			author = {Lucie Devey},
			title = {A combinatorial description of stability for toric vector bundles},
			year = {2022},
		 note={arXiv:2212.11020},
		}  
		
		\bib{DJS}{article}{
			label={DJS},
			author={Di Rocco, Sandra},
			author={Jabbusch, Kelly},
			author={Smith, Gregory G.},
			title={Toric vector bundles and parliaments of polytopes},
			journal={Trans. Amer. Math. Soc.},
			volume={370},
			date={2018},
			number={11},
			pages={7715--7741},
		}
		
		\bib{DressWenzel}{article}{
			author={Dress, Andreas W. M.},
			author={Wenzel, Walter},
			title={Valuated matroids},
			journal={Adv. Math.},
			volume={93},
			date={1992},
			number={2},
			pages={214--250},
			issn={0001-8708},
		}
		

\bib{3264}{book}{
	author={Eisenbud, David},
	author={Harris, Joe},
	title={3264 and all that---a second course in algebraic geometry},
	publisher={Cambridge University Press, Cambridge},
	date={2016},
	pages={xiv+616},
	isbn={978-1-107-60272-4},
	isbn={978-1-107-01708-5},
	doi={10.1017/CBO9781139062046},
}

                
\bib{FarkasJensenPayne}{unpublished}{
      title={The Kodaira dimensions of $\overline{\mathcal{M}}_{22}$ and $\overline{\mathcal{M}}_{23}$}, 
      author={Farkas, Gavril},
      author={Jensen, David},
      author={Payne, Sam},
      year={2000},
      note={arXiv:2005.00622},
}

                \bib{Fulton}{book}{
			author={Fulton, William},
			title={Introduction to toric varieties},
			series={Annals of Mathematics Studies},
			volume={131},
			note={The William H. Roever Lectures in Geometry},
			publisher={Princeton University Press, Princeton, NJ},
			date={1993},
			pages={xii+157},
			isbn={0-691-00049-2},
		}
		
		
                
                \bib{GeorgeManon1}{article}{
label={GM23a},                  
   author={George, Courtney},
   author={Manon, Christopher},
   title={Cox rings of projectivized toric vector bundles and toric flag
   bundles},
   journal={J. Pure Appl. Algebra},
   volume={227},
   date={2023},
   number={11},
   pages={Paper No. 107437, 25},
   issn={0022-4049},
}

\bib{GeorgeManon2}{unpublished}{
label={GM23b}  
      title={Positivity properties of divisors on Toric Vector Bundles}, 
      author={Courtney George},
      author ={Christopher Manon},
      year={2023},
      note={arXiv:2308.09014},
                }

                \bib{GeorgeManon3}{unpublished}{
label={GM23c},                
      title={Algebra and Geometry of Irreducible toric vector bundles of rank $n$ on $\mathbb{P}^n$}, 
      author={George, Courtney},
      author ={Manon, Christopher},
      year={2023},
      note={arXiv:2308.09017},
}                
		

		
		\bib{Giansiracusa2}{article}{
			author={Giansiracusa, Jeffrey},
			author={Giansiracusa, Noah},
			title={Equations of tropical varieties},
			journal={Duke Math. J.},
			volume={165},
			date={2016},
			number={18},
			pages={3379--3433},
			issn={0012-7094},
		}

		\bib{gonzalez}{book}{
			author={Gonzalez, Jose L.},
			title={Toric projective bundles},
			note={Thesis (Ph.D.)--University of Michigan},
			publisher={ProQuest LLC, Ann Arbor, MI},
			date={2011},
			pages={83},
			isbn={978-1124-90560-0},
		}

		\bib{GUZ}{article}{
			author={Gross, Andreas},
			author={Ulirsch, Martin},
			author={Zakharov, Dmitry},
			title={Principal bundles on metric graphs: the ${\rm GL}_n$ case},
			journal={Adv. Math.},
			volume={411},
			date={2022},
			pages={Paper No. 108775, 45},
			issn={0001-8708},
		}


\bib{Helgason}{article}{
   author={Helgason, Thorkell},
   title={Aspects of the theory of hypermatroids},
   conference={
      title={Hypergraph Seminar (Proc. First Working Sem., Ohio State Univ.,
      Columbus, Ohio, 1972; dedicated to Arnold Ross)},
   },
   book={
      series={Lecture Notes in Math.},
      volume={Vol. 411},
      publisher={Springer, Berlin-New York},
   },
   date={1974},
   pages={191--213},
}                
                
		\bib{HNS}{article}{
			author={Hering, Milena},
			author={Nill, Benjamin},
			author={S\"{u}\ss , Hendrik},
			title={Stability of tangent bundles on smooth toric Picard-rank-2
				varieties and surfaces},
			conference={
				title={Facets of algebraic geometry. Vol. II},
			},
			book={
				series={London Math. Soc. Lecture Note Ser.},
				volume={473},
				publisher={Cambridge Univ. Press, Cambridge},
			},
			date={2022},
			pages={1--25},
		}
		
		\bib{HLSV}{article}{
			author={Husi\'{c}, Edin},
			author={Loho, Georg},
			author={Smith, Ben},
			author={V\'{e}gh, L\'{a}szl\'{o} A.},
			title={On complete classes of valuated matroids},
			conference={
				title={Proceedings of the 2022 Annual ACM-SIAM Symposium on Discrete
					Algorithms (SODA)},
			},
			book={
				publisher={[Society for Industrial and Applied Mathematics (SIAM)],
					Philadelphia, PA},
			},
			isbn={978-1-61197-707-3},
			date={2022},
			pages={945--962},
		}
		
		

\bib{JensenRanganathan}{article}{
   author={Jensen, David},
   author={Ranganathan, Dhruv},
   title={Brill-Noether theory for curves of a fixed gonality},
   journal={Forum Math. Pi},
   volume={9},
   date={2021},
   pages={Paper No. e1, 33},
}

		\bib{JMT19}{article}{
			author={Jun, Jaiung},
			author={Mincheva, Kalina},
			author={Tolliver, Jeffrey},
			title={Picard groups for tropical toric schemes},
			journal={Manuscripta Math.},
			volume={160},
			date={2019},
			number={3-4},
			pages={339--357},
			issn={0025-2611},
		}
		
		\bib{JMT20}{article}{
			author={Jun, Jaiung},
			author={Mincheva, Kalina},
			author={Tolliver, Jeffrey},
			title={Vector bundles on tropical schemes},
			journal={J. Algebra},
			volume={637},
			date={2024},
			pages={1--46},
			issn={0021-8693},
		}
		
		\bib{Kajiwara}{article}{
			author={Kajiwara, Takeshi},
			title={Tropical toric geometry},
			conference={
				title={Toric topology},
			},
			book={
				series={Contemp. Math.},
				volume={460},
				publisher={Amer. Math. Soc., Providence, RI},
			},
			isbn={978-0-8218-4486-1},
			date={2008},
			pages={197--207},
		}

		\bib{Katz}{article}{
			author={Katz, Eric},
			title={Matroid theory for algebraic geometers},
			conference={
				title={Nonarchimedean and tropical geometry},
			},
			book={
				series={Simons Symp.},
				publisher={Springer, [Cham]},
			},
			isbn={978-3-319-30944-6},
			isbn={978-3-319-30945-3},
			date={2016},
			pages={435--517},
		}
		
		
	        \bib{KavehManonOlder}{unpublished}{
                        label={KM22}
			title={Toric flat families, valuations, and applications to projectivized toric vector bundles}, 
			author={Kiumars Kaveh and Christopher Manon},
			year={2019},
			note={arXiv:1907.00543}
		}
		
		\bib{KMTVBvalTrop23}{unpublished}{
                  label={KM23}
				title={Toric vector bundles, valuations and tropical geometry}, 
		author={Kiumars Kaveh and Christopher Manon},
         year = {2023},
			note={arXiv:2304.11211}
		}

                \bib{KavehManon24}{unpublished}{
                  label={KM24}
				title={Toric Matroid Bundles}
		author={Kiumars Kaveh and Christopher Manon},
                year = {2024},
                note={Preprint}
}                

		\bib{KD}{article}{
			AUTHOR = {Khan, Bivas},
			AUTHOR ={ Dasgupta, Jyoti},
			TITLE = {Toric vector bundles on {B}ott tower},
			JOURNAL = {Bull. Sci. Math.},
			FJOURNAL = {Bulletin des Sciences Math\'{e}matiques},
			VOLUME = {155},
			YEAR = {2019},
			PAGES = {74--91},
			ISSN = {0007-4497},
			
			
		}

		
		\bib{Kly}{article}{
			AUTHOR = {Klyachko, Aleksander A.},
			TITLE = {Equivariant bundles over toric varieties},
			JOURNAL = {Izv. Akad. Nauk SSSR Ser. Mat.},
			FJOURNAL = {Izvestiya Akademii Nauk SSSR. Seriya Matematicheskaya},
			VOLUME = {53},
			YEAR = {1989},
			NUMBER = {5},
			PAGES = {1001--1039, 1135},
			ISSN = {0373-2436},
		}
		

                
		
		\bib{Kool}{article}{
			author={Kool, Martijn},
			title={Fixed point loci of moduli spaces of sheaves on toric varieties},
			journal={Adv. Math.},
			volume={227},
			date={2011},
			number={4},
			pages={1700--1755},
			issn={0001-8708},
		}

\bib{LasVergnas}{article}{
   author={Las Vergnas, Michel},
   title={On products of matroids},
   journal={Discrete Math.},
   volume={36},
   date={1981},
   number={1},
   pages={49--55},
   issn={0012-365X},
}

\bib{Lovasz}{article}{
   author={Lov\'{a}sz, L\'aszl\'o},
   title={Submodular functions and convexity},
   conference={
      title={Mathematical programming: the state of the art},
      address={Bonn},
      date={1982},
   },
   book={
      publisher={Springer, Berlin},
   },
   isbn={3-540-12082-3},
   date={1983},
   pages={235--257},
}
		
		\bib{TropicalBook}{book}{
			author={Maclagan, Diane},
			author={Sturmfels, Bernd},
			title={Introduction to tropical geometry},
			series={Graduate Studies in Mathematics},
			volume={161},
			publisher={American Mathematical Society, Providence, RI},
			date={2015},
			pages={xii+363},
			isbn={978-0-8218-5198-2},
			doi={10.1090/gsm/161},
		}

		\bib{TropicalIdeals}{article}{
			author={Maclagan, Diane},
			author={Rinc\'{o}n, Felipe},
			title={Tropical ideals},
			journal={Compos. Math.},
			volume={154},
			date={2018},
			number={3},
			pages={640--670},
			issn={0010-437X},
		}
		

                
\bib{Balancing}{article}{
   author={Maclagan, Diane},
   author={Rinc\'{o}n, Felipe},
   title={Varieties of tropical ideals are balanced},
   journal={Adv. Math.},
   volume={410},
   date={2022},
   pages={Paper No. 108713, 44},
   issn={0001-8708},
}                
                

\bib{MurotaTamura}{article}{
   author={Murota, Kazuo},
   author={Tamura, Akihisa},
   title={On circuit valuation of matroids},
   journal={Adv. in Appl. Math.},
   volume={26},
   date={2001},
   number={3},
   pages={192--225},
   issn={0196-8858},
   review={\MR{1818743}},
}      

\bib{Okonek}{book}{
	author={Okonek, Christian},
	author={Schneider, Michael},
	author={Spindler, Heinz},
	title={Vector bundles on complex projective spaces},
	series={Modern Birkh\"{a}user Classics},
	note={Corrected reprint of the 1988 edition;
		With an appendix by S. I. Gelfand},
	publisher={Birkh\"{a}user/Springer Basel AG, Basel},
	date={2011},
	pages={viii+239},
	isbn={978-3-0348-0150-8},
}          
		
		\bib{Oxley}{book}{
			author={Oxley, James G.},
			title={Matroid theory},
			series={Oxford Science Publications},
			publisher={The Clarendon Press, Oxford University Press, New York},
			date={1992},
			pages={xii+532},
			isbn={0-19-853563-5},
		}

		\bib{pangT}{thesis}{
			author = {Pang, Thiam-Sun},
			year = {2015},
			title = {The Harder-Narasimhan filtrations and rational contractions},
                        type = {Ph.D. Thesis},
                        school  = {Albert-Ludwigs-Universit\"at Freiburg im Breisgau}
		}
		
		\bib{Payne}{article}{
			AUTHOR = {Payne, Sam},
			TITLE = {Moduli of toric vector bundles},
			JOURNAL = {Compos. Math.},
			FJOURNAL = {Compositio Mathematica},
			VOLUME = {144},
			YEAR = {2008},
			NUMBER = {5},
			PAGES = {1199--1213},
			ISSN = {0010-437X},
		}
		
		\bib{PayneTropical}{article}{
			author={Payne, Sam},
			title={Analytification is the limit of all tropicalizations},
			journal={Math. Res. Lett.},
			volume={16},
			date={2009},
			number={3},
			pages={543--556},
			issn={1073-2780},
		}
		
		\bib{PerT}{thesis}{
			AUTHOR = {Perling, Markus},
			TITLE = {Resolutions and Moduli for Equivariant Sheaves over Toric Varieties},
			type = {Ph.D. thesis},
			school = {University of Kaiserslautern},
			YEAR = {2003},
		}

                \bib{Perling}{article}{
   author={Perling, Markus},
   title={Graded rings and equivariant sheaves on toric varieties},
   journal={Math. Nachr.},
   volume={263/264},
   date={2004},
   pages={181--197},
   issn={0025-584X},
}
		\bib{Sims}{article}{
			author={Sims, Julie A.},
			title={An extension of Dilworth's theorem},
			journal={J. London Math. Soc. (2)},
			volume={16},
			date={1977},
			number={3},
			pages={393--396},
			issn={0024-6107},
		}
		
		
		
		\bib{Stanley16}{article}{
			author={Stanley, Richard P.},
			title={Modular elements of geometric lattices},
			journal={Algebra Universalis},
			volume={1},
			date={1971/72},
			pages={214--217},
			issn={0002-5240},
			doi={10.1007/BF02944981},
		}
		
		\bib{StanleyEC1}{book}{
			AUTHOR = {Stanley, Richard P.},
			TITLE = {Enumerative combinatorics. {V}olume 1},
			SERIES = {Cambridge Studies in Advanced Mathematics},
			VOLUME = {49},
			EDITION = {Second},
			PUBLISHER = {Cambridge University Press, Cambridge},
			YEAR = {2012},
			PAGES = {xiv+626},
			ISBN = {978-1-107-60262-5},
		}
		
		
		\bib{White}{book}{
			author={White, Neil},
			title={Theory of Matroids (Encyclopedia of Mathematics and its Applications).},
			series={},
			publisher={Cambridge: Cambridge University Press.},
			date={1986},
			pages={},
		}


\bib{Ziegler91}{article}{
			author={Ziegler, G\"unter M.},
			title={Binary supersolvable matroids and modular constructions},
			journal={Proc. Amer. Math. Soc.},
			volume={113},
			date={1991},
			number={3},
			pages={817--829},
			issn={0002-9939},
		}
                
\bib{Ziegler}{book}{
  title={Combinatorial Models for Subspace Arrangements},
  author={Ziegler, G\"unter M.},
  url={https://books.google.de/books?id=aOAmygAACAAJ},
  year={1992},
}
                
		
	\end{biblist}
	
\end{bibdiv}
\end{document}